\documentclass{scrartcl}
\usepackage[numbers]{natbib}
\usepackage{fontenc}
\usepackage{shadethm}
\usepackage{amsthm}
\usepackage{float}
\usepackage{bm}
\usepackage{mathtools}
\usepackage{graphicx}
\usepackage{subfig}
\usepackage{stmaryrd}
\usepackage{mathrsfs} 
\usepackage{amsmath}
\usepackage{amssymb}
\usepackage{wasysym}
\usepackage{sectsty}
\usepackage[hyperfootnotes=false]{hyperref}
\usepackage[a4paper, left=2.7cm, right=2.7cm, top=2.5cm]{geometry}
\usepackage{chngcntr}
\counterwithin*{equation}{section}

\usepackage{cleveref}
\DeclareMathAlphabet{\mathpzc}{OT1}{pzc}{m}{it}

\sectionfont{\normalsize\centering}
\subsectionfont{\footnotesize\centering}

\theoremstyle{plain}
\newtheorem{thm}{Theorem}[section] 

\theoremstyle{definition}
\newtheorem{defn}[thm]{Definition} 
\newtheorem{lem}[thm]{Lemma}
\newtheorem{prop}[thm]{Proposition}
\newtheorem{rem}[thm]{Remark}
\newtheorem{cor}[thm]{Corollary}

\def\XXint#1#2#3{{\setbox0=\hbox{$#1{#2#3}{\int}$ }
		\vcenter{\hbox{$#2#3$ }}\kern-.6\wd0}}

\usepackage[utf8]{inputenc}
\usepackage[german,english,russian]{babel}

\newcounter{MPequ}

\newcounter{AppA}
\newenvironment{AppA}
{\stepcounter{AppA}%
	\addtocounter{equation}{0}%
	\equation}
{\endequation}

\begin{document}\selectlanguage{english}
\begin{center}
\normalsize \textbf{\textsf{Asymptotic windings, surface helicity and their applications in plasma physics}}
\end{center}
\begin{center}
	Wadim Gerner\footnote{\textit{E-mail address:} \href{mailto:wadim.gerner@inria.fr}{wadim.gerner@inria.fr}}
\end{center}
\begin{center}
{\footnotesize	Sorbonne Universit\'e, Inria, CNRS, Laboratoire Jacques-Louis Lions (LJLL), Paris, France}
\end{center}
{\small \textbf{Abstract:} 
In [J. Cantarella, J. Parsley, J. Geom. Phys. 60:1127 (2010)] Cantarella and Parsley introduced the notion of submanifold helicity. In the present paper we investigate properties of surface helicity and in particular answer two open questions posed in the aforementioned work: (i) We give a precise mathematically rigorous physical interpretation of surface helicity in terms of linking of distinct field lines. (ii) We prove that surface helicity is non-trivial if and only if the underlying surface has non-trivial topology (i.e. at least one hole).

We then focus on toroidal surfaces which are of relevance in plasma physics and express surface helicity in terms of average poloidal and toroidal windings of the individual field lines of the underlying vector field which enables us to provide a connection between surface helicity and rotational transform. Further, we show how some of our results may be utilised in the context of coil designs for plasma fusion confinement devices in order to obtain coil configurations of particular "simple" shape.

Lastly, we consider the problem of optimising surface helicity among toroidal surfaces of fixed area and show that toroidal surfaces admitting a symmetry constitute global minimisers.
\section{Introduction}
	In the realm of $3$-dimensional (magneto-)hydrodynamics there is a well-known preserved physical quantity called helicity. To set the stage, we let $\Omega\subset\mathbb{R}^3$ be a bounded smooth domain and we let $u$ be a smooth vector field which is divergence-free on $\Omega$ and tangent to its boundary. The $3$-d Biot-Savart helicity of $u$ may then be defined as follows
	\begin{gather}
		\nonumber
		\mathcal{H}(u):=\frac{1}{4\pi}\int_{\Omega\times \Omega}u(x)\cdot\left(u(y)\times \frac{x-y}{|x-y|^3}\right)d^3yd^3x.
	\end{gather}
	An alternative way to express helicity consists of defining the Biot-Savart potential of $u$ by \begin{gather}
		\nonumber
		\operatorname{BS}_{\Omega}(u)(x):=\frac{1}{4\pi}\int_{\Omega}u(y)\times \frac{x-y}{|x-y|^3}d^3y
	\end{gather}
	so that $\mathcal{H}(u)$ is just the $L^2$-inner product of $u$ and $\operatorname{BS}_{\Omega}(u)$, c.f. \cite{CDG01}. From a physical perspective helicity may be regarded as a measure of the average linking of distinct field lines of the underlying vector field $u$, see \cite{M69},\cite{Arnold2014},\cite{V03}. The relevance of helicity is that it is preserved under the action of volume preserving diffeomorphisms, c.f. \cite[Section 2.3 Corollary]{Arnold2014}, \cite[Theorem A]{CDGT002},\cite[Lemma 4.5]{G20} and in fact is in some sense the only such invariant \cite{EPT16}. In the context of the incompressible Euler equations the velocity field $v$ satisfies the equations
	\begin{gather}
		\nonumber
		\partial_tv+\nabla_vv=\nabla p\text{, }\operatorname{div}(v)=0\text{ in }\Omega\text{, }v\parallel\partial\Omega,
	\end{gather}
	where $p$ denotes the pressure function. Upon applying the curl operator to the first equation one verifies that the vorticity $\omega$ of the fluid, i.e. the curl of the velocity field $v$, is a solution to
	\begin{gather}
		\nonumber
		\partial_t\omega=-[v,\omega],
	\end{gather} 
	where $[\cdot,\cdot]$ denotes the Lie-bracket of vector fields. This implies that $\omega(t,x)=((\psi_t)_*\omega_0)(x)$, where $\psi_t$ denotes the flow of $v$, $(\psi_t)_*$ denotes the pushforward with respect to $\psi_t$ and $\omega_0(x):=\omega(0,x)=\operatorname{curl}(v(0,x))$ is the initial vorticity. Since $v\parallel \partial\Omega$ and $\operatorname{div}(v)=0$ we conclude that $\psi_t$ is well-defined and volume-preserving. We hence conclude from our preceding discussion that the helicity of the vorticity vector field is preserved in time. In the context of ideal magnetohydrodynamics, one can show in a similar spirit that the helicity of the underlying magnetic field is preserved in time. Even more, helicity also provides a lower bound on the $L^2$-norm of the underlying vector field, which in the context of ideal magnetohydrodynamics corresponds to the magnetic energy of the underlying magnetic field, c.f. \cite[Section III Theorem 1.5]{AK21}, \cite[Theorem B]{CDG00}
	\begin{gather}
		\nonumber
		|\mathcal{H}(u)|\leq c(\Omega)\int_{\Omega}|u(x)|^2d^3x
	\end{gather}
where $c(\Omega)>0$ is a constant independent of $u$.
It is further interesting to understand which other diffeomorphisms, beyond the volume-preserving diffeomorphisms, preserve helicity. This question has been investigated in \cite{CP10} and it turns out that in general there exist non-volume preserving diffeomorphisms which preserve helicity.

In \cite[Section 6.1]{CP10} Cantarella and Parsley introduced shortly the notion of submanifold helicity, a special case of which is the surface helicity, which corresponds to the $(0,2,3)$-helicity in the terminology of \cite{CP10}. As we have seen, the notion of $3$-d helicity is an important quantity in the context of fluid dynamics. It is therefore of interest to understand its $2$-d analogue since vector fields tangent to surfaces appear naturally in many applications, some of which we will discuss in the next sections.
\newline
\newline
The main goal of the present manuscript is to illuminate our understanding of surface helicity and establish connections to other known dynamical quantities.
\newline
\newline
\textbf{Outline of the paper:} In the next section we present the main results which we divided in subsections to increase readability. We first discuss answers to some open questions raised in \cite[Section 6.1]{CP10} regarding the surface helicity and provide in particular a general interpretation of surface helicity in terms of linkage of distinct field lines of the underlying vector field. In \Cref{SS22} we provide an alternative interpretation on toroidal surfaces which connects surface helicity to other important dynamical quantities on the torus. Based on this alternative interpretation we connect in \Cref{SS23} the notion of surface helicity to the notion of rotational transform in the context of plasma fusion. We then deal in \Cref{SS24} with the question of optimising surface helicity among toroidal surfaces of prescribed area and prove that toroidal surfaces admitting a symmetry are global minimisers while the existence of global maximisers is stated as an open problem. Finally, in \Cref{SS25} we discuss how some of the results obtained in the present manuscript can be used to construct "simple" surface currents on a coil winding surface which approximate desired target magnetic fields in a given plasma domain, which is of relevance in stellarator plasma fusion confinement devices. The proofs of the theorems are contained in \Cref{S3}.

\section{Main results}
Before we state the main results we establish some notation.
\newline
\newline
\textbf{Notation:} For a given, $C^{1,1}$-regular, closed, connected surface $\Sigma\subset \mathbb{R}^3$ we let $L^p(\Sigma)$, $1\leq p\leq \infty$, denote the standard space of $L^p$-integrable functions on $\Sigma$ and we define the vectorial counterpart by
\begin{gather}
	\nonumber
	L^p\mathcal{V}(\Sigma):=\left\{v\in\left(L^p(\Sigma)\right)^3\mid v(x)\cdot \mathcal{N}(x)=0\text{ a.e. }x\in \Sigma\right\},
\end{gather}
where $\mathcal{N}$ denotes the outward pointing unit normal on $\Sigma$, i.e. the unit normal pointing in the direction of the unbounded component of $\mathbb{R}^3\setminus \Sigma$. In addition we let $C^{0,1}\mathcal{V}(\Sigma)$ denote the Lipschitz-continuous vector fields on $\Sigma$ which are tangent to $\Sigma$ at every point. We say that $v\in L^{p}\mathcal{V}(\Sigma)$ is divergence-free if $\int_{\Sigma} v\cdot \nabla_{\Sigma}\alpha\text{ } d\sigma=0$ for all $\alpha\in C^1(\Sigma)$, where $\nabla_{\Sigma}$ denotes the surface gradient of $\alpha$ which can be computed as $\nabla_{\Sigma}\alpha=\nabla \tilde{\alpha}-(\mathcal{N}\cdot \nabla \tilde{\alpha})\mathcal{N}$, where $\tilde{\alpha}$ is an arbitrary $C^1$-extension of $\alpha$ to $\overline{\Omega}$ and $\nabla$ denotes the standard Euclidean gradient, and where $d\sigma$ denotes the standard surface measure. We denote by $L^p\mathcal{V}_0(\Sigma)$ the subspaces of $L^p\mathcal{V}(\Sigma)$ which consist of divergence-free vector fields. Lastly, we say that $v\in L^p\mathcal{V}(\Sigma)$ is curl-free if $\int_{\Sigma}v\cdot (\nabla_{\Sigma}\alpha\times \mathcal{N})d\sigma=0$ for all $\alpha\in C^1(\Sigma)$.
\subsection{Properties of Surface helicity}
We start by defining the $2$-dimensional analogue of the Biot-Savart operator.
\begin{defn}[Surface Biot-Savart operator]
	\label{D21}
	Let $\Sigma\subset \mathbb{R}^3$ be a closed, connected $C^{1,1}$-surface. Given $1<p<\infty$, the following operator is a well-defined, continuous operator called the \textit{surface Biot-Savart operator}
	\begin{gather}
		\nonumber
		\operatorname{BS}_{\Sigma}:L^p\mathcal{V}(\Sigma)\rightarrow L^q\mathcal{V}(\Sigma)\text{, }v\mapsto \frac{1}{4\pi}\left(\int_{\Sigma}v(y)\times \frac{x-y}{|x-y|^3}d\sigma(y)\right)^\parallel
	\end{gather}
	where $q:=\frac{2p}{2-p}$ if $1<p<2$, where $1\leq q<\infty$ can be taken arbitrary if $p=2$ and $q=\infty$ if $2<p\leq \infty$.
\end{defn}
In order to define the "tangent part" $\cdot^\parallel$ we cover $\Sigma$ by any collection of open domains $U_k$ which support local orthonormal frames $\{E^k_1,E^k_2\}$ of class $C^{0,1}$. On each such $U_k$ we define (we drop the superscript $k$)
\[
\left(\int_\Sigma v(y)\times \frac{x-y}{|x-y|^3}d\sigma(y)\right)^\parallel:=\sum_{i=1}^2\left(\int_{\Sigma}v(y)\times \frac{x-y}{|x-y|^3}d\sigma(y)\cdot E_i(x)\right)E_i(x).
\]
The well-definedness of \Cref{D21} is shown in \Cref{AS1}. In particular we show that the defined vector fields agree on the overlaps of the $U_k$ almost surely and that any two covers $U_k$ and $\tilde{U}_k$ induce the same vector field.
\newline
\newline
We come now to the definition of (cross-)helicity.
\begin{defn}[(Cross-)Helicity]
	\label{D22}
	Let $\Sigma\subset\mathbb{R}^3$ be a connected, closed, $C^{1,1}$-surface. Then we define the \textit{cross-helicity} as follows
	\begin{gather}
		\nonumber
		\mathcal{H}_c:L^{\frac{4}{3}}\mathcal{V}(\Sigma)\times L^{\frac{4}{3}}\mathcal{V}(\Sigma)\rightarrow\mathbb{R}\text{, }(v,w)\mapsto \int_{\Sigma}v(x)\cdot \operatorname{BS}_{\Sigma}(w)(x)d\sigma(x)
		\\
		\nonumber
		=\frac{1}{4\pi}\int_{\Sigma\times \Sigma}v(x)\cdot \left(w(y)\times \frac{x-y}{|x-y|^3}\right)d\sigma(y)d\sigma(x).
	\end{gather}
Further, we define the \textit{surface helicity} of a vector field by
\begin{gather}
	\nonumber
	\mathcal{H}:L^{\frac{4}{3}}\mathcal{V}(\Sigma)\rightarrow\mathbb{R}\text{, }v\mapsto \mathcal{H}_c(v,v).
\end{gather}
\end{defn}
\begin{rem}
	\label{R23}
	Due to the continuity of $\operatorname{BS}_{\Sigma}:L^{\frac{4}{3}}\mathcal{V}(\Sigma)\rightarrow L^4\mathcal{V}(\Sigma)$ we see that cross helicity (and therefore helicity) is continuous and that there is some constant $c(\Sigma)>0$ such that for all $v,w\in L^{\frac{4}{3}}\mathcal{V}(\Sigma)$ we have $|\mathcal{H}_c(v,w)|\leq c_0\|v\|_{L^{\frac{4}{3}}(\Sigma)}\|w\|_{L^{\frac{4}{3}}(\Sigma)}$.
\end{rem}
In \cite{CP10} Cantarella and Parsley introduced the notion of helicity of closed forms on (sub-)manifolds of Euclidean spaces. Of particular interest are the $(1,3,3)$-helicity and the $(0,2,3)$-helicity which correspond to helicities of $2$-forms on domains in $\mathbb{R}^3$ and of $1$-forms on embedded surfaces in $\mathbb{R}^3$ respectively. It is shown \cite[Theorem 2.15]{CP10} that the $(1,3,3)$-helicity coincides with the classical helicity of divergence-free fields, tangent to the boundary of a given domain $\Omega$ upon an identification of a given closed $2$-form $\alpha$ with a vector field $v$ by contraction of the corresponding volume form.

Regarding the $(0,2,3)$-helicity the authors in \cite{CP10} mention that it should be possible to regard it as a measure of the linkage of distinct field lines of a vector field associated with a given closed $1$-form $\alpha$ and they raise the question of whether or not the $(0,2,3)$-helicity may attain non-trivial values (i.e. whether there exists $v\in C^\infty\mathcal{V}_0(\Sigma)$ with $\mathcal{H}(v)\neq 0$). They conjecture that this should be the case whenever $\Sigma$ has a non-trivial homology.

To summarise, the following two central problems regarding the $(0,2,3)$ helicity were open \cite[Section 6.1]{CP10}

\begin{enumerate}
	\item \textit{Show that if $\Sigma$ has non-trivial homology, then the $(0,2,3)$-helicity is not identically zero}.
	\item \textit{Provide a mathematically rigorous physical interpretation of $(0,2,3)$-helicity.}
\end{enumerate}
In the upcoming result we denote by $\mathcal{H}_{(0,2,3)}$ the $(0,2,3)$-helicity on (closed) $1$-forms and by $\mathcal{H}$ the (vector field) helicity as introduced in \Cref{D22}. Given some $v\in C^\infty\mathcal{V}(\Sigma)$ we denote by $\iota_v\omega_\Sigma$ the contraction of the area form $\omega_\Sigma$ via the vector field $v$. In particular, if $v$ is divergence-free, the associated $1$-form is closed. Since the authors in \cite{CP10} deal with smooth domains and surfaces we formulate the following result in the smooth setting.
\begin{thm}
	\label{T24}
	Let $\Sigma\subset\mathbb{R}^3$ be a smooth, connected, closed surface and let $v\in C^\infty\mathcal{V}_0(\Sigma)$. Then
	\[
	\mathcal{H}_{(0,2,3)}(\iota_v\omega_\Sigma)=\mathcal{H}(v).
	\]
\end{thm}
We exploit this relationship to provide an answer to both of the open problems raised in \cite{CP10} regarding $(0,2,3)$-helicity.

The following provides an answer to the first open problem regarding surface helicity.
\begin{thm}[Helicity is non-trivial]
	\label{T25}
	Let $\Sigma\subset\mathbb{R}^3$ be a closed, connected $C^{1,1}$-surface. Then there exists some $v\in L^2\mathcal{V}_0(\Sigma)$ with $\mathcal{H}(v)\neq 0$ if and only if $g(\Sigma)\geq 1$, where $g(\Sigma)$ denotes the genus of $\Sigma$. Further, if $v\in L^2\mathcal{V}_0(\Sigma)$ is co-exact, i.e. it can be written as $v=\operatorname{grad}_\Sigma(f)\times \mathcal{N}$, for some function $f\in H^1(\Sigma)$ and where $\mathcal{N}$ denotes the outward unit normal, then $\mathcal{H}_c(v,w)=0$ for all $w\in L^2\mathcal{V}_0(\Sigma)$.
\end{thm}
Regarding the second question posed above we intend to follow the ideas of \cite{Arnold2014} and \cite{V03} regarding the physical interpretation of the standard $3$-d helicity (which corresponds to $(1,3,3)$-helicity in the notation of \cite{CP10}).

Note that in general $\Sigma$ might contain a large "flat region" (a copy of a large $2$-d disc) and then if the field lines of a given vector field $v\in C^{0,1}\mathcal{V}_0(\Sigma)$ are not necessarily periodic, we might attempt, in the same spirit as in \cite{V03}, to close the field lines by geodesics. However, at least as long as the endpoints of the considered field lines are contained in the flat region, the geodesics will be lines passing through the corresponding endpoints of the distinct field line segments. But then these artificially added line segments will intersect with "high probability" (since the directions of the lines passing through the pairs of endpoints will be "generically" not parallel). Thus, one should expect that there is a set of positive measure in $\Sigma\times \Sigma$ for which the "artificially" closed field lines will intersect if we close field line segments by means of geodesics. In addition, the artificially added parts may intersect the original field line segments, see \Cref{Geodesic}.

\begingroup\centering
\begin{figure}[H]
	\hspace{4.5cm}\includegraphics[width=0.5\textwidth, keepaspectratio]{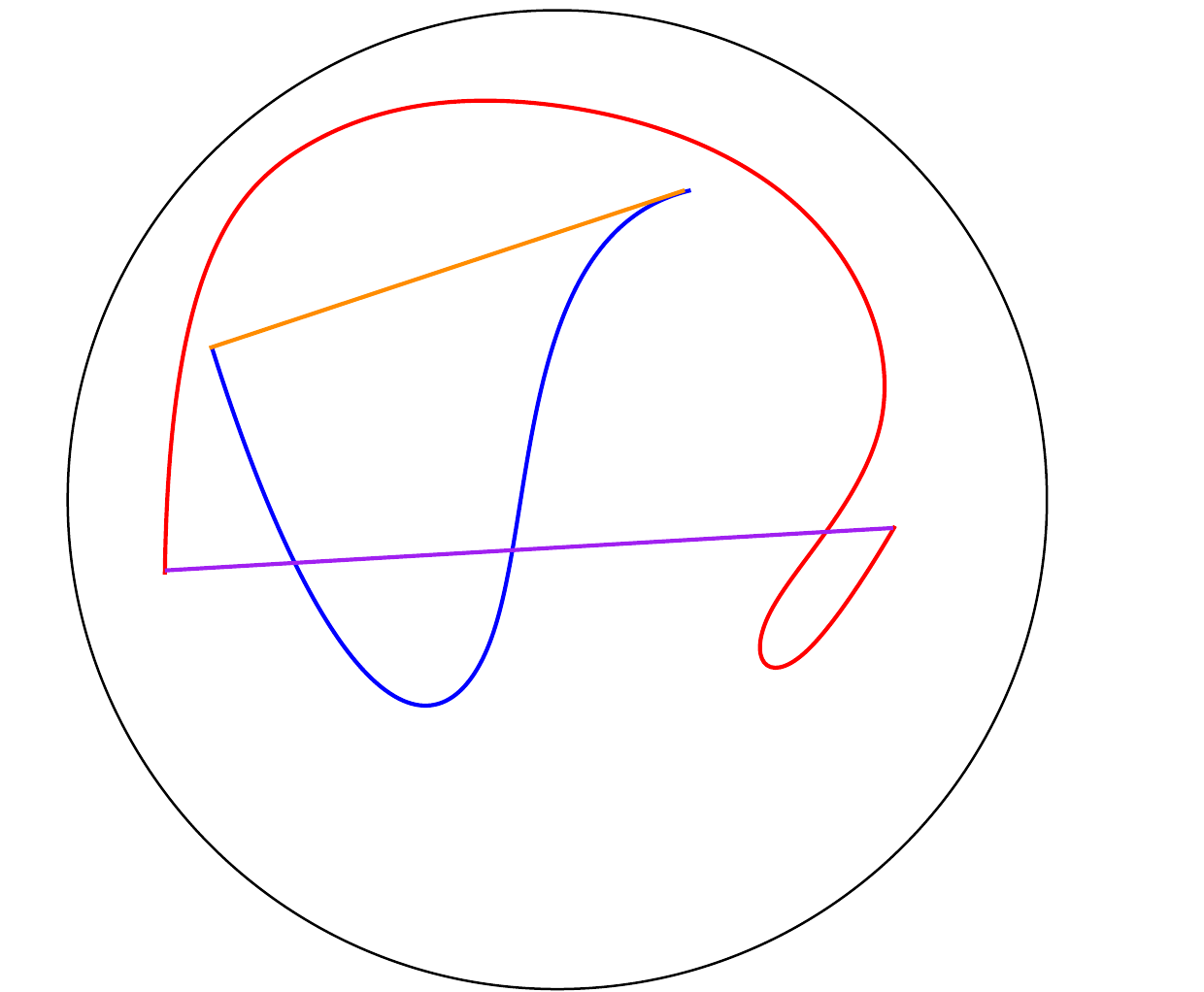}
	\caption{Two field lines segments (blue and red) artificially closed by geodesics (orange and purple) may lead to (self-)intersecting loops.} 
	\label{Geodesic}
\end{figure}
\endgroup
Therefore it is necessary to find a new way to appropriately close the field lines of the underlying vector field to be able to obtain an interpretation of helicity in terms of average linking numbers.

Given some $v\in C^{0,1}\mathcal{V}_0(\Sigma)$ and $x\in \Sigma$ we let $\gamma_x(t)$ denote the (unique) field line of $v$ starting at $x$. Now if we are given some $x\in \Sigma$ and $T>0$ it can happen that $\gamma_x(T)=x=\gamma_x(0)$ in which case the field line $\gamma_x$ is closed (possibly a point). In general however we will find $\gamma_x(T)\neq x=\gamma_x(0)$ and therefore it is necessary to close the field line in a suitable way to give an interpretation of helicity in terms of linking of suitable closed curves. It is well-known that, if $\Sigma\in C^2$, there exists some $\tau_0(\Sigma)>0$ such that the map $\Psi_{\tau}(x):=x+\tau\mathcal{N}(x)$, $x\in \Sigma$, is a diffeomorphism onto its image for every $|\tau|<\tau_0$ where $\mathcal{N}$ denotes the outward pointing unit normal, \cite[Chapter 7: Product Neighbourhood Theorem]{M65}. We define $\Sigma_{\tau}:=\Psi_{\tau}(\Sigma)$ so that in particular $\Sigma_0=\Sigma$. Now it is the case that for almost every $(x,y)\in \Sigma\times \Sigma$ the images of the corresponding field lines $\gamma_x(\mathbb{R})$ and $\gamma_y(\mathbb{R})$ are disjoint. Given $T,S>0$ we then consider the curves $\gamma_x:[0,T]\rightarrow \Sigma$ and $\gamma_y:[0,S]\rightarrow \Sigma$. In order to close them (in case they are not already closed) we proceed as follows: We connect $\gamma_x(T)$ with $\Psi_{\tau}(\gamma_x(T))$ for a suitably chosen $\tau$ (which we will be able to choose independent of $x,y,T$ and $S$) by following the path $\Psi_{t}(\gamma_x(T))$ from $t=0$ until $t=\tau$. Similarly, we can connect $\Psi_{\tau}(x)$ and $x=\gamma_x(0)$ by inverting the path $\Psi_{t}(x)$ from $t=0$ until $t=\tau$. We note that $\Psi_{\tau}(\gamma_x(T)),\Psi_{\tau}(x)\in \Sigma_{\tau}$ and that we are left with connecting $\Psi_{\tau}(\gamma_x(T))$ and $\Psi_{\tau}(x)$. We connect these two points by any curve contained in $\Sigma_{\tau}$ which is of unit speed and whose length can be uniformly bounded independently of $x$ and $T$. For instance, we can connect these two points by any (length minimising) geodesic on $\Sigma_{\tau}$. We close the curve $\gamma_y:[0,S]\rightarrow \Sigma$ in a similar fashion, but in contrast we connect $y$ and $\gamma_y(S)$ with their corresponding points on $\Sigma_{-\tau}$ and connect the corresponding two points on $\Sigma_{-\tau}$ by a geodesic on $\Sigma_{-\tau}$. Let us denote these two closed curves obtained in that way by $\sigma^{\tau}_{x,T}$ and $\sigma^{-\tau}_{y,S}$, see \Cref{ClosedCurve} for an illustration. By construction these two closed curves intersect if and only if $\gamma_x([0,T])$ and $\gamma_y([0,S])$ intersect, which we will show does not happen for almost every $(x,y)\in \Sigma\times \Sigma$.
Note also that we explicitly leave the surface $\Sigma$ to close our field lines.
By convention we set the linking number of a closed curve and any point not lying on this curve to zero.

\begingroup\centering
\begin{figure}[H]
	\hspace{4cm}\includegraphics[width=0.48\textwidth, keepaspectratio]{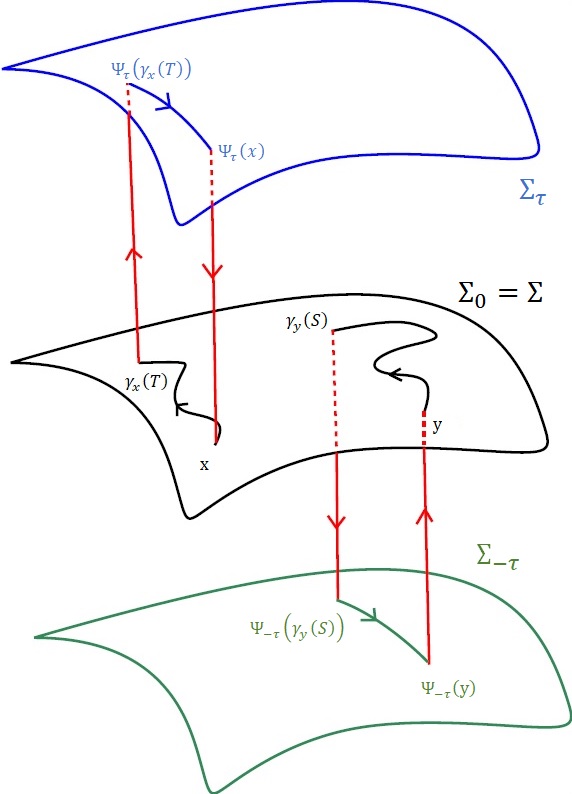}
	\caption{Illustration of the "closing process". The points $x,\gamma_x(T)$ are connected with their corresponding points on $\Sigma_{\tau}$ by straight lines which intersect $\Sigma$ orthogonally and the corresponding points on $\Sigma_{\tau}$ are connected by a geodesic. The arrows indicate the direction in which we walk along the closed curves.} 
	\label{ClosedCurve}
\end{figure}
\endgroup

With these preliminary observations we obtain the following result where we assume the surface $\Sigma$ to be of class $C^2$ to avoid technicalities.
\begin{thm}[Physical interpretation of surface helicity]
	\label{T26}
	Let $\Sigma\subset\mathbb{R}^3$ be a closed, connected $C^{2}$-surface. Then there exists some $\tau(\Sigma)>0$ such that for every $v\in C^{0,1}\mathcal{V}_0(\Sigma)$ and every sequences $(T_n)_n,(S_n)_n\subset (0,\infty)$ which are strictly increasing and diverging to $\infty$, the functions 
	\begin{gather}
		\nonumber
		\Sigma\times \Sigma\rightarrow \mathbb{R}\text{, }(x,y)\mapsto \frac{\operatorname{lk}\left(\sigma^{\tau}_{x,T_n}\text{,}\sigma^{-\tau}_{y,S_n}\right)}{T_nS_n}
	\end{gather}
	are well-defined elements of $L^1(\Sigma\times \Sigma)$ for every $n$ and converge to a limit function in $L^1(\Sigma\times \Sigma)$. Further, we have the identities
	\begin{gather}
		\nonumber
		\mathcal{H}(v)=\lim_{n\rightarrow\infty}\int_{\Sigma\times \Sigma}\frac{\operatorname{lk}\left(\sigma^{\tau}_{x,T_n}\text{,}\sigma^{-\tau}_{y,S_n}\right)}{T_nS_n}d\sigma(x)d\sigma(y)=\int_{\Sigma\times \Sigma}\lim^{L^1}_{n\rightarrow\infty}\frac{\operatorname{lk}\left(\sigma^{\tau}_{x,T_n}\text{,}\sigma^{-\tau}_{y,S_n}\right)}{T_nS_n}d\sigma(x)d\sigma(y).
	\end{gather}
\end{thm}
Here $\lim^{L^1}$ indicates that we consider the $L^1$-limit (and not the pointwise limit which exists only upon passing to appropriate subsequences). Therefore surface helicity may be regarded as either the "\textit{asymptotic average linking number of distinct field lines of $v$}" or, as is done more commonly in the $3$-$d$ case, the "\textit{average asymptotic linking number of distinct field lines of $v$}".
\begin{rem}
	\label{R27}
	If all field lines of $v$ are periodic and there exists some $\tau>0$ such that all field lines of $v$ have a period of at most $\tau$, then we can instead close the field line segments of $v$ by continuing to follow the flow of $v$ until we return to the respective starting points. In this periodic situation no additional technical issues arise by assuming that $\Sigma\in C^{1,1}$.
\end{rem}
\subsection{Surface helicity on toroidal surfaces}
\label{SS22}
We have seen in \Cref{T26} that on an arbitrary (closed) surface $\Sigma$, we can interpret helicity as a measure of entanglement of distinct field lines of the underlying vector field. In the $3$-dimensional situation entanglement of distinct field lines does not allow to infer information about the dynamical properties of individual field lines. However, in the $2$-dimensional situation closed field lines are co-dimension $1$ submanifolds and therefore the dynamics of individual field lines may very well influence the ability of the remaining field lines to link. To illustrate this point further, we imagine we are given a toroidal surface $\Sigma\cong T^2$ in $\mathbb{R}^3$ and a (regular) vector field $v$ tangent to $\Sigma$. We suppose that $v$ has a closed field line $\gamma$ which wraps once in poloidal direction around $\Sigma$. Since distinct field lines do not cross we infer that the remaining field lines of $v$ must be contained within $\Sigma\setminus \gamma$. But we can continuously "unfold" $\Sigma\setminus\gamma$ and we see that in fact the remaining field lines may be viewed as being contained in a cylinder so that no linking may take place, c.f. \Cref{PoloidalCut}. Consequently, the helicity of $v$ must be zero by means of the linking interpretation of helicity, \Cref{T26}. We therefore see that the dynamics of individual field lines may have a strong influence on surface helicity and one may thus expect that an alternative interpretation of surface helicity in terms of averages of dynamical properties of individual field lines should be possible.

\begingroup\centering
\begin{figure}[H]
	\hspace{4cm}\includegraphics[width=0.45\textwidth, keepaspectratio]{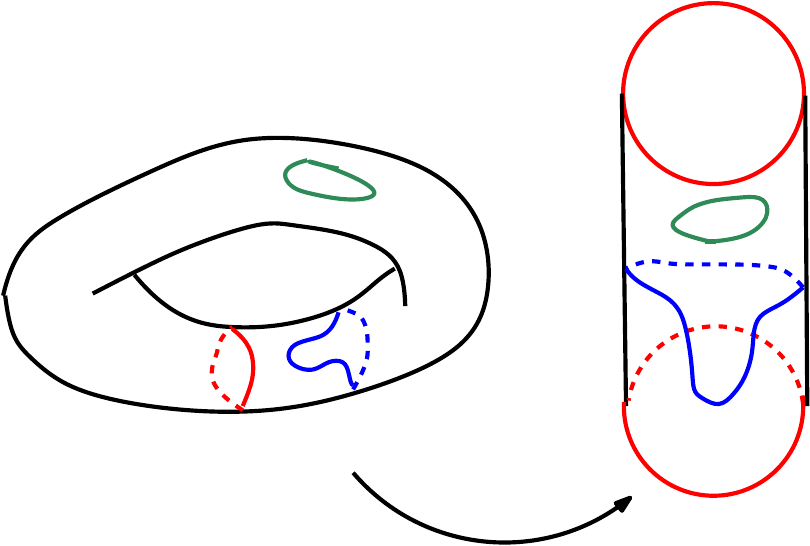}
	\caption{The presence of a poloidal field line $\gamma$ (red) prevents linking of any other field lines (here blue and green).} 
	\label{PoloidalCut}
\end{figure}
\endgroup

We focus here on $C^{1,1}$-surfaces $\Sigma\subset\mathbb{R}^3$ which bound solid tori, i.e. $\Sigma=\partial\Omega$ for a domain $\Omega\subset\mathbb{R}^3$ which is diffeomorphic to the solid torus $D^2\times S^1$, where $D^2$ is the unit disc in $\mathbb{R}^2$ and $S^1$ is its boundary. We note that these type of surfaces are the most relevant in plasma physics, as they appear as the boundaries of plasma domains or model the surfaces along which the coils wind which generate the plasma confining magnetic field \cite{M87}.

Before we state the main result of this subsection let us introduce some notions.
\begin{defn}[Poloidal \& Toroidal curves]
	\label{D28}
	Let $T^2\cong \Sigma\subset\mathbb{R}^3$ bound a $C^{1,1}$-solid torus $\Omega\subset\mathbb{R}^3$, i.e. $\Omega\cong D^2\times S^1$ and $\partial\Omega=\Sigma\in C^{1,1}$. We say that a simple closed $C^1$-curve $\sigma_p$ on $\Sigma$ is \textit{purely poloidal} if $\sigma_p$ is contractible when viewed as curve within $\overline{\Omega}$, but represents a non-trivial element of the first fundamental group of $\Sigma$. We call a simple closed $C^1$-curve $\sigma_t$ in $\Sigma$ \textit{toroidal} if it generates the first fundamental group of $\overline{\Omega}$ once viewed as a curve in $\overline{\Omega}$.
\end{defn}
We note first that, up to orientation and a standard identification, any two purely poloidal curves are path-homotopic within $\Sigma$. Similarly, up to orientation, all toroidal curves are homotopic within $\overline{\Omega}$. However, two toroidal curves may not be (even up to orientation) homotopic within $\Sigma$. In addition, any choice of $\sigma_p$ and $\sigma_t$ forms a pair of generators of the first fundamental group of $\Sigma$.

To motivate the next definition we consider an arbitrary simple closed $C^1$-curve $\sigma$ contained in $\Sigma$ and for a fixed pair $\sigma_p$ and $\sigma_t$ of a purely poloidal and a toroidal curve we can express $\sigma=P\sigma_p\oplus Q\sigma_t$ where $\oplus$ denotes the concatenation of curves and $P,Q\in \mathbb{Z}$ denote the amount of times the curves $\sigma_p$ and $\sigma_t$ are traversed. We observe first that there is a unique pair $\gamma_p,\gamma_t\in \mathcal{H}(\Sigma):=\{v\in W^{1,2}\mathcal{V}(\Sigma)\mid \operatorname{curl}_{\Sigma}(v)=0=\operatorname{div}_{\Sigma}(v)\}$ which forms a basis of $\mathcal{H}(\Sigma)$ and satisfies
\begin{gather}
	\label{E1}
	\int_{\sigma_p}\gamma_t=0=\int_{\sigma_t}\gamma_p\text{, }\int_{\sigma_p}\gamma_p=1=\int_{\sigma_t}\gamma_t.
\end{gather}
We observe further, which for example follows from an identical reasoning as in \cite[Lemma A.1]{G24}, that in fact $\mathcal{H}(\Sigma)\subset \bigcap_{1<p\leq\infty}W^{1,p}\mathcal{V}(\Sigma)\subset \bigcap_{0<\alpha<1}C^{0,\alpha}\mathcal{V}(\Sigma)$ and therefore the line integrals in (\ref{E1}) make sense in a classical sense. Now since $\sigma$ is homotopic to $P\sigma_p\oplus Q\sigma_t$, we find
\begin{gather}
	\nonumber
	\int_{\sigma}\gamma_t=P\int_{\sigma_p}\gamma_t+Q\int_{\sigma_t}\gamma_t=Q
\end{gather}
and further
\begin{gather}
	\label{E2}
	\lim_{T\rightarrow\infty}\frac{1}{T}\int_{\sigma[0,T]}\gamma_t=\frac{Q}{\tau},
\end{gather}
where $\tau>0$ is the period of $\sigma$ and $\sigma[0,T]$ is the path $[0,T]\rightarrow\Sigma$, $t\mapsto\sigma(t)$ (which is not simple for $T\geq \tau$). Similarly we find $P=\int_{\sigma}\gamma_p$ and $\lim_{T\rightarrow\infty}\frac{1}{T}\int_{\sigma[0,T]}\gamma_p=\frac{P}{\tau}$. This motivates the following definition.
\begin{defn}[Weighted asymptotic poloidal \& toroidal windings]
	\label{D29}
	Let $T^2\cong\Sigma\subset\mathbb{R}^3$ bound a $C^{1,1}$-solid torus $\Omega\subset\mathbb{R}^3$. Given a toroidal curve $\sigma_t$ and a purely poloidal curve $\sigma_p$ we define the \textit{weighted asymptotic toroidal windings} resp. \textit{weighted asymptotic poloidal windings} of a (not necessarily periodic) $C^1$-curve $\sigma:[0,\infty)\rightarrow\Sigma$ by
	\begin{gather}
		\nonumber
		\hat{q}(\sigma):=\lim_{T\rightarrow\infty}\frac{1}{T}\int_{\sigma[0,T]}\gamma_t\text{, }\hat{p}(\sigma):=\lim_{T\rightarrow\infty}\frac{1}{T}\int_{\sigma[0,T]}\gamma_p.
	\end{gather}
\end{defn}
\begin{rem}
	\label{R210}
	We note that $\gamma_t$ in (\ref{E1}) is independent of the chosen purely poloidal curve $\sigma_p$ and it only depends on the toroidal curve $\sigma_t$ via its orientation, i.e. if two toroidal curves are oriented such that they are homotopic within $\overline{\Omega}$ then the corresponding $\gamma_t$ will be the same and if they are oriented in opposite direction, then the corresponding $\gamma_t$ will differ by a minus sign. Consequently, $\hat{q}(\sigma)$ also depends only on the chosen orientation of the toroidal curve. However, $\gamma_p$ will in general depend on the homotopy type of $\sigma_t$ when viewed as a curve in $\Sigma$ so that even if two toroidal curves are homotopic within $\overline{\Omega}$, the corresponding $\gamma_p$ may be distinct (but will be the same if the toroidal curves are homotopic within $\Sigma$).
		\end{rem}
The following is the natural generalisation to vector fields.
\begin{defn}[(Average) Asymptotic windings]
	\label{D211}
	Let $T^2\cong\Sigma\subset\mathbb{R}^3$ bound a $C^{1,1}$-solid torus $\Omega\subset\mathbb{R}^3$. Given a toroidal curve $\sigma_t$, a purely poloidal curve $\sigma_p$ and a vector field $v\in C^{0,1}\mathcal{V}_0(\Sigma)$, we define the \textit{weighted asymptotic toroidal windings} and the \textit{weighted asymptotic poloidal windings} of $v$ at a point $x\in \Sigma$ as follows
	\begin{gather}
		\nonumber
		\hat{q}(x):=\hat{q}(\gamma_x)\text{, }\hat{p}(x):=\hat{p}(\gamma_x),
	\end{gather}
	where $\gamma_x$ denotes the field line of $v$ starting at $x$. We further define the \textit{average weighted asymptotic toroidal windings} and the \textit{average weighted asymptotic poloidal windings} of $v$ by
	\begin{gather}
		\nonumber
		\overline{Q}(v):=\frac{\int_{\Sigma}\hat{q}(x)d\sigma(x)}{|\Sigma|}\text{, }\overline{P}(v):=\frac{\int_{\Sigma}\hat{p}(x)d\sigma(x)}{|\Sigma|},
	\end{gather}
	where $|\Sigma|$ denotes the area of $\Sigma$.
\end{defn}
We note that in order to compute the average one should have $\hat{q},\hat{p}\in L^1(\Sigma)$. This however follows from standard ergodic theoretical arguments. We state it as a separate lemma since it is of importance when establishing a connection between surface helicity and toroidal and poloidal averages.
\begin{lem}
	\label{L212}
	Let $T^2\cong\Sigma\subset\mathbb{R}^3$ bound a $C^{1,1}$-solid torus $\Omega\subset\mathbb{R}^3$. Let $\sigma_p$ and $\sigma_t$ be a given purely poloidal and a toroidal curve respectively. If $v\in C^{0,1}\mathcal{V}_0(\Sigma)$, then $\hat{q},\hat{p}\in L^1(\Sigma)$ and we have
	\begin{gather}
		\nonumber
		\overline{Q}(v)=\frac{\int_{\Sigma}v(x)\cdot \gamma_t(x)d\sigma(x)}{|\Sigma|}\text{, }\overline{P}(v)=\frac{\int_{\Sigma}v(x)\cdot \gamma_p(x)d\sigma(x)}{|\Sigma|}.
	\end{gather}
\end{lem}
\begin{rem}
	One natural way to define \textit{purely toroidal} simple closed curves is to demand that a given toroidal loop is trivial within $\mathbb{R}^3\setminus \Omega$, see also \cite[Theorem B.2]{CP10}. Another useful way to choose a toroidal curve $\sigma_t$ can consist in making sure that the resulting induced element $\gamma_p$ is $L^2(\Sigma)$-orthogonal to $\gamma_t$ so that the quantities $\overline{Q}(v)$ and $\overline{P}(v)$ essentially correspond to the $L^2(\Sigma)$-orthogonal projection of $v$ onto the space $\mathcal{H}(\Sigma)$ in the sense that $\pi(v)=|\Sigma|\left(\overline{Q}(v)\frac{\gamma_t}{\|\gamma_t\|^2_{L^2(\Sigma)}}+\overline{P}(v)\frac{\gamma_p}{\|\gamma_p\|^2_{L^2(\Sigma)}}\right)$ where $\pi$ denotes the projection onto $\mathcal{H}(\Sigma)$ and $v\in C^{0,1}\mathcal{V}_0(\Sigma)$. Both choices can have advantages in distinct situations and therefore we allow general toroidal curves in our considerations.
\end{rem}
To establish a connection between the quantities $\overline{Q}$ and $\overline{P}$, which make statements about the average behaviour of individual field lines, and helicity which is a measure of the average linking of distinct field lines, we need to introduce one final notion.
We define the space $\mathcal{H}_N(\Omega)$ of \textit{harmonic Neumann fields} on a bounded $C^{1,1}$-domain $\Omega\subset\mathbb{R}^3$ as follows
\begin{gather}
	\nonumber
	\mathcal{H}_N(\Omega):=\{\Gamma\in H^1(\Omega,\mathbb{R}^3)\mid \operatorname{curl}(\Gamma)=0=\operatorname{div}(\Gamma)\text{, }\Gamma\parallel \partial\Omega\}.
\end{gather}
We note first, c.f. \cite[Lemma A.1]{G24}, that $\mathcal{H}_N(\Omega)\subset \bigcap_{1\leq p<\infty}W^{1,p}\mathcal{V}(\Omega)$ and thus, due to Sobolev embeddings $\mathcal{H}_N(\Omega)\subset \bigcap_{0<\alpha<1}C^{0,\alpha}\mathcal{V}(\overline{\Omega})$. Since $\operatorname{curl}(\Gamma)=0$ it follows that the restriction $\Gamma|_{\Sigma}$ satisfies $\operatorname{curl}_{\Sigma}(\Gamma|_{\Sigma})=0$ in the weak sense and according to the Hodge-decomposition theorem \cite[Corollary 3.5.2]{S95} we can therefore write
\begin{gather}
	\nonumber
	\Gamma|_{\Sigma}=\nabla_{\Sigma}\alpha+\gamma
\end{gather}
for suitable $\alpha\in H^1(\Sigma)$ and $\gamma\in \mathcal{H}(\Sigma)$. Here $\gamma$ is the $L^2$-orthogonal projection of $\Gamma|_{\Sigma}$ onto $\mathcal{H}(\Sigma)$. We note that since $\mathcal{H}_N(\Omega)$ is $L^2(\Omega)$-orthogonal to the gradient fields and because $\Gamma=\operatorname{curl}(A)$ for a suitable vector potential $A$ (for instance $A$ may be taken to be the Biot-Savart potential of $\Gamma$) we must have $\gamma\neq 0$ whenever $\Gamma\neq 0$.

We are now finally ready to give an alternative interpretation of surface helicity in terms of the "dynamics of individual field lines".
\begin{thm}[Physical interpretation of surface helicity on toroidal surfaces]
	\label{T213}
	Let $T^2\cong \Sigma\subset\mathbb{R}^3$ bound a $C^{1,1}$-solid torus $\Omega\subset\mathbb{R}^3$. Let $\sigma_p$ be a purely poloidal curve and $\sigma_t$ be a toroidal curve. Further, fix any $\Gamma\in \mathcal{H}_N(\Omega)$ such that the $L^2(\Sigma)$-orthogonal projection $\gamma$ of $\Gamma|_{\Sigma}$ onto $\mathcal{H}(\Sigma)$ is $L^2(\Sigma)$-normalised. In addition let $\tilde{\gamma}:=\gamma\times\mathcal{N}$. Then, given any $v\in C^{0,1}\mathcal{V}_0(\Sigma)$, we have
	\begin{gather}
		\label{E3}
		\frac{\mathcal{H}(v)}{|\Sigma|^2}=\overline{Q}(v)\overline{P}(v)\int_{\sigma_t}\gamma\int_{\sigma_p}\tilde{\gamma}+\left[\mathcal{H}(\gamma)\left(\int_{\sigma_t}\gamma\right)^2+\int_{\sigma_t}\gamma\int_{\sigma_t}\tilde{\gamma}\right]\overline{Q}^2(v).
	\end{gather}
	If, moreover, $\sigma_t$ bounds a bounded $C^{1,1}$-surface $\mathcal{A}$ outside of $\Omega$, i.e. $\mathcal{A}\subset \mathbb{R}^3\setminus\overline{\Omega}$ and $\partial\mathcal{A}=\sigma_t$, then
	\begin{gather}
		\nonumber
		\frac{\mathcal{H}(v)}{|\Sigma|^2}=\overline{Q}(v)\overline{P}(v)\int_{\sigma_t}\gamma\int_{\sigma_p}\tilde{\gamma}.
	\end{gather}
\end{thm}
\begin{rem}[Sanity Check]
	\label{R214}
	We observe that $\dim(\mathcal{H}_N(\Omega))=1$ whenever $\Omega$ is a solid torus, \cite[Theorem 2.6.1]{S95}. Hence the condition of $L^2(\Sigma)$-normalisation of $\gamma$ determines $\Gamma$ uniquely up to a minus sign. If a chosen $\Gamma$ is replaced by $-\Gamma$ we see that $\gamma$ as well as $\tilde{\gamma}$ both are also changed by a minus sign and hence (since $\mathcal{H}(-\gamma)=\mathcal{H}(\gamma)$) the above formula is independent of that choice, as one would expect. Similarly, if $\sigma_p$ and $\tilde{\sigma}_p$ are purely poloidal curves on $\Sigma$ then either they are homotopic and thus the integrals involving $\sigma_p$ are unchanged or otherwise $\sigma_p$ and $\tilde{\sigma}_p$ are oppositely oriented, in which case $\int_{\tilde{\sigma}_p}\tilde{\gamma}=-\int_{\sigma_p}\tilde{\gamma}$, but $\overline{P}(v)$ also changes sign (since it is defined via $\gamma_p$ which in turn is defined via $\sigma_p$). In addition, $\gamma_t$ is independent of the chosen $\sigma_p$ and thus $\overline{Q}$ is also independent of that choice. Therefore the right hand side of (\ref{E3}) is independent of the chosen purely poloidal curve. Finally, we observe that $\Gamma|_{\Sigma}$ and $\gamma$ differ only by a gradient field such that $\int_{\sigma_t}\gamma=\int_{\sigma_t}\Gamma$ from which it follows that $\int_{\sigma_t}\gamma$ only depends on the orientation of $\sigma_t$ (considered as a curve in $\overline{\Omega}$) and it is clear once again that changing the orientation of $\sigma_t$ leaves the right hand side of (\ref{E3}) invariant. Since changing orientation of $\sigma_t$ leaves the right hand side invariant, we are left with understanding how the quantities $\int_{\sigma_t}\tilde{\gamma}$, $\overline{P}(v)$ and $\overline{Q}(v)$ change if we replace $\sigma_t$ by a toroidal curve $\tilde{\sigma}_t$ which is homotopic to $\sigma_t$ in $\overline{\Omega}$. It follows easily from (\ref{E1}) that $\gamma_t$ and hence $\overline{Q}(v)$ remains unchanged. On the other hand, since $\sigma_t$ and $\sigma_p$ form a basis of the fundamental group of $\Sigma$, we may express $\tilde{\sigma}_t=\sigma_t\oplus P\sigma_p$, where we used that we may assume that $\tilde{\sigma}_t$ is homotopic to $\sigma_t$. Then $\int_{\tilde{\sigma}_t}\tilde{\gamma}=\int_{\sigma_t}\tilde{\gamma}+P\int_{\sigma_p}\tilde{\gamma}$ and one can verify that $\tilde{\gamma}_p-\gamma_p=-P\gamma_t$ so that it follows from \Cref{L212} that $\widetilde{P}(v)-\overline{P}(v)=-P\overline{Q}(v)$ (where $\tilde{\gamma}_p$, $\widetilde{P}(v)$ are computed with respect to $\sigma_p$ and $\tilde{\sigma}_t$). We hence see that the additional terms appearing through the change of $\overline{P}(v)$ and $\int_{\sigma_t}\tilde{\gamma}$ when $\sigma_t$ is replaced by $\tilde{\sigma}_t$ cancel each other. We conclude that the right hand side of (\ref{E3}) is independent of the chosen curves $\sigma_p$ and $\sigma_t$, as it should be because the left hand side of (\ref{E3}) is of course independent of any such choices.
\end{rem}
\subsection{Surface helicity and rotational transform}
\label{SS23}
Assume once more that $T^2\cong \Sigma\subset\mathbb{R}^3$ bounds a $C^{1,1}$-solid torus $\Omega\subset\mathbb{R}^3$ and suppose that we fixed some purely poloidal curve $\sigma_p$ and a toroidal curve $\sigma_t$ on $\Sigma$. Now, if $\sigma:\mathbb{R}\rightarrow \Sigma$ is any other closed $C^1$-curve, then we can write it as $\sigma=P\sigma_p\oplus Q\sigma_t$, where $Q$ may be viewed as the "toroidal" turns of $\sigma$ and $P$ may be viewed as the "poloidal" turns of $\sigma$. We recall that according to (\ref{E2}) we can write
\begin{gather}
	\nonumber
	\frac{Q}{\tau}=\lim_{T\rightarrow\infty}\frac{1}{T}\int_{\sigma[0,T]}\gamma_t\text{, }\frac{P}{\tau}=\lim_{T\rightarrow\infty}\int_{\sigma[0,T]}\gamma_p,
\end{gather}
where $\gamma_p$,$\gamma_t$ are determined by the equations in (\ref{E1}). Consequently, if we are interested in the ratio of "poloidal" twists by "toroidal" twists $\frac{P}{Q}$, we obtain the identity (assuming $Q\neq 0$)
\begin{gather}
	\nonumber
	\iota(\sigma):=\frac{P}{Q}=\frac{\lim_{T\rightarrow\infty}\frac{1}{T}\int_{\sigma[0,T]}\gamma_p}{\lim_{T\rightarrow\infty}\frac{1}{T}\int_{\sigma[0,T]}\gamma_t}=\frac{\hat{p}(\sigma)}{\hat{q}(\sigma)},
\end{gather}
where $\hat{p}(\sigma)$,$\hat{q}(\sigma)$ denote the "asymptotic windings" of $\sigma$. These observations motivate the following definition.
\begin{defn}[Rotational transform]
	\label{D215}
	Let $T^2\cong\Sigma\subset\mathbb{R}^3$ bound a $C^{1,1}$-solid torus $\Omega\subset\mathbb{R}^3$. Let $\sigma_p$ and $\sigma_t$ be a given purely poloidal and a toroidal curve respectively and let $v\in C^{0,1}\mathcal{V}(\Sigma)$. If $\hat{q}(x)\neq 0$ for a given $x\in \Sigma$, c.f. \Cref{D211}, we define the \textit{rotational transform} of $v$ at $x$ as
	\begin{gather}
		\nonumber
		\iota_v(x):=\frac{\hat{p}(x)}{\hat{q}(x)},
	\end{gather}
provided the limits in the definition of $\hat{p}(x)$ and $\hat{q}(x)$ exist.
\end{defn}
\begin{rem}
	\label{R216}
	\begin{enumerate}
		\item According to \Cref{L212}, if $v\in C^{0,1}\mathcal{V}_0(\Sigma)$, then $\hat{p}(x)$ and $\hat{q}(x)$ exist for a.e. $x\in \Sigma$ w.r.t. the standard surface measure.
		\item We observe that the definition of the rotational transform depends on the choice of $\sigma_p$ and $\sigma_t$ in the following way: We recall first the defining equations (\ref{E1}) of $\gamma_p$ and $\gamma_t$
		\begin{gather}
			\nonumber
			\int_{\sigma_p}\gamma_t=0=\int_{\sigma_t}\gamma_p\text{, }\int_{\sigma_p}\gamma_p=1=\int_{\sigma_t}\gamma_t.
		\end{gather}
		If $\tilde{\sigma}_p$ is any other closed, purely poloidal curve, then either $\tilde{\sigma}_p$ is homotopic to $\sigma_p$ or otherwise it is homotopic to $-\sigma_p$ (the curve obtained from $\sigma_p$ by inverting the orientation). In both cases $\gamma_t$ remains unchanged, while $\gamma_p$ changes sign and so $\hat{q}$ remains unchanged, while $\hat{p}$ changes sign, i.e. at most the sign of $\iota$ is changed, depending whether $\tilde{\sigma}_p$ is oriented in the same way as $\sigma_p$ or oppositely oriented. On the other hand, if we change $\sigma_t$ to some $\tilde{\sigma}_t$, then by definition either $\tilde{\sigma}_t$ will be homotopic to $\sigma_t$ or $-\sigma_t$ as a curve in $\overline{\Omega}$ (but not necessarily within $\Sigma$). Since $\sigma_p$ and $\sigma_t$ generate the first fundamental group of $\Sigma$, we can then however express $\tilde{\sigma}_t=\pm\sigma_t\oplus P\sigma_p$ for a suitable $P\in \mathbb{Z}$, where the sign in $\pm$ depends on whether $\tilde{\sigma}_t$ is homotopic to $\sigma_t$ or $-\sigma_t$ within $\overline{\Omega}$. Consequently, we see that $\gamma_t$ at most changes sign, while $\tilde{\gamma}_p=\gamma_p\mp P\gamma_t$. Thus, $\hat{q}$ may change its sign, depending whether $\sigma_t$ and $\tilde{\sigma}_t$ are oriented in the same, or opposite, way within $\overline{\Omega}$ and $\hat{\tilde{p}}(x)=\hat{p}(x)\mp P\hat{q}(x)$, where $\hat{\tilde{p}}$ indicates that this quantity is computed with respect to the curve $\tilde{\sigma}_t$. In conclusion $\tilde{\iota}_v(x)=\pm \iota_v(x)\mp P$, where once more the tilde indicates that $\iota_v$ is computed with respect to the toroidal curve $\tilde{\sigma}_t$. In particular, if $\sigma_p$,$\tilde{\sigma}_p$ are two purely poloidal curves of the same orientation and $\sigma_t$,$\tilde{\sigma}_t$ are two poloidal curves of the same orientation (within $\overline{\Omega}$), then $\tilde{\iota}_v(x)=\iota_v(x)+\beta$ for some constant $\beta\in \mathbb{Z}$, which is independent of $x$ and the vector field $v$ under consideration. Hence, if we wish to study a problem which aims to maximise $\iota$ (possibly under some additional constraints), then the maximisers will be independent of the chosen curves (as long as they are oriented in the same way). Finally, also notice that the property $\hat{q}(x)\neq 0$ at a given $x\in \Sigma$ is independent of the chosen curves.
		\end{enumerate}
\end{rem}
We will now first state the main mathematical result of this section, before we shortly discuss the relevance of rotational transform in the context of plasma fusion confinement devices.

In order to do so we state here the (standard) definition of rectifiability of a vector field on a $2$-torus $T^2=\mathbb{R}^2\slash \mathbb{Z}^2$.
\begin{defn}[Rectifiability]
	\label{D217}
	Let $T^2\cong \Sigma$ be a $C^{1,1}$-surface, which is diffeomorphic to the $2$-torus. A vector field $v\in C^{0,1}\mathcal{V}(\Sigma)$ is said to be \textit{rectifiable} if there exists a diffeomorphisms $\psi:\Sigma\rightarrow T^2$ such that $\psi_*v=a\partial_1+b\partial_2$ for suitable constants $a,b\in \mathbb{R}$ and the standard coordinate fields $\partial_1,\partial_2$ on $T^2$, where $\psi_*v$ denotes the corresponding pushforward vector field.
\end{defn}
With this we can formulate the main result of this subsection.
\begin{thm}[Formula for rotational transform]
	\label{T218}
	Let $T^2\cong \Sigma\subset\mathbb{R}^3$ bound a $C^{1,1}$-solid torus $\Omega\subset\mathbb{R}^3$. Let $\sigma_p$ and $\sigma_t$ be a given purely poloidal and a toroidal curve on $\Sigma$ respectively. Further, let $\gamma_p,\gamma_t$ be a basis of $\mathcal{H}(\Sigma)$ determined by the condition (\ref{E1}). Assume that $v\in C^{0,1}\mathcal{V}(\Sigma)$ is rectifiable and that there exists some $f\in C^{0,1}(\Sigma,(0,\infty))$ such that $w:=fv\in C^{0,1}\mathcal{V}_0(\Sigma)$, i.e. $\operatorname{div}_{\Sigma}(w)=0$. If $\overline{Q}(w)\neq 0$, then the rotational transforms of $v$ and $w$ are well-defined at a.e. $x\in \Sigma$ and there holds
	\begin{gather}
		\nonumber
		\iota_v(x)=\iota_w(x)=\frac{\int_{\Sigma}w\cdot \gamma_pd\sigma}{\int_{\Sigma}w\cdot \gamma_t d\sigma}=\frac{\overline{P}(w)}{\overline{Q}(w)},
	\end{gather}
	where the subscript indicates whose rotational transform is considered. In particular, $\iota_v(x)$ is independent of $x$.
\end{thm}
\begin{rem}
	\label{R219}
	In general, for a given vector field $v\in C^{0,1}\mathcal{V}(\Sigma)$, even if the rotational transform $\iota_v(x)$ is an element of $L^1(\Sigma)$, it is not true that $\frac{1}{|\Sigma|}\int_{\Sigma}\iota(x)d\sigma(x)=\frac{\overline{P}(v)}{\overline{Q}(v)}$ because the average of a quotient does not need to coincide with the quotient of the respective averages.
\end{rem}
As an immediate consequence of \Cref{T218} and \Cref{T213}, we obtain the following relation between surface helicity and rotational transform.
\begin{cor}[Helicity and rotational transform]
	\label{C220}
	Let $T^2\cong \Sigma\subset\mathbb{R}^3$ bound a $C^{1,1}$-solid torus $\Omega\subset\mathbb{R}^3$. Let $\sigma_p$ be a purely poloidal curve on $\Sigma$ and let $\sigma_t$ be a toroidal curve on $\Sigma$ which bounds a bounded $C^{1,1}$-surface $\mathcal{A}$ outside of $\Omega$. Assume further that $v\in C^{0,1}\mathcal{V}(\Sigma)$ is rectifiable and that there exists some $f\in C^{0,1}(\Sigma,(0,\infty))$ such that $w:=fv\in C^{0,1}\mathcal{V}_0(\Sigma)$, i.e. $\operatorname{div}_{\Sigma}(w)=0$. If $\overline{Q}(w)\neq 0$, then
	\begin{gather}
		\nonumber
		\frac{\mathcal{H}(w)}{|\Sigma|^2}=\iota_w\overline{Q}^2(w)\left(\int_{\sigma_t}\gamma\right)\left(\int_{\sigma_p}\tilde{\gamma}\right),
	\end{gather}
	where $\iota_w$ is the constant rotational transform of $w$ and $\gamma$,$\tilde{\gamma}$ are defined as in \Cref{T213} and in particular are independent of $w$.
\end{cor}
\begin{rem}
	\label{R221}
	\begin{enumerate}
		\item We point out shortly that a corresponding formula for the vector field $v$ does not need to hold necessarily, since we require the divergence-free property in order to invoke \Cref{T213}, which is the reason why the conclusion is solely formulated for the vector field $w$.
		\item It is also customary to call a vector field $v\in C^{0,1}\mathcal{V}(\Sigma)$ semi-rectifiable, provided there exists a (strictly) positive $C^{0,1}$-function $f$ such that $fv$ is rectifiable, c.f. \cite{PPS22}. Hence \Cref{C220} can be reformulated by assuming that a given $w\in C^{0,1}\mathcal{V}_0(\Sigma)$ is semi-rectifiable.
	\end{enumerate}
\end{rem}
\begin{rem}[Plasma equilibria]
	\label{R222}
	In the realm of plasma physics one is interested in steady solutions of the equations of ideal magnetohydrodynamics with a resting plasma \cite[Equations (1) \& (3)]{Hel14}. Such plasma equilibria are determined by the PDE
	\begin{gather}
		\nonumber
		B\times \operatorname{curl}(B)=\nabla p\text{ and }\operatorname{div}(B)=0\text{ in }\mathcal{P}\text{, }B\parallel \partial \mathcal{P}
	\end{gather}
	where $\mathcal{P}$ denotes the region containing the plasma ("plasma region"), $B$ is the corresponding magnetic field and $p$ is the pressure. According to Arnold's structure theorem, \cite[Chapter II Theorem 1.2]{AK21}, \cite[Section 1.2]{Arnold2014}, if $B$ is analytic and $B$ and $\operatorname{curl}(B)$ are not everywhere collinear, the domain $\mathcal{P}$ will decompose, after removing a $2$-dimensional analytic subset from it, into a finite number of chambers, each of which is invariant under the flow of $B$ and is one of the following two types:
	\begin{enumerate}
		\item Either the chamber foliates further into toroidal surfaces, each of which is again invariant under the flow of $B$ and such that on each such surface $B$ is rectifiable and hence in particular either all field lines of $B$ are closed (with the same period) or all of the field lines of $B$ are dense within the surface.
		\item Or the chamber foliates into cylindrical surfaces diffeomorphic to $S^1\times \mathbb{R}$, each of which is invariant under the flow of $B$ and such that all field lines of $B$ are closed on any such surface.
	\end{enumerate}
	Of particular interest are the invariant toroidal surfaces. In particular, in order to counter the plasma particle drift in a fusion plasma it is desirable to have a high rotational transform \cite{Spi58}, see also \Cref{SB} for a precise relationship between the different notions of rotational transform. We therefore wish to understand the quantity $\iota_B$ on regular toroidal level sets of the pressure $p$. We note that if $\Sigma$ is a regular toroidal level set of $p$, then $|\nabla p|\neq 0$ on all of $\Sigma$ and in fact, letting $f:=\frac{1}{|\nabla p|}$ one can verify that $\widetilde{B}:=fB$ is a divergence-free field on $\Sigma$ (divergence-free with respect to the surface measure on $\Sigma$), c.f. \cite{PPS22}. It follows therefore from \Cref{T218} that we can compute the rotational transform of $B$ on a given regular toroidal level set of $\Sigma$ as
	\begin{gather}
		\nonumber
		\iota_B(x)=\frac{\int_{\Sigma}\widetilde{B}\cdot \gamma_pd\sigma}{\int_{\Sigma}\widetilde{B}\cdot \gamma_td\sigma}
	\end{gather}
	for a fixed purely poloidal curve $\sigma_p$ and a toroidal curve $\sigma_t$ bounding a surface outside of the domain enclosed by the level set and their corresponding induced harmonic fields $\gamma_p$,$\gamma_t\in \mathcal{H}(\Sigma)$. Note once more, that the rotational transform of $B$ (on any fixed regular toroidal surface) is constant.
\end{rem}
\subsection{Surface helicity optimisation}
\label{SS24}
In \cite{CDGT002} Cantarella, DeTurck, Gluck and Teytel considered the following optimisation in $3$-dimensional space: First, given a smooth, bounded domain $\Omega\subset \mathbb{R}^3$, they show that there exists some smooth vector field $v$ on $\Omega$ with $\operatorname{div}(v)=0$, $v\parallel\partial\Omega$ such that
\begin{gather}
	\nonumber
	\frac{\mathcal{H}_{3D}(v)}{\|v\|^2_{L^2(\Omega)}}=\sup_{\substack{w\neq 0\\ \operatorname{div}(w)=0\\ w\parallel\partial\Omega}}\frac{\mathcal{H}_{3D}(w)}{\|w\|^2_{L^2(\Omega)}}=:\lambda(\Omega),
\end{gather}
where $\mathcal{H}_{3D}$ denotes the standard $3$-dimensional Biot-Savart helicity, see \cite[Section 9]{CDG00} for a proof of the above fact. Further, \cite[Theorem B]{CDG00}, implies that $|\lambda(\Omega)|\leq c_0\sqrt[3]{|\Omega|}$ for an absolute constant $c_0$ independent of $\Omega$. In view of this it is natural to ask whether there exists a bounded, smooth domain $\Omega$ which maximises $\lambda(\Omega)$ among all other domains of the same volume. Physically this amounts to finding a domain $\Omega$ which supports a magnetic field (tangent to the boundary of $\Omega$) whose field lines are (on average) as entangled as possible among all other magnetic fields and all other domains of the same volume.

Some necessary conditions, assuming the existence of optimal domains, have been derived in \cite[Theorem D]{CDGT002}. Up to this day the question of existence of optimal domains for this problem has remained open, even though some recent advances have been made within certain classes of domains, c.f. \cite{EGPS23},\cite{G24OptimalHelicityDomain}.
\newline
\newline
Here we consider a corresponding problem for the surface helicity. The first result is the following.
\begin{thm}[Helicity optimisers exist]
	\label{T223}
	Let $\Sigma\subset\mathbb{R}^3$ be a closed, connected $C^{1,1}$-surface. Then there exists some $v_{\pm}\in L^2\mathcal{V}_0(\Sigma)$ with $\|v_{\pm}\|_{L^2(\Sigma)}=1$ and such that
	\begin{gather}
		\label{E4}
		\pm\mathcal{H}(v_{\pm})=\max_{\substack{w\neq 0\\
		w\in L^2\mathcal{V}_0(\Sigma)}}\frac{\pm\mathcal{H}(w)}{\|w\|^2_{L^2(\Sigma)}}=:\Lambda_{\pm}(\Sigma).
	\end{gather}
	Further,
	\begin{enumerate}
		\item If $\Sigma\cong S^2$, then every non-zero element of $L^2\mathcal{V}_0(\Sigma)$ realises the above Rayleigh quotients.
		\item If $\Sigma\not\cong S^2$, then $v_{\pm}\in \mathcal{H}(\Sigma)$ (i.e. $v_{\pm}$ are additionally curl-free).
	\end{enumerate}
\end{thm}
To state the next result we denote by $\pi$ the $L^2$-orthogonal projection from $L^2\mathcal{V}(\Sigma)$ onto $L^2\mathcal{V}_0(\Sigma)$.
\begin{thm}[Spectral theoretical characterisation of helicity optimisers]
	\label{T224}
	Let $\Sigma\subset\mathbb{R}^3$ be a closed, connected $C^{1,1}$-surface. Then the operator $\pi\circ \operatorname{BS}_{\Sigma}:L^2\mathcal{V}_0(\Sigma)\rightarrow L^2\mathcal{V}_0(\Sigma)$ is compact and self-adjoint and admits a largest positive and smallest negative eigenvalue. Further, $v\in L^2\mathcal{V}_0(\Sigma)$ maximises helicity among all other fields $w\in L^2\mathcal{V}_0(\Sigma)$ of the same $L^2$-norm if and only if $v$ is an eigenfield of $\pi\circ \operatorname{BS}_{\Sigma}$ corresponding to the largest positive eigenvalue. A corresponding result holds for the helicity minimisation problem.
\end{thm}
\begin{rem}
	In general, it is highly non-trivial to compute the corresponding eigenfields of $\pi\circ \operatorname{BS}_{\Sigma}$ which realise the Rayleigh quotients in (\ref{E4}) even for rotationally symmetric surfaces. However, we provide in \Cref{SC} an, in comparison, easy way to compute these eigenfields as well as the corresponding values $\Lambda_+(\Sigma)$, $\Lambda_-(\Sigma)$ if the surface $\Sigma$ is toroidal. This will allow us to obtain exact solutions in the case of rotationally symmetric surfaces.
\end{rem}
We now deal with the question of optimising the following quantity, recall (\ref{E4}),
\begin{gather}
	\label{ExtraExtraGleich}
	\Lambda(\Sigma):=\max\{\Lambda_+(\Sigma),\Lambda_-(\Sigma)\}=\max_{\substack{w\neq 0\\
			w\in L^2\mathcal{V}_0(\Sigma)}}\frac{|\mathcal{H}(w)|}{\|w\|^2_{L^2(\Sigma)}}
\end{gather}
among $C^{1,1}$-closed surfaces. Note that according to \Cref{T25} we have $\Lambda(\Sigma)\geq 0$ with equality if and only if $\Sigma\cong S^2$. Hence, $\Sigma$ is a global minimiser of $\Lambda(\Sigma)$ if and only if $\Sigma\cong S^2$. Here, due to their relevance in plasma physics, we want to focus on surfaces $\Sigma\cong T^2$.
\begin{thm}[Optimal non-trivial lower bound for $\Lambda(\Sigma)$]
	\label{T225}
	Let $T^2\cong\Sigma\subset\mathbb{R}^3$ be a closed, connected $C^{1,1}$-surface and let $\Omega$ be the domain of finite volume bounded by $\Sigma$. Then
	\begin{gather}
		\nonumber
		\Lambda(\Sigma)\geq \frac{1}{2}.
	\end{gather}
	Equality holds if and only if $\mathcal{H}(\gamma)=0$, where $\gamma$ is the $L^2$-orthogonal projection of $\Gamma|_{\Sigma}$ onto $\mathcal{H}(\Sigma)$ and $\Gamma$ is any fixed non-zero harmonic Neumann field of $\Omega$, i.e. any $H^1$-regular, curl-free, div-free vector field on $\Omega$ which is tangent to $\partial\Omega=\Sigma$.
\end{thm}
\begin{rem}
	With some more effort it should be possible to extend the same lower bound to surfaces of arbitrary (strictly positive) genus using similar methods as in the proof of \Cref{T225} by working with a suitable basis of $\mathcal{H}_N(\Omega)$ in order to obtain an appropriate matrix representation of the surface Biot-Savart operator. However, we do not pursue such a possible extension further and instead list this as an open problem below.
\end{rem}
Using the linking interpretation of helicity, \Cref{T26} and \Cref{R27}, we will be able to conclude the following.
\begin{cor}
	\label{C226}
	Let $\Sigma\subset\mathbb{R}^3$ be a closed, connected, $C^{1,1}$-regular, axisymmetric torus. Then $\Lambda(\Sigma)=\frac{1}{2}$. Consequently
	\begin{gather}
		\nonumber
		\Lambda(\Sigma)=\min_{T^2\cong\widetilde{\Sigma}\in C^{1,1}}\Lambda(\widetilde{\Sigma}).
	\end{gather}
\end{cor}
\begin{rem}
	\label{R227}
	We note that in \Cref{T225} and \Cref{C226} we did not specify the area of $\Sigma$. As we shall see, the scaling properties of the functional $\Lambda$ imply that minimising or maximising $\Lambda$ among surfaces with prescribed area is equivalent to optimising $\Lambda$ among surfaces without an area constraint.
\end{rem}
We recall that the corresponding $3$-d problem motivated to try to find a global maximiser of $\Lambda(\Sigma)$ rather than a global minimiser. We pose this as a remaining open problem.
\newline
\newline
\textbf{Open Problems:}
\begin{enumerate}
	\item Does there exist a constant $0<c<\infty$ such that $\Lambda(\Sigma)\leq c$ for all $T^2\cong \Sigma$ which are $C^{1,1}$-regular?
	\item If the answer to (i) is positive, does there exist some $T^2\cong \Sigma$ such that $\Lambda(\Sigma)=\max_{T^2\cong \widetilde{\Sigma}\in C^{1,1}}\Lambda(\widetilde{\Sigma})$?
	\item Do the answers to (i) and (ii) change if we drop the assumption $T^2\cong\Sigma$ and instead allow $\Sigma$ to have arbitrary genus?
	\item Does \Cref{T225} extend to surfaces of arbitrary genus, i.e. is it true that every closed, connected $C^{1,1}$-surface $\Sigma\subset\mathbb{R}^3$ of genus $g\geq 1$ satisfies $\Lambda(\Sigma)\geq \frac{1}{2}$?
\end{enumerate}
\subsection{Simple Surface currents}
\label{SS25}
When building plasma fusion confinement devices one wishes to generate magnetic fields in order to confine the fusion plasma. In particular, in so called stellarator devices, one uses a complex coil structure to generate the fields \cite{Xu16}. One way to model these complex coil structures is the so called coil winding surface (CWS) model, \cite{M87}, where one assumes that the coils are infinitely thin and wind around a given closed surface. The induced current is then modelled as a vector field tangent to the CWS and is assumed to be divergence-free in view of Maxwell's equations. Given a $C^{1,1}$-regular CWS $\Sigma\subset \mathbb{R}^3$ and a divergence-free current $j$ on the CWS, the corresponding magnetic field induced by $j$ within the domain $\Omega$ enclosed by $\Sigma$ is given by the Biot-Savart law
\begin{gather}
	\nonumber
	\operatorname{BS}(j)(x)=\frac{1}{4\pi}\int_{\Sigma}j(y)\times \frac{x-y}{|x-y|^3}d\sigma(y)\text{, }x\in \Omega.
\end{gather}
Now, for a fixed current $j$, there exist in general further currents $j^\prime$ which induce the same magnetic field within $\Omega$. In fact, since $\operatorname{BS}$ is a linear operator, we must have that $j-j^\prime\in \operatorname{Ker}(\operatorname{BS})$. To be more precise, we may view $\operatorname{BS}$ as a bounded operator from $L^2\mathcal{V}_0(\Sigma)$ into $L^2\mathcal{V}(\Omega)$ (the square integrable vector fields on $\Omega$), c.f. \cite[Lemma C.1]{G24}. It has been further shown, see \cite[Theorem 5.1]{G24} that $\dim\left(\operatorname{Ker}(\operatorname{BS})\right)=g(\Sigma)$, where $g(\Sigma)$ denotes the genus of $\Sigma$. In our applications, the CWSs are toroidal, i.e. $T^2\cong \Sigma$, and thus $\dim\left(\operatorname{Ker}(\operatorname{BS})\right)=1$. This gives us some flexibility and hence we may look for currents, and consequently coil configurations, which have a "simple" form. To this end we make the following definition.
\begin{defn}[Simple currents]
	\label{D228}
	Let $T^2\cong \Sigma\subset\mathbb{R}^3$ bound a $C^{1,1}$-solid torus $\Omega\subset\mathbb{R}^3$. Given a purely poloidal curve $\sigma_p$ and a toroidal curve $\sigma_t$ we further let $\gamma_t\in \mathcal{H}(\Sigma)$ denote the induced harmonic field. Given a current $j\in L^2\mathcal{V}_0(\Sigma)$ we define $\overline{Q}(j):=\frac{\int_{\Sigma}j\cdot \gamma_t d\sigma}{|\Sigma|}$ and say that $j$ \textit{simple} if $\overline{Q}(j)=0$. 
\end{defn}
\begin{rem}
	\label{R229}
	\begin{enumerate}
		\item In view of \Cref{L212}, \Cref{D228} coincides with \Cref{D211} whenever $j$ is of class $C^{0,1}$.
		\item We have previously explained in \Cref{R216} that $\gamma_t$ is independent of the choice of $\sigma_p$ and that changing $\sigma_t$ may at most change $\gamma_t$ to $-\gamma_t$ and therefore the definition of a simple current is independent of the chosen curves $\sigma_p$ and $\sigma_t$.
		\item If $j\in C^{0,1}\mathcal{V}_0(\Sigma)$, then according to \Cref{L212} we see that $j$ is simple if and only if $\int_{\Sigma}\hat{q}(x)d\sigma(x)=0$, or in other words if and only if, on average, the weighted asymptotic toroidal windings of the field lines of $j$ are zero. Therefore, in an oversimplified manner of speaking, the field lines of a simple current should not form any "knotted" structures and be expected to tend to be either more poloidal or contractible in nature.
		\item A word of caution: It might turn out that the field lines of $j$ are all toroidal curves and that simply on average the corresponding field lines point as often in positive toroidal direction as they point in negative toroidal direction. As an example consider $\Sigma\subset\mathbb{R}^3$ to be a standard rotationally symmetric torus and let $\phi$ denote the toroidal coordinate and $\theta$ the poloidal coordinate. We may then define $j:=\sin(\theta)e_{\phi}$, where $e_{\phi}$ denotes the standard normalised toroidal field. Then $j\in C^{0,1}\mathcal{V}_0(\Sigma)$ and $\overline{Q}(j)=0$, as follows easily from \Cref{D211}, and hence $j$ is simple. However, $j$ has only a toroidal component, see \Cref{SimpleCurrent}.
	\end{enumerate}
\end{rem}

\begingroup\centering
\begin{figure}[H]
	\hspace{4.5cm}\includegraphics[width=0.4\textwidth, keepaspectratio]{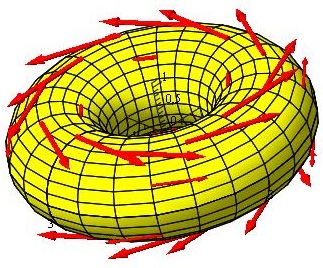}
	\caption{An example of a simple current with only toroidal field lines which point in opposite directions on the upper and lower "hemisphere" of the torus.}
	\label{SimpleCurrent}
\end{figure}
\endgroup

Here is our main result regarding the existence of simple current configurations.
\begin{thm}
	\label{T230}
	Let $T^2\cong \Sigma\subset\mathbb{R}^3$ bound a $C^{1,1}$-solid torus $\Omega\subset\mathbb{R}^3$. Then for every $j\in L^2\mathcal{V}_0(\Sigma)$ there exists a unique $j^\prime\in L^2\mathcal{V}_0(\Sigma)$ such that $\operatorname{BS}(j^\prime)=\operatorname{BS}(j)$ in $\Omega$ and $j^\prime$ is simple. Further, if $j\in L^p\mathcal{V}_0(\Sigma)$ for some $2\leq p\leq \infty$ or $j\in C^{0,\alpha}\mathcal{V}_0(\Sigma)$ for some $0\leq \alpha<1$, then so is $j^\prime$, where $C^{0,0}\equiv C^0$ denotes the continuous currents.
\end{thm}
In the context of plasma fusion confinement devices one of the situations one faces is the following: We suppose we are given a CWS $\Sigma$ and a plasma region $P$ which is assumed to be precompact within the finite domain $\Omega$ bounded by the CWS and is thought of to contain the fusion plasma. Further, we assume we are given a square integrable, harmonic, target magnetic field $B_T\in L^2\mathcal{H}(P):=\{X\in L^2(P,\mathbb{R}^3)\mid \operatorname{curl}(X)=0=\operatorname{div}(X)\}$ which we wish to generate by a surface current in order to confine the fusion plasma. It is known, see \cite[Corollary 3.10 (ii b)]{G24}, that under reasonable assumptions, the image of $\operatorname{BS}$, with domain $L^2\mathcal{V}_0(\Sigma)$, is contained in and $L^2(P)$-dense in the space $L^2\mathcal{H}(P)$. We therefore obtain the following important corollary from \Cref{T230}, where the slightly technical assumptions in the upcoming results are depicted in \Cref{Density}.

\begingroup\centering
\begin{figure}[H]
	\hspace{5.5cm}\includegraphics[width=0.27\textwidth, keepaspectratio]{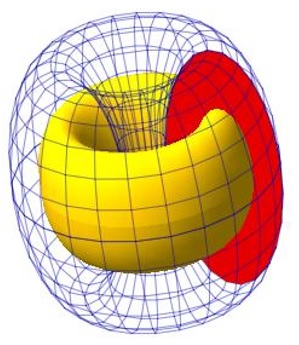}
	\caption{The CWS $\Sigma$ depicted as a blue grid, the plasma domain $P$ depicted in yellow and the disc $D$ depicted in red. Both $D$ and $D\cap P$ are bounded by poloidal loops respectively.}
	\label{Density}
\end{figure}
\endgroup

\begin{cor}[Density of magnetic fields induced by simple currents]
	\label{C231}
	Let $T^2\cong \Sigma\subset\mathbb{R}^3$ bound a $C^{1,1}$-solid torus $\Omega\subset\mathbb{R}^3$. Assume further, that we are given another $C^{1,1}$-solid torus $P$ which is contained in $\Omega$ and of positive distance to $\Sigma$ and that there exists a $C^2$-disc $D\subset \Omega$ such that $\partial D\subset \Sigma$ is a purely poloidal $C^1$-curve in $\Sigma$ and such that $D\cap P$ is also a disc bounded by a purely poloidal $C^1$-curve contained in $\partial P$. Then for every $B_T\in L^2\mathcal{H}(P)$ and every $\epsilon>0$ there exists a current $j\in L^2\mathcal{V}_0(\Sigma)$ with $\overline{Q}(j)=0$ and $\|\operatorname{BS}(j)-B_T\|_{L^2(P)}\leq \epsilon$.
\end{cor}

Our final result provides a minimisation procedure which allows one to obtain the desired approximating simple currents as minimisers of a specific "energy functional". The procedure is a modification  of a minimisation procedure which has already been used successfully in the past. We first recall the already established procedure, c.f. \cite{PRS22}. Given a positive parameter $\lambda>0$ one can consider the following minimisation problem
\begin{gather}
	\nonumber
	C(\lambda;B_T):=\inf_{j\in L^2\mathcal{V}_0(\Sigma)}\left(\|\operatorname{BS}(j)-B_T\|^2_{L^2(P)}+\lambda\|j\|^2_{L^2(\Sigma)}\right).
\end{gather}
It follows from standard variational techniques and the convexity of the functional involved that the infimum is in fact a minimum and that there exists a unique current $j_{\lambda}\in L^2\mathcal{V}_0(\Sigma)$ which realises $C(\lambda;B_T)$. Using the density of the image of the Biot-Savart operator within $L^2\mathcal{H}(P)$ it has been shown that $\|\operatorname{BS}(j_{\lambda})-B_T\|_{L^2(P)}\rightarrow 0$ as $\lambda\searrow 0$, see \cite[Corollary 4.3]{G24}. In general, there is no reason to expect that the obtained currents $j_{\lambda}$ satisfy $\overline{Q}(j_{\lambda})=0$. However, one can enforce this condition by observing that 
\begin{gather}
	\nonumber
L^2\mathcal{V}^{\overline{Q}=0}_0(\Sigma):=\{j\in L^2\mathcal{V}_0(\Sigma)\mid \overline{Q}(j)=0\}
\end{gather}
is an $L^2$-closed subspace, and hence a Hilbert space in its own right, and by setting up a corresponding minimisation procedure
\begin{gather}
	\label{E25}
	C_0(\lambda;B_T):=\inf_{j\in L^2\mathcal{V}^{\overline{Q}=0}_{0}(\Sigma)}\left(\|\operatorname{BS}(j)-B_T\|^2_{L^2(P)}+\lambda\|j\|^2_{L^2(\Sigma)}\right).
\end{gather}
Just like in the case of $C(\lambda;B_T)$ it follows from standard arguments that the infimum in $C_0(\lambda;B_T)$ is actually achieved and that the global minimiser is unique. Our final result is the following
\begin{thm}[Approximating sequence of simple currents]
	\label{T232}
	Let $T^2\cong \Sigma\subset\mathbb{R}^3$ bound a $C^{1,1}$-solid torus $\Omega\subset\mathbb{R}^3$. Assume further that we are given another $C^{1,1}$-solid torus $P$ which is contained in $\Omega$ and of positive distance to $\Sigma$ and that there exists a $C^2$-disc $D\subset \Omega$ such that $\partial D\subset \Sigma$ is a purely poloidal $C^1$-curve in $\Sigma$ and such that $D\cap P$ is also a disc bounded by a purely poloidal $C^1$-curve contained in $\partial P$. Given any $B_T\in L^2\mathcal{H}(P)$ we denote by $j^0_{\lambda}$ the (unique) global minimiser realising $C_0(\lambda;B_T)$ in (\ref{E25}) for fixed $\lambda>0$. Then $\|\operatorname{BS}(j^0_{\lambda})-B_T\|_{L^2(P)}\rightarrow 0$ as $\lambda\searrow 0$.
\end{thm}
\section{Proofs of main results}
\label{S3}
\subsection{Proof of \Cref{T24}}
Before we come to the proof let us recall how the $(0,2,3)$-helicity is defined. Given a closed $1$-form $\alpha$ on a surface $\Sigma$ we can consider the inclusion $\iota: (\Sigma\times \Sigma)\setminus \{x=y\}\rightarrow \Sigma\times \Sigma$ and the projections $\pi_x:\Sigma\times \Sigma\rightarrow \Sigma,(x,y)\mapsto x$ and $\pi_y:\Sigma\times \Sigma\rightarrow \Sigma,(x,y)\mapsto y$. We define the Gauss map $g:(\Sigma\times \Sigma)\setminus \{x=y\}\rightarrow S^2\text{, }(x,y)\mapsto \frac{x-y}{|x-y|}$ and we let $\omega^2_{S^2}$ denote the standard normalised(!) area form on $S^2$. We can then define $\alpha_x:=(\pi_x)^\#\alpha$, $\alpha_y:=(\pi_y)^\#\alpha$ (here $f^\#$ denotes the pullback via a map $f$). Now, strictly speaking, the space $C_2(\Sigma):=(\Sigma\times \Sigma)\setminus \{x=y\}$ is first compactified by means of the Fulton-MacPherson compactification, c.f. \cite[Definition 2.3]{CP10}, and it is shown that the forms $\alpha_x$, $\alpha_y$ as well as the Gauss map extend smoothly to the compactification \cite[Theorem 2.4 \& Lemma 2.10]{CP10}. Further, the compactification has a manifold structure \cite[Theorem 2.4]{CP10}. Denoting by $C_2[\Sigma]$ the compactification and with a slight abuse of notation we can define the helicity as
\[
\mathcal{H}_{(0,2,3)}(\alpha):=\int_{C_2[\Sigma]}\alpha_x\wedge\alpha_y\wedge g^\#\omega^2_{S^2}.
\]
We note that the space $C_2[\Sigma]$ is $4$-dimensional and so it makes sense to integrate a $4$-form. The same reasoning as in \cite[Theorem 2.15]{CP10} tells us that it is enough to prove that for all $(x,y)\in (\Sigma\times \Sigma)\setminus \{x\neq y\}$ we have the identity
\[
\alpha_x\wedge\alpha_y\wedge g^\#\omega^2_{S^2}=\frac{1}{4\pi}v(x)\cdot \left(v(y)\times \frac{x-y}{|x-y|^3}\right)\omega_\Sigma(x)\wedge \omega_\Sigma(y)
\]
where $\omega_\Sigma$ is the area-form on $\Sigma$ induced by the Euclidean metric and where $v$ is the unique vector field on $\Sigma$ satisfying $\iota_v\omega_\Sigma=\alpha$.
\begin{proof}[Proof of \Cref{T24}]
	We prove more generally that if $\alpha$ and $\beta$ are two $1$-forms on $\Sigma$, then
	\[
	\beta_x\wedge \alpha_y\wedge g^\#\omega^2_{S^2}=\frac{1}{4\pi}w(x)\cdot \left(v(y)\times \frac{x-y}{|x-y|^3}\right)\omega_\Sigma(x)\wedge \omega_\Sigma(y)
	\]
	where $w$ is the vector field associated with $\beta$ and $v$ is the vector field associated with $\alpha$ by means of contracting the area form. We fix any $x,y\in \Sigma$ with $x\neq y$ and let $\mu,\eta$ be charts of $\Sigma$ around $x$ and $y$ respectively. Additionally, we choose $\mu\times \eta$ as a chart around $(x,y)\in (\Sigma\times \Sigma)\setminus \{x=y\}$. Further, we let $p,q$ be any two points contained in the domains of $\mu$ and $\eta$ respectively. We see that, if we denote by $d\mu^i$,$d\eta^j$ the induced covector bases at $p,q\in \Sigma$ respectively and with a slight abuse of notation by $d\mu^i(p,q)$,$d\eta^j(p,q)$ the induced basis at $(p,q)\in (\Sigma\times \Sigma)\setminus \{x=y\}$, then, if we express $\beta(p)=\beta_i(p)d\mu^i(p)$, we find 
	\begin{gather}
		\nonumber
	\beta_x(p,q)=((\pi_x)^\#\beta)(p)=\sum_{i=1}^2\beta_i(p)d\mu^i(p,q).
	\end{gather}
	Similarly, $\alpha_y(p,q)=\sum_{j=1}^2\alpha_j(q)d\eta^j(p,q)$. As for the remaining term $g^\#\omega^2_{S^2}$ we recall that we have the identity
	\begin{gather}
		\nonumber
	\omega^2_{S^2}=\frac{1}{4\pi}\iota^\#\left(x dy\wedge dz-ydx\wedge dz+zdx\wedge dy\right)
	\end{gather}
	where $\iota$ denotes the inclusion $S^2\hookrightarrow\mathbb{R}^3$. Hence, we may view the Gauss map $g$ as a map into $\mathbb{R}^3$ rather than $S^2$ and are left with computing the pullback
	\begin{gather}
		\nonumber
	G^\#\left(x dy\wedge dz-ydx\wedge dz+zdx\wedge dy\right)(p,q)
	\end{gather}
	where $G:(\Sigma\times \Sigma)\setminus \{x=y\}\rightarrow \mathbb{R}^3,(x,y)\mapsto \frac{x-y}{|x-y|}$. In order to find an appropriate expression for the coefficients of this $2$-form we observe that
	\begin{gather}
		\nonumber
		G_*(p,q)\partial_{\mu^i}(p,q)=\frac{\partial_{\mu^i}(p)+\frac{(p-q)\cdot \partial_{\mu^i}(p)}{|p-q|^2}(q-p)}{|p-q|} \\
		\nonumber
		G_*(p,q)\partial_{\eta^j}(p,q)=-\frac{\partial_{\eta^j}(q)+\frac{(p-q)\cdot \partial_{\eta^j}(q)}{|p-q|^2}(q-p)}{|p-q|}
	\end{gather}
	where the vectors tangent to $\Sigma$ are viewed as vectors in $\mathbb{R}^3$. Defining $\omega(x,y,z):=xdy\wedge dz-ydx\wedge dz+z dx\wedge dy$, we observe that by definition of the pullback we have
	\begin{gather}
		\nonumber
	(G^\#\omega)(p,q)\left(\partial_{\mu^1}(p,q),\partial_{\mu^2}(p,q)\right)=\omega\left(\frac{p-q}{|p-q|}\right)\left(G_*(p,q)\partial_{\mu^1}(p,q),G_*(p,q)\partial_{\mu^2}(p,q)\right).
	\end{gather}
	Letting $\hat{r}:=\frac{p-q}{|p-q|}$, $\nu:=\partial_{\mu_1}(p)$ and $\xi:=\partial_{\mu_2}(p)$ we note that we need therefore to compute
	\begin{gather}
		\nonumber
	\omega(\hat{r})\left(\nu-(\hat{r}\cdot \nu)\hat{r},\xi-(\hat{r}\cdot \xi)\hat{r}\right)
	\end{gather}
	and an explicit computation yields
	\begin{gather}
		\nonumber
	\omega(\hat{r})\left(\nu-(\hat{r}\cdot \nu)\hat{r},\xi-(\hat{r}\cdot \xi)\hat{r}\right)=\left(\nu\times \xi\right)\cdot \hat{r}.
	\end{gather}
	Using this formula we can compute all the coefficient functions and eventually arrive at
	\begin{gather}
		\nonumber
		\left(G^\#\omega\right)(p,q)
		\\
		\nonumber
		=\frac{p-q}{4\pi|p-q|^3}\cdot\left(\partial_{\mu^1}(p)\times \partial_{\mu^2}(p)d\mu^1(p,q)\wedge d\mu^2(p,q)+\partial_{\eta^1}(q)\times \partial_{\eta^2}(q)d\eta^1(p,q)\wedge d\eta^2(p,q)\right.
		\\
		\nonumber
		\left. -\sum_{i,j=1}^2\partial_{\mu^i}(p)\times \partial_{\eta^j}(q)d\mu^i(p,q)\wedge d\eta^j(p,q)\right)
	\end{gather}
	where we again view the tangent vectors at $\Sigma$ as vectors in $\mathbb{R}^3$. We observe that the $2$-form $\beta_x\wedge \alpha_y$ only contains basis vectors of the form $d\mu^i(p,q)\wedge d\eta^j(p,q)$ so that all the $d\mu^1\wedge d\mu^2$ and $d\eta^1\wedge d\eta^2$ terms from $G^\#\omega$ will vanish in the product $\beta_x\wedge \alpha_y\wedge G^\#\omega$ because at least one of the covectors $d\mu^i$ or $d\eta^j$ in these expressions must repeat, so that only the mixed terms are of relevance. We obtain
	\begin{gather}
		\nonumber
		\beta_x(p,q)\wedge \alpha_y(p,q)\wedge g^\#\omega^2_{S^2}(p,q)
		\\
		\nonumber
		=\frac{q-p}{4\pi|p-q|^3}\cdot\sum_{i,j,k,m=1}^2\beta_k(p)\alpha_m(p)(\partial_{\mu^i}(p)\times \partial_{\eta^j}(q))d\mu^k(p,q)\wedge d\eta^m(p,q)\wedge d\mu^i(p,q)\wedge d\eta^j(p,q)
		\\
		\nonumber
		=\frac{q-p}{4\pi|p-q|^3}\cdot\left((\beta_2(p)\partial_{\mu_1}(p)-\beta_1(p)\partial_{\mu_2}(p))\times (\alpha_1(q)\partial_{\eta^2}(q)-\alpha_2(q)\partial_{\eta^1}(q))\right)d\mu^1\wedge d\mu^2\wedge d\eta^1\wedge d\eta^2.
	\end{gather}
	We recall that we started with two points $x,y\in \Sigma$ with $x\neq y$ and $\mu$ and $\eta$ were any arbitrary charts centred around these two points. In particular, we may pick our charts such that $g_{ij}(x)=\delta_{ij}=\tilde{g}_{ij}(y)$ at these two points where $g_{ij}$ and $\tilde{g}_{ij}$ denote the metric tensor coefficients in the respective coordinate charts. Further, with the right choice of orientation on $\Sigma$ we may assume that $\partial_{\mu^1}(x)\times \partial_{\mu^2}(x)=\mathcal{N}(x)$ where $\mathcal{N}$ denotes the outward pointing unit normal. Using the musical isomorphism we can then identify the forms $\alpha$ and $\beta$ with vector fields $X$ and $Z$ respectively (note that this identification differs from identifying a $1$-form by contracting the area-form). We can then expand $(x-y)$ in the basis $\partial_{\mu^1}(x),\partial_{\mu^2}(x),\mathcal{N}(x)$ and using the properties of the chart we find
	\begin{gather}
		\nonumber
	(x-y)\times (\beta_2(x)\partial_{\mu^1}(x)-\beta_1(x)\partial_{\mu^2}(x))=[(y-x)\cdot Z(x)]\mathcal{N}(x)+[(x-y)\cdot \mathcal{N}(x)]Z(x).
	\end{gather}
	Applying the cyclic properties of the inner product we find
	\begin{gather}
		\nonumber
		(x-y)\cdot [(\beta_2(x)\partial_{\mu^1(x)}-\beta_1(x)\partial_{\mu^2}(x))\times (\alpha_1(y)\partial_{\eta^2}(y)-\alpha_2(y)\partial_{\eta^1}(y))]
		\\
		\nonumber
		=(\alpha_1(y)\partial_{\eta^2}(y)-\alpha_2(y)\partial_{\eta^1}(y))\cdot \left([(y-x)\cdot Z(x)]\mathcal{N}(x)+[(x-y)\cdot \mathcal{N}(x)]Z(x)\right).
	\end{gather}
	Now we note that if $a\in \mathbb{R}^3$ is an arbitrary vector, then upon expanding it in the (orthonormal) basis $\{\partial_{\eta^1}(y),\partial_{\eta^2}(y),\mathcal{N}(y)\}$, we obtain
	\begin{gather}
		\nonumber
	\mathcal{N}(y)\cdot [X(y)\times a]=\alpha_1(y)a^2-\alpha_2(y)a^1=(\alpha_1(y)\partial_{\eta^2}(y)-\alpha_2(y)\partial_{\eta^1}(y))\cdot a.
	\end{gather}
	We hence arrive at
	\begin{gather}
		\nonumber
		\beta_x(x,y)\wedge \alpha_y(x,y)\wedge g^\#\omega^2_{S^2}(x,y)
		\\
		\nonumber
		=\frac{\mathcal{N}(y)}{4\pi|x-y|^3}\cdot\left[[((y-x)\cdot Z(x))\mathcal{N}(x)+((x-y)\cdot \mathcal{N}(x))Z(x)]\times X(y)\right]\omega_\Sigma(x)\wedge \omega_\Sigma(y)
		\\
		\nonumber
		=\frac{X(y)\times \mathcal{N}(y)}{4\pi|x-y|^3}\cdot [((y-x)\cdot Z(x))\mathcal{N}(x)+((x-y)\cdot \mathcal{N}(x))Z(x)]\omega_\Sigma(x)\wedge \omega_\Sigma(y)
		\\
		\nonumber
		=\frac{X(y)\times \mathcal{N}(y)}{4\pi|x-y|^3}\cdot [(Z(x)\times \mathcal{N}(x))\times (y-x)]\omega_\Sigma(x)\wedge \omega_\Sigma(y)
	\end{gather}
	where we used the vector-triple product rule in the last step. Lastly, we observe that the vector field associated with a $1$-form $\alpha$ by contracting the area-form can be obtained from the vector field associated with the $1$-form $\alpha$ via the musical isomorphism by taking the cross product with the outward pointing unit normal. So if we let $v,w$ be the vector fields associated with the $1$-forms $\alpha$ and $\beta$ by contracting the area-form respectively, we arrive at
	\begin{gather}
		\nonumber
	\beta_x(x,y)\wedge \alpha_y(x,y)\wedge (g^\#\omega^2_{S^2})(x,y)=v(y)\cdot \left(w(x)\times \frac{y-x}{4\pi|x-y|^3}\right)\omega_\Sigma(x)\wedge \omega_\Sigma(y).
	\end{gather}
	Consequently
	\begin{gather}
		\nonumber
	\int_{C_2[\Sigma]}\beta_x\wedge \alpha_y\wedge g^\#\omega^2_{S^2}=\int_\Sigma v(y)\cdot \operatorname{BS}_\Sigma(w)(y)d\sigma(y)
	\end{gather}
	and hence $\mathcal{H}_{(0,2,3)}(\alpha)=\mathcal{H}(v)$.
\end{proof}
\subsection{Proof of \Cref{T25}}
We recall that we want to prove that for a given closed, connected surface $\Sigma\subset\mathbb{R}^3$ there exists some $v\in L^2\mathcal{V}_0(\Sigma)$ with $\mathcal{H}(v)\neq 0$ if and only if $g(\Sigma)\geq 1$, where $g(\Sigma)$ denotes the genus of $\Sigma$. In addition, we also need to show that $\mathcal{H}_c(v,w)=0$ for all $w\in L^2\mathcal{V}_0(\Sigma)$ whenever $v$ is a co-exact vector field, where $\mathcal{H}_c$ denotes the cross-helicity and $\mathcal{H}$ denotes the helicity of vector fields.
\begin{proof}[Proof of \Cref{T25}]
	$\quad$
	\newline
	\newline
	\underline{Step 1:} We first prove that if $g(\Sigma)\geq 1$, then there exists some element $v\in L^2\mathcal{V}_0(\Sigma)$ with $\mathcal{H}(v)\neq 0$.
	\newline
	\newline
	We first observe that each closed surface $\Sigma$ bounds a bounded domain $\Omega\subset\mathbb{R}^3$ with $\partial\Omega=\Sigma$. We then have the well-known relation $\dim\left(\mathcal{H}_N(\Omega)\right)=g(\Sigma)$, where $\mathcal{H}_N(\Omega)$ denotes the space of curl-free, div-free, $H^1$-vector fields on $\Omega$ which are tangent to the boundary $\partial\Omega$ and $g(\Sigma)$ denotes the genus of $\Sigma$, \cite[Hodge Decomposition Theorem]{CDG02}. In our situation $g(\Sigma)\geq 1$ and so we can fix some non-zero $\Gamma\in \mathcal{H}_N(\Omega)$. Since $\Gamma$ is tangent to $\Sigma$ we may view it as a vector field on $\Sigma$. The fact that $\Gamma$ is curl-free on $\Omega$ implies that so is its restriction to $\Sigma$ (in the weak sense). We recall that $\mathcal{H}_N(\Omega)\subset \bigcap_{0<\alpha<1}C^{0,\alpha}\mathcal{V}(\overline{\Omega})$ and hence we can write according to the Hodge decomposition theorem $\Gamma|_{\Sigma}=\operatorname{grad}_\Sigma(f)+\gamma$ for suitable $f\in \bigcap_{0<\alpha<1}C^{1,\alpha}(\Sigma)$ and $\gamma\in \mathcal{H}(\Sigma)$ where $\mathcal{H}(\Sigma)$ denotes the space of curl-free and div-free fields on $\Sigma$ of class $\bigcap_{1\leq p<\infty}W^{1,p}$. Since $\Gamma$ is non-zero so must be $\gamma$ (since if the restriction of a harmonic Neumann field to the boundary is a gradient field, then the harmonic Neumann field is already identically zero). We can now let $\mathcal{N}$ denote the outward pointing unit normal and define $\tilde{\gamma}:=\gamma\times \mathcal{N}$ (so that in turn $\gamma=\mathcal{N}\times \tilde{\gamma}$). This new field $\tilde{\gamma}$ is still divergence-free and curl-free (in the language of differential forms the $1$-form associated with $\tilde{\gamma}$ by the musical isomorphism can be obtained from that of $\gamma$ by applying the Hodge-star operator). If either $\mathcal{H}(\gamma)\neq 0$ or $\mathcal{H}(\tilde{\gamma})\neq 0$, then the theorem is proven (since both vector fields are div-free). If on the other hand $\mathcal{H}(\gamma)=0=\mathcal{H}(\tilde{\gamma})$, then we can consider the vector field $v:=\tilde{\gamma}+\gamma$ and notice that, due to the symmetry property of the Biot-Savart operator, we have
	\begin{gather}
		\nonumber
	\mathcal{H}(v)=\mathcal{H}(\gamma)+\mathcal{H}(\tilde{\gamma})+2\mathcal{H}_c(\tilde{\gamma},\gamma)=2\mathcal{H}_c(\tilde{\gamma},\gamma)
	\end{gather}
	where we recall that $\mathcal{H}_c(\tilde{\gamma},\gamma)=\left(\tilde{\gamma},\operatorname{BS}_{\Sigma}(\gamma)\right)_{L^2(\Sigma)}$ is the "cross-helicity". We note that since  $\tilde{\gamma}=\gamma\times \mathcal{N}$ we get by means of the cyclic properties of the Euclidean product
	\begin{gather}
		\nonumber
		4\pi\mathcal{H}_c(\tilde{\gamma},\gamma)=\int_{\Sigma}\int_{\Sigma}(\tilde{\gamma}(x)\times \gamma(y))\cdot \frac{x-y}{|x-y|^3}d\sigma(x)d\sigma(y)
		\\
		\nonumber
		=\int_{\Sigma}\int_{\Sigma}((\gamma(x)\times \mathcal{N}(x))\times \gamma(y))\cdot \frac{x-y}{|x-y|^3}d\sigma(x)d\sigma(y)
		\\
		\nonumber
		=\int_{\Sigma}\int_{\Sigma} [(\gamma(x)\cdot \gamma(y))\mathcal{N}(x)-(\gamma(y)\cdot \mathcal{N}(x))\gamma(x)]\cdot \frac{x-y}{|x-y|^3}d\sigma(y)d\sigma(x)
		\\
		\nonumber
		=\int_{\Sigma}\int_{\Sigma}\gamma(y)\cdot\left[\left(\mathcal{N}(x)\cdot \frac{x-y}{|x-y|^3}\right)\gamma(x)-\left(\gamma(x)\cdot\frac{x-y}{|x-y|^3} \right)\mathcal{N}(x)\right]d\sigma(y)d\sigma(x).
	\end{gather}
	We first claim that
	\begin{gather}
		\nonumber
	\mathbb{R}^3\rightarrow\mathbb{R}\text{, }y\mapsto \int_{\Sigma}\left(\gamma(x)\cdot \frac{x-y}{|x-y|^3}\right)(\mathcal{N}(x)-\tilde{\mathcal{N}}(y))d\sigma(x)
	\end{gather}
	is a continuous function, where $\tilde{\mathcal{N}}$ is any $C^{0,1}$-regular (compactly supported) extension of $\mathcal{N}$ to $\mathbb{R}^3$. But this follows easily from the generalised dominated convergence theorem because we have the estimate
	\begin{gather}
		\nonumber
	\left|\left(\gamma(x)\cdot \frac{x-y}{|x-y|^3}\right)(\mathcal{N}(x)-\tilde{\mathcal{N}}(y))\right|\leq \frac{c}{|x-y|}
	\end{gather}
	for a suitable constant $c>0$ which is independent of $y$ and from the fact that the function $\mathbb{R}^3\rightarrow\mathbb{R}$, $y\mapsto\int_{\Sigma}\frac{1}{|x-y|}d\sigma(x)$ is continuous.
	
	Therefore, if for some fixed $y\in \Sigma$ we let $(y_n)_n\subset \Omega$ be any sequence converging to $y$, we find
	\begin{gather}
		\nonumber
		\gamma(y)\cdot\int_{\Sigma}\left(\gamma(x)\cdot\frac{x-y}{|x-y|^3}\right)\mathcal{N}(x)d\sigma(x)
		\\
		\nonumber
		=\gamma(y)\cdot\int_{\Sigma}\left(\gamma(x)\cdot\frac{x-y}{|x-y|^3}\right)(\mathcal{N}(x)-\mathcal{N}(y))d\sigma(x)
		\\
		\label{E24}
		=\lim_{n\rightarrow\infty}\gamma(y)\cdot\int_{\Sigma}\left(\gamma(x)\cdot\frac{x-y_n}{|x-y_n|^3}\right)(\mathcal{N}(x)-\tilde{\mathcal{N}}(y_n))d\sigma(x)
	\end{gather}
	where we used that $\gamma(y)\cdot \mathcal{N}(y)=0$ because $\gamma$ is a tangent field of $\Sigma$. We observe that since $y_n\notin \Sigma$ and since $\tilde{\mathcal{N}}(y_n)$ is independent of $x$, the term involving $\tilde{\mathcal{N}}(y_n)$ vanishes because
	\begin{gather}
		\nonumber
	\int_{\Sigma}\gamma(x)\cdot \frac{x-y_n}{|x-y_n|^3}d\sigma(x)=-\int_{\Sigma}\gamma(x)\cdot \nabla^{\Sigma}_x\left(\frac{1}{|x-y_n|}\right)d\sigma(x)=0\text{ for all }n
	\end{gather}
	where we used that $\gamma$ is divergence-free and that $x\mapsto \frac{1}{|x-y_n|}$ defines a $C^1$-function on $\Sigma$ for all fixed $n$ because each $y_n$ has a positive distance to $\Sigma$. We conclude
	\begin{gather}
		\nonumber
	\gamma(y)\cdot\int_{\Sigma}\left(\gamma(x)\cdot\frac{x-y}{|x-y|^3}\right)\mathcal{N}(x)d\sigma(x)=\lim_{n\rightarrow\infty}\gamma(y)\cdot\int_{\Sigma}\left(\gamma(x)\cdot\frac{x-y_n}{|x-y_n|^3}\right)\mathcal{N}(x)d\sigma(x).
	\end{gather}
	On the other hand we have the well-known jump formula, \cite[Theorem 4.30]{RCM21} which tells us that we can find a suitable sequence $(y_n)_n\subset \Omega$ converging to $y$ such that
	\begin{gather}
		\nonumber
	\int_{\Sigma}\left(\mathcal{N}(x)\cdot \frac{x-y}{|x-y|^3}\right)\gamma(x)d\sigma(x)=-\frac{4\pi}{2}\gamma(y)+\lim_{n\rightarrow\infty}\int_{\Sigma}\left(\mathcal{N}(x)\cdot \frac{x-y_n}{|x-y_n|^3}\right)\gamma(x)d\sigma(x).
	\end{gather}
	Combining our considerations so far, we see that for all $y\in \Sigma$ there exists some sequence $(y_n)_n\subset \Omega$ converging to $y$ such that
	\begin{gather}
		\nonumber
		4\pi\mathcal{H}_c(\gamma,\tilde{\gamma})=-\frac{4\pi}{2}\|\gamma\|^2_{L^2(\Sigma)}
		\\
		\nonumber
		+\int_{\Sigma}\lim_{n\rightarrow\infty}\gamma(y)\cdot \left(\int_{\Sigma}\left(\mathcal{N}(x)\cdot \frac{x-y_n}{|x-y_n|^3}\right)\gamma(x)d\sigma(x)-\int_{\Sigma}\left(\gamma(x)\cdot\frac{x-y_n}{|x-y_n|^3}\right)\mathcal{N}(x)d\sigma(x)\right)d\sigma(y).
	\end{gather}
	We observe now that $\gamma$ admits a curl- and div-free extension to $\Omega$, i.e. there exists some $\tilde{v}\in \bigcap_{1\leq p<\infty}W^{1,p}\mathcal{V}(\Omega)$ with $\operatorname{curl}(\tilde{v})=0$, $\operatorname{div}(\tilde{v})=0$ in $\Omega$ and $\tilde{v}^\parallel=\gamma$, where $\tilde{v}^\parallel:=\tilde{v}-(\mathcal{N}\cdot \tilde{v})\mathcal{N}$. Indeed, we recall that $\Gamma|_{\Sigma}=\operatorname{grad}_{\Sigma}(f)+\gamma$ for a suitable function $f\in \bigcap_{0<\alpha<1}C^{1,\alpha}(\Sigma)$, and thus if we let $\tilde{f}$ denote the harmonic extension of $f$, c.f. \cite[Theorem 2.4.2.5]{Gris85}, then $\Gamma-\operatorname{grad}(\tilde{f})\in \bigcap_{1\leq p<\infty}W^{1,p}\mathcal{V}(\Omega)$ is the desired curl- and div-free extension. One can then use Gauss' theorem and the fact that $y_n\in \Omega$ for fixed $n$, to obtain
	\begin{gather}
		\nonumber
		\int_{\Sigma}\left(\mathcal{N}(x)\cdot \frac{x-y_n}{|x-y_n|^3}\right)\gamma(x)d\sigma(x)-\int_{\Sigma}\left(\gamma(x)\cdot\frac{x-y_n}{|x-y_n|^3}\right)\mathcal{N}(x)d\sigma(x)
		\\
		\nonumber
		=\int_{\Sigma}\left(\mathcal{N}(x)\cdot \frac{x-y_n}{|x-y_n|^3}\right)\tilde{v}(x)d\sigma(x)-\int_{\Sigma}\left(\tilde{v}(x)\cdot\frac{x-y_n}{|x-y_n|^3}\right)\mathcal{N}(x)d\sigma(x)
		\\
		\nonumber
		=4\pi \tilde{v}(y_n)+\int_{\Omega}\operatorname{div}(\tilde{v})(x)\frac{x-y_n}{|x-y_n|^3}d^3x+\int_{\Omega}\frac{y_n-x}{|y_n-x|^3}\times \operatorname{curl}(\tilde{v})(x)d^3x=4\pi\tilde{v}(y_n),
	\end{gather}
	where the first identity follows from the fact that the terms involving the normal part of $\tilde{v}$ cancel each other. Taking the limit $n\rightarrow\infty$ we conclude
	\begin{gather}
		\nonumber
	4\pi\mathcal{H}_c(\gamma,\tilde{\gamma})=\frac{4\pi}{2}\|\gamma\|^2_{L^2(\Sigma)}.
	\end{gather}
	We recall that $\mathcal{H}(\gamma+\tilde{\gamma})=2\mathcal{H}_c(\gamma,\tilde{\gamma})$ and therefore we arrive at
	\begin{gather}
		\nonumber
	\mathcal{H}(\gamma+\tilde{\gamma})=2\mathcal{H}_c(\gamma,\tilde{\gamma})=\|\gamma\|^2_{L^2(\Sigma)}>0
	\end{gather}
	because $\gamma\not\equiv 0$. Hence, in any case, there must exist some $\hat{\gamma}\in \mathcal{H}(\Sigma)$ satisfying $\mathcal{H}(\hat{\gamma})\neq 0$.
	\newline
	\newline
	\underline{Step 2:} In this step we prove that if $w\in L^2\mathcal{V}_0(\Sigma)$ and $f\in H^1(\Sigma)$, then $\mathcal{H}_c(\nabla^\perp f,w)=0$, where $\nabla^\perp f:=\nabla_{\Sigma}f\times \mathcal{N}$. In particular, as $\mathcal{H}(\nabla^\perp f)=\mathcal{H}_c(\nabla^\perp f,\nabla^\perp f)$ we will get $\mathcal{H}(\nabla^\perp f)=0$. Since on a genus zero surface all divergence-free fields are co-exact this will conclude the proof.
	\newline
	\newline
	First we observe that by a density argument we may without loss of generality assume that $w,\nabla_{\Sigma}f\in \bigcap_{1\leq p<\infty}W^{1,p}\mathcal{V}(\Sigma)$. We follow the arguments of the first step. We simply replace $\gamma(y)$ by $w(y)$ and $\gamma(x)$ by $\nabla_\Sigma f(x)$ in the arguments. The only step where the arguments need to be modified is in (\ref{E24}). More specifically, the argument that $\gamma(y)\cdot \tilde{\mathcal{N}}(y_n) \int_{\Sigma}\left(\gamma(x)\cdot \frac{x-y_n}{|x-y_n|^3}\right)d\sigma(x)$ vanishes no longer applies because in our new situation $\gamma(x)$ has to be replaced by $\nabla_{\Sigma} f(x)$ which is no longer divergence-free. However, we see that
	\begin{gather}
		\nonumber
		\int_{\Sigma}\left(\nabla_{\Sigma}f(x)\cdot \frac{x-y_n}{|x-y_n|^3}\right)d\sigma(x)=-\int_{\Sigma}\nabla_{\Sigma} f(x)\cdot \nabla^{\Sigma}_x\frac{1}{|x-y_n|}d\sigma(x)=\int_{\Sigma}\frac{\Delta_{\Sigma}f(x)}{|x-y_n|}d\sigma(x)
	\end{gather}
	where $\Delta_{\Sigma}:=\operatorname{div}_{\Sigma}\circ \nabla_{\Sigma}$ denotes the Laplace operator on $\Sigma$. We note that the function $\mathbb{R}^3\rightarrow\mathbb{R}$, $y\mapsto\int_{\Sigma}\frac{h(x)}{|x-y|}d\sigma(x)$ is continuous for any function $h\in L^p(\Sigma)$, for some $p>2$, and therefore
	\begin{gather}
		\nonumber
	\lim_{n\rightarrow\infty}w(y)\cdot \tilde{\mathcal{N}}(y_n)\int_{\Sigma}\left(\nabla_{\Sigma} f(x)\cdot \frac{x-y_n}{|x-y_n|^3}\right)d\sigma(x)=w(y)\cdot \mathcal{N}(y)\int_{\Sigma}\frac{\Delta_{\Sigma} f(x)}{|x-y|}d\sigma(x)=0
	\end{gather}
	because $\tilde{\mathcal{N}}(y)=\mathcal{N}(y)$ (since $\tilde{\mathcal{N}}$ is an extension of $\mathcal{N}$) and because $w(y)$ is tangent to $\Sigma$. The remaining arguments apply verbatim which leads us to
	\begin{gather}
		\nonumber
	2\mathcal{H}_c(\nabla^\perp f,w)=\left( \nabla_{\Sigma} f,w\right)_{L^2(\Sigma)}=0
	\end{gather}
	because $w$ is divergence-free.
\end{proof}
\subsection{Proof of \Cref{T26}}
\begin{proof}[Proof of\Cref{T26}]
	We recall that $v\in C^{0,1}\mathcal{V}_0(\Sigma)$ is a div-free Lipschitz vector field on $\Sigma$. Further, we denote by $\gamma_x$ the (unique) integral curve of $v$ starting at $x$ and for given $T>0$ we denote by $\gamma_x[0,T]$ the path along the field line $\gamma_x$ from time $t=0$ until $t=T$.
	\newline
	\newline
	\underline{Step 1:} In the first step we prove that for every $T,S\in (0,\infty)$ the two path $\gamma_x[0,T]$ and $\gamma_y[0,S]$ are disjoint almost surely (and hence by construction the artificially closed path $\sigma^{\tau}_{x,T}$ and $\sigma^{-\tau}_{y,S}$ will be also almost surely disjoint and have a well-defined linking number).
	
	For given $T,S>0$ we consider the set $B_{T,S}:=\{(x,y)\in \Sigma\times \Sigma|\gamma_x[0,T]\cap \gamma_y[0,S]\neq\emptyset\}$. Now we note that $(x,y)\in B_{T,S}\Leftrightarrow\exists z\in\gamma_x[0,T]\cap \gamma_y[0,S]$.	
	
	In conclusion we must have $\gamma_x(\tau)=z=\gamma_y(\lambda)\Leftrightarrow y=\gamma_x(\tau-\lambda)$ for suitable $0\leq \tau\leq T$, $0\leq \lambda\leq S$, where we used the properties of the flow. Hence, for any $T,S>0$ we have
	\begin{gather}
		\nonumber
	B_{T,S}\subset \bigcup_{n\in \mathbb{N}}\{(x,y)\in \Sigma\times \Sigma|y\in \gamma_x(-n,n)\}.
	\end{gather}
	We observe that the right hand side is independent of $T$ and $S$ and that by sigma-subadditivity it is enough to prove that $\{(x,y)\in \Sigma\times \Sigma|y\in \gamma_x(-n,n)\}$ is a null-set in order to show that for almost every $(x,y)\in \Sigma\times \Sigma$ the sets $\gamma_x[0,T]$ and $\gamma_y[0,S]$ do not intersect (in fact this null set is independent of the choice of $T$ and $S$).
	
	Denoting by $\mu$ the Riemannian measure on $\Sigma$ we find, using Fubini's theorem,
	\begin{gather}
		\nonumber
	(\mu\times \mu)\left(\{(x,y)\in \Sigma\times \Sigma|y\in \gamma_x(-n,n)\}\right)=\int_{\Sigma}\int_{\Sigma}\chi_{\{(x,y)\in \Sigma\times \Sigma|y\in \gamma_x(-n,n)\}}d\sigma(y)d\sigma(x).
	\end{gather}
	We observe that
	\begin{gather}
		\nonumber
	\chi_{\{(x,y)\in \Sigma\times \Sigma|y\in \gamma_x(-n,n)\}}(x,y)=\chi_{B_x}(y)
	\end{gather}
	where $B_x:=\{y\in \Sigma|y\in \gamma_x(-n,n)\}$ so that
	\begin{gather}
		\nonumber
	\int_{\Sigma}\chi_{B_x}(y)d\sigma(y)=\mu(B_x)=\mu\left(\gamma_x(-n,n)\right)=0
	\end{gather}
	where we used in the last step that $\gamma_x(-n,n)$ is either a point or otherwise a local embedding of a $1$-manifold and hence has Hausdorff-dimension at most $1$.
	\newline
	\newline
	\underline{Step 2:} In this step we prove the remaining claims of the theorem.
	\newline
	\newline
		In order to specify our choice of $\tau(\Sigma)>0$, we observe first that there exists some $c=c(\Sigma)>0$ such that, c.f. \cite[Lemma 41.12]{Ser17},
		\begin{equation}
			\label{EX3}
			|\mathcal{N}(x)\cdot (x-y)|\leq c|x-y|^2\text{ for all }x,y\in \Sigma.
		\end{equation}
		Further, we note that the function
		\begin{gather}
			\nonumber
		\kappa:\Sigma\times \Sigma\rightarrow \mathbb{R}\text{, }(x,y)\mapsto \mathcal{N}(x)\cdot \mathcal{N}(y)
		\end{gather}
		is uniformly continuous and thus
		\begin{equation}
			\label{EX4}
			\mathcal{N}(x)\cdot \mathcal{N}(y)\geq \frac{1}{2}\text{ for all }x,y\in \Sigma, |x-y|<\delta(\Sigma)
		\end{equation}
		for a suitable $\delta(\Sigma)>0$. We can then choose $\tau(\Sigma)>0$ such that $\tau(\Sigma)\leq \min\{\frac{\delta}{3},\frac{1}{6c}\}$ and such that $\Psi_{t}(x):=x+t\mathcal{N}(x)$ is a diffeomorphism from $\Sigma$ onto its image for every $|t|\leq \tau$.
		
		In order to obtain the desired linking interpretation we observe first that, c.f. \cite[Theorem 2]{TME67}:
		
		For every $f\in L^1(\Sigma\times \Sigma)$, the sequence of functions $\frac{\int_0^{S_n}\int_0^{T_n}f(\gamma_x(t),\gamma_y(s))dtds}{T_nS_n}$ admits a well-defined $L^1$-limit denoted by $\hat{f}\in L^1(\Sigma\times \Sigma)$ and \[
		\int_{\Sigma\times \Sigma}f(x,y)d\sigma(x)d\sigma(y)=\int_{\Sigma\times\Sigma}\hat{f}(x,y)d\sigma(x)d\sigma(y).
		\]
		We define $f(x,y):=\frac{1}{4\pi}(v(x)\times v(y))\cdot \frac{x-y}{|x-y|^3}\in L^1(\Sigma\times \Sigma)$, which can be seen by writing $v(x)\times v(y)=v(x)\times (v(y)-v(x))$, using the Lipschitz property of $v$ and the fact that $\frac{1}{|x-y|}$ is integrable over $\Sigma\times \Sigma$. We therefore arrive at
		\begin{equation}
			\label{EX5}
			\mathcal{H}(v)=\frac{1}{4\pi}\int_{\Sigma\times \Sigma}\lim^{L^1}_{n\rightarrow\infty}\frac{\int_0^{T_n}\int_0^{S_n}(\dot{\gamma}_x(t)\times \dot{\gamma}_y(s))\cdot \frac{\gamma_x(t)-\gamma_y(s)}{|\gamma_x(t)-\gamma_y(s)|^3}dtds}{T_nS_n}d\sigma(x)d\sigma(y).
		\end{equation}
		We recall the linking integral for two piecewise differentiable closed curves $\sigma_1,\sigma_2$ with $\sigma_i(0)=\sigma_i(T_i)$
		\begin{equation}
			\label{EX6}
			\operatorname{lk}(\sigma_1,\sigma_2)=\frac{1}{4\pi}\int_0^{T_1}\int_0^{T_2}(\dot{\sigma}_1(t)\times \dot{\sigma}_2(s))\cdot \frac{\sigma_1(t)-\sigma_2(s)}{|\sigma_1(t)-\sigma_2(s)|^3}dtds.
		\end{equation}
		Now, as discussed before the statement of \Cref{T26}, we close the curves $\gamma_x([0,T_n])$ and $\gamma_y([0,S_n])$ and obtain closed curves $\sigma^{\tau}_{x,T_n}$ and $\sigma^{-\tau}_{y,S_n}$. For notational simplicity we drop the index $n$ in $T_n$ and $S_n$. According to (\ref{EX5}) and (\ref{EX6}) we have to show that the additional terms involved in the expression of $\operatorname{lk}\left(\sigma^{\tau}_{x,T},\sigma^{-\tau}_{y,S}\right)$ converge to zero in the $L^1(\Sigma\times \Sigma)$-sense.
		
		To this end we first let $a\in \Sigma\setminus \gamma_x(\mathbb{R})$, $0\leq s\leq \tau$ be any fixed elements and we observe that
		\begin{gather}
			\nonumber
		|a-s\mathcal{N}(a)-\gamma_x(t)|^2=|a-\gamma_x(t)|^2+s^2-2s(a-\gamma_x(t))\cdot \mathcal{N}(a).
		\end{gather}
		Now, since $a,\gamma_x(t)\in \Sigma$ we can use the estimate $|(a-\gamma_x(t))\cdot \mathcal{N}(a)|\leq c|a-\gamma_x(t)|^2$ so that
		\begin{equation}
			\nonumber
			|a-s\mathcal{N}(a)-\gamma_x(t)|^2\geq |a-\gamma_x(t)|^2+s^2-2cs|a-\gamma_x(t)|^2
			\\
			\nonumber
			\geq |a-\gamma_x(t)|^2+\frac{s^2}{2}-2c^2|a-\gamma_x(t)|^4
		\end{equation}
		where we used the elementary inequality $2db\leq \frac{d^2}{\epsilon}+\epsilon b^2$ for all $d,b\in \mathbb{R}$, $\epsilon>0$ with $d=s$, $b=c|a-\gamma_x(t)|^2$ and $\epsilon=2$. Now we distinguish two cases: Either $2c^2|a-\gamma_x(t)|^2\leq \frac{1}{2}\Leftrightarrow |a-\gamma_x(t)|^2\leq \frac{1}{4c^2}$, in which case we obtain
		\begin{gather}
			\nonumber
		|a-s\mathcal{N}(a)-\gamma_x(t)|^2\geq \frac{|a-\gamma_x(t)|^2+s^2}{2}\geq \frac{|a-\gamma_x(t)|^2+s^2}{4},
		\end{gather}
		otherwise we find $|a-\gamma_x(t)|^2\geq \frac{1}{4c^2}$. But in this case we can use the estimate $|2s(a-\gamma_x(t))\cdot \mathcal{N}(a)|\leq 2s^2+\frac{|a-\gamma_x(t)|^2}{2}$ and therefore in this case
		\begin{gather}
			\nonumber
		|a-s\mathcal{N}(a)-\gamma_x(t)|^2\geq \frac{|a-\gamma_x(t)|^2}{2}-s^2\geq \frac{|a-\gamma_x(t)|^2}{4}+\frac{1}{16c^2}-s^2.
		\end{gather}
		To estimate the last term we recall that we are interested only in the situation $0\leq s\leq \tau$ and thus we can use the fact that $\tau\leq \frac{1}{6c}$, i.e. $\frac{1}{36c^2}\geq \tau^2$ and therefore $\frac{1}{16c^2}\geq 2\tau^2\geq 2s^2$ which yields
		\begin{gather}
			\nonumber
		|a-s\mathcal{N}(a)-\gamma_x(t)|^2\geq\frac{|a-\gamma_x(t)|^2+s^2}{4}
		\end{gather}
		in this case. We conclude that in any case we can estimate a part of the linking integral as follows:
		\begin{gather}
			\nonumber
		\left|\int_0^{T}\int_0^{\tau}\left(v(\gamma_x(t))\times\mathcal{N}(a)\right)\cdot \frac{a-s\mathcal{N}(a)-\gamma_x(t)}{|a-s\mathcal{N}(a)-\gamma_x(t)|^3}dsdt\right|\leq 4c_v\int_0^T\int_0^\tau\frac{1}{|a-\gamma_x(t)|^2+s^2}dsdt
		\end{gather}
		for a suitable constant $c_v>0$ which is given by the (finite) $C^0(\Sigma)$ norm of $v$. We can further explicitly compute and estimate
		\begin{gather}
			\nonumber
		\int_0^\tau\frac{1}{|a-\gamma_x(t)|^2+s^2}ds=\frac{\arctan\left(\frac{\tau}{|a-\gamma_x(t)|}\right)}{|a-\gamma_x(t)|}\leq \frac{\pi}{2|a-\gamma_x(t)|}.
		\end{gather}
		We recall that in our application we want to connect either the point $y$ or $\gamma_y(S)$ with the corresponding points on $\Sigma_{-\tau}$ via a straight line following the normal direction of the starting point. We note that $\gamma_y(S)=\psi_S(y)$ where $\psi$ denotes the flow of $v$ and because $v$ is divergence-free, the corresponding flow is area-preserving. Hence we want to consider the situation where $a=\psi(y)$ for a suitable area preserving diffeomorphism of $\Sigma$. But then we can estimate
		\begin{gather}
			\nonumber
			\int_{\Sigma\times\Sigma}\left|\int_0^T\int_0^\tau\left(v(\gamma_x(t))\times \mathcal{N}(\psi(y))\right)\cdot\frac{\psi(y)-s\mathcal{N}(\psi(y))-\gamma_x(t)}{|\psi(y)-s\mathcal{N}(\psi(y))-\gamma_x(t)|^3}dsdt\right|d\sigma(x)d\sigma(y)
			\\
			\nonumber
			\leq 2\pi c_v\int_{\Sigma\times \Sigma}\int_0^T\frac{1}{|\psi(y)-\gamma_x(t)|}dtd\sigma(x)d\sigma(y)=2\pi c_vT\int_{\Sigma\times \Sigma}\frac{1}{|x-y|}d\sigma(x)d\sigma(y)
		\end{gather}
		where we used Fubini's theorem in the last step and the fact that $\psi$ as well as $\gamma_x(t)=\psi_t(x)$ are area preserving diffeomorphisms for every $t$. Therefore, if we divide the above expression by $T$ and $S$ and take the limit $T,S\rightarrow\infty$ we see that the corresponding part of the linking integral converges to zero in the $L^1$-sense.
		
		Now we connect $\Psi_{-\tau}(y)$ and $\Psi_{-\tau}(\gamma_y(S))$ on the surface $\Sigma_{-\tau}$ by a curve $\gamma$ which is of unit speed and whose length is uniformly bounded independently of $y$ and $S$, for instance we can connect them by a length minimising geodesic on $\Sigma_{-\tau}$ in which case its length is always bounded by the intrinsic diameter of $\Sigma_{-\tau}$. In that case we can use the fact that $\Sigma_{-\tau}$ has a positive distance to $\Sigma$ and therefore $|\gamma_x(t)-\gamma(s)|\geq d(-\tau)>0$ for a suitable $d$ which is independent of $x,t,y,S$ and $s$. Consequently one can use the rough, pointwise, upper bound
		\begin{gather}
			\nonumber
		\left|v(\gamma_x(t))\times \dot{\gamma}(s)\cdot \frac{\gamma_x(t)-\gamma(s)}{|\gamma_x(t)-\gamma(s)|^3}\right|\leq \frac{c_v}{|\gamma_x(t)-\gamma(s)|^2}\leq 
		\frac{c_v}{d^2}
		\end{gather}
		in order to see that the corresponding part of the linking integral converges to zero in the $L^1$-sense.
		The only remaining part of the linking integral which is not covered by the considerations made so far is of the following type
		\begin{gather}
			\nonumber
		\int_0^{\tau}\int_0^\tau\mathcal{N}(b)\times \mathcal{N}(a)\cdot \frac{b+t\mathcal{N}(b)-a+s\mathcal{N}(a)}{|b+t\mathcal{N}(b)-a+s\mathcal{N}(a)|^3}dsdt
		\end{gather}
		where in our application $a\in \{y,\gamma_y(S)\}$ and $b\in \{x,\gamma_x(T)\}$. But here we can argue similar in spirit as in the first case, namely we can consider
		\begin{gather}
			\nonumber
			|a-s\mathcal{N}(a)-(b+t\mathcal{N}(b))|^2=|(b-a)+t\mathcal{N}(b)+s\mathcal{N}(a)|^2
			\\
			\nonumber
			=|b-a|^2+|s\mathcal{N}(a)+t\mathcal{N}(b)|^2+2s\mathcal{N}(a)\cdot (b-a)+2t\mathcal{N}(b)\cdot (b-a).
		\end{gather}
		We note that, since $a,b\in \Sigma$ and $0\leq t\leq \tau$, $|2t\mathcal{N}(b)\cdot (b-a)|\leq 2c\tau |b-a|^2$. We recall that $\tau \leq \frac{1}{6c}$ and thus $|2t\mathcal{N}(b)\cdot (b-a)|\leq \frac{|b-a|^2}{3}$ and so we arrive at
		\begin{gather}
			\nonumber
			|a-s\mathcal{N}(a)-(b+t\mathcal{N}(b))|^2\geq \frac{|b-a|^2}{3}+|s\mathcal{N}(a)+t\mathcal{N}(b)|^2=\frac{|b-a|^2}{3}+t^2+s^2+2st\mathcal{N}(a)\cdot \mathcal{N}(b).
		\end{gather}
		We recall that if $|a-b|\leq \delta(\Sigma)$ then $\mathcal{N}(a)\cdot \mathcal{N}(b)\geq \frac{1}{2}$ and hence, since $s,t\geq 0$ we arrive in this case at
		\begin{gather}
			\nonumber
			|a-s\mathcal{N}(a)-(b+t\mathcal{N}(b))|^2\geq \frac{|a-b|^2}{3}+t^2+s^2\geq \frac{|a-b|^2+t^2+s^2}{6}\geq \frac{|a-b|^2+t^2}{6}.
		\end{gather}
		On the other hand, if $|a-b|\geq \delta$, we can estimate $2st\leq \frac{t^2}{2}+2s^2$ and thus
		\begin{gather}
			\nonumber
			|a-s\mathcal{N}(a)-(b+t\mathcal{N}(b))|^2\geq \frac{|b-a|^2}{3}+\frac{t^2}{2}-s^2.
		\end{gather}
		Further, $\frac{|b-a|^2}{6}\geq \frac{\delta^2}{6}\geq \tau^2\geq s^2$ for all $0\leq s\leq \tau$ because by choice of $\tau$ we have $\tau \leq \frac{\delta}{3}$. Consequently we arrive at
		\begin{gather}
			\nonumber
			|a-s\mathcal{N}(a)-(b+t\mathcal{N}(b))|^2\geq \frac{|b-a|^2+t^2}{6}\text{ for all }a,b\in \Sigma.
		\end{gather}
		From here on out we can argue identically as in the first part of the proof that the corresponding part of the linking integral converges to zero in the $L^1$-sense. Hence, combining our findings with (\ref{EX5}) and the definition of the linking integral (\ref{EX6}) we conclude the validity of the theorem.
\end{proof}
\subsection{Proof of \Cref{T213}}
We start by proving \Cref{L212}.
\begin{proof}[Proof of \Cref{L212}]
	We will prove the statement only for $\overline{Q}(v)$, since an identical reasoning applies to $\overline{P}(v)$. We recall first that for given $v\in C^{0,1}\mathcal{V}_0(\Sigma)$ we denote its integral curves starting at a point $x\in\Sigma$ by $\gamma_x$. Further we have by definition $\hat{q}(x)=\lim_{T\rightarrow\infty}\frac{1}{T}\int_{\gamma_x[0,T]}\gamma_t$ where $\gamma_t\in \mathcal{H}(\Sigma)$ is a harmonic field on $\Sigma$ defined by the relations $\int_{\sigma_p}\gamma_t=0$ and $\int_{\sigma_t}\gamma_t=1$, where $\sigma_p$ and $\sigma_t$ are some fixed purely poloidal and a toroidal closed curves respectively. Finally, if $\hat{q}\in L^1(\Sigma)$, we had set $\overline{Q}(v):=\frac{\int_{\Sigma}\hat{q}(x)d\sigma(x)}{|\Sigma|}$.
	\newline
	\newline
	To see that $\hat{q}\in L^1(\Sigma)$ and $\overline{Q}(v)=\frac{\int_{\Sigma}v(x)\cdot \gamma_t(x)d\sigma(x)}{|\Sigma|}$ we spell out the definition of the line integral, which leads us to
	\begin{gather}
		\nonumber
		\int_{\gamma_x[0,T]}\gamma_t=\int_0^T \dot{\gamma}_x(\tau)\cdot \gamma_t(\gamma_x(\tau))d\tau=\int_0^Tv(\gamma_x(\tau))\cdot \gamma_t(\gamma_x(\tau))d\tau.
	\end{gather}
	Now we define the function $f:\Sigma\rightarrow\mathbb{R}$, $f(x):=v(x)\cdot \gamma_t(x)\in L^1(\Sigma)$ (recall that $\gamma_t\in L^2\mathcal{V}(\Sigma)$). We note that $\gamma_x(\tau)=\psi_{\tau}(x)$ where $\psi_{\tau}$ denotes the area-preserving flow of $v$. Hence, standard ergodic theoretical results, c.f. \cite[Theorem 2]{TME67}, imply that $\hat{q}\in L^1(\Sigma)$ and that $\overline{Q}(v)=\frac{\int_{\Sigma}f(x)d\sigma(x)}{|\Sigma|}=\frac{\int_{\Sigma}v(x)\cdot \gamma_t(x)d\sigma(x)}{|\Sigma|}$ as was to be shown.
\end{proof}
Before we come to the proof of \Cref{T213} we will need one additional lemma. To this end we recall that $\mathcal{H}_N(\Omega)$ denotes the space of curl-free, div-free, vector fields tangent to the boundary of a given domain $\Omega$.
\begin{lem}
	\label{L31}
	Let $\Sigma\subset\mathbb{R}^3$ be a closed, connected $C^{1,1}$-surface and let $\Omega\subset\mathbb{R}^3$ be the bounded domain enclosed by $\Sigma$, $\partial\Omega=\Sigma$. Given any $\Gamma\in \mathcal{H}_N(\Omega)$ we let $\gamma\in \mathcal{H}(\Sigma)$ denote the $L^2(\Sigma)$-orthogonal projection of $\Gamma|_{\Sigma}$ onto the space of harmonic fields $\mathcal{H}(\Sigma)$. Then $\mathcal{H}(\gamma\times \mathcal{N})=0$, where $\mathcal{N}$ denotes the outward pointing unit normal.
\end{lem}
\begin{proof}[Proof of \Cref{L31}]
	We follow the proof of \Cref{T25}. Let us set $\tilde{\gamma}:=\gamma\times \mathcal{N}$. Then
	\begin{gather}
		\nonumber
		4\pi\mathcal{H}(\tilde{\gamma})=\int_{\Sigma}\int_{\Sigma}\tilde{\gamma}(x)\cdot \left(\tilde{\gamma}(y)\times \frac{x-y}{|x-y|^3}\right)d\sigma(x)d\sigma(y)
		\\
		\nonumber
		=\int_{\Sigma}\int_{\Sigma}\left(\tilde{\gamma}(x)\times \tilde{\gamma}(y)\right)\cdot\frac{x-y}{|x-y|^3}d\sigma(x)d\sigma(y)=\int_{\Sigma}\int_{\Sigma}\left((\gamma(x)\times \mathcal{N}(x))\times \tilde{\gamma}(y)\right)\cdot\frac{x-y}{|x-y|^3}d\sigma(y)d\sigma(x).
	\end{gather}
	From here on all the arguments in the computations of $4\pi\mathcal{H}_c(\gamma,\tilde{\gamma})$ apply verbatim, with the only caveat that $\gamma(y)$ has to be replaced by $\tilde{\gamma}(y)$ in the appropriate places during the computations. This leads us to the identity $2\mathcal{H}(\tilde{\gamma})=\int_{\Sigma}\tilde{\gamma}(y)\cdot \gamma(y)d\sigma(y)=0$ because $\tilde{\gamma}$ and $\gamma$ are pointwise everywhere orthogonal to each other.
\end{proof}
For potential future reference we also state the following immediate consequence.
\begin{cor}
	\label{C32}
	Let $\Sigma\subset\mathbb{R}^3$ be a closed, connected $C^{1,1}$-surface and let $\Omega\subset\mathbb{R}^3$ be the bounded domain enclosed by $\Sigma$, $\partial\Omega=\Sigma$. Given any $\Gamma\in \mathcal{H}_N(\Omega)$ we have $\Gamma\times\mathcal{N}\in L^2\mathcal{V}_0(\Sigma)$ and $\mathcal{H}(\Gamma\times\mathcal{N})=0$.
\end{cor}
\begin{proof}[Proof of \Cref{C32}]
	As mentioned in the proof of \Cref{T25} we can decompose $\Gamma|_{\Sigma}$ as $\Gamma|_{\Sigma}=\nabla_{\Sigma}f+\gamma$ for a suitable function $f\in\bigcap_{0<\alpha<1}C^{1,\alpha}(\Sigma)$ and $\gamma\in \mathcal{H}(\Sigma)$. Since taking the cross product with $\mathcal{N}$ corresponds, in the language of differential forms, to applying the Hodge star operator, it is immediate that $\Gamma\times \mathcal{N}\in L^2\mathcal{V}_0(\Sigma)$. In addition, $\Gamma\times \mathcal{N}$ and $\gamma\times \mathcal{N}$ differ only by a co-exact vector field, so that the corollary follows from \Cref{L31} and \Cref{T25}.
\end{proof}
\begin{proof}[Proof of \Cref{T213}]
	Since by assumption $T^2\cong \Sigma$ we know from standard Hodge-theory that $\dim(\mathcal{H}(\Sigma))=2$. Further, if we fix any $\Gamma\in \mathcal{H}_N(\Omega)\setminus\{0\}$ where $\Omega\subset\mathbb{R}^3$ is the solid torus enclosed by $\Sigma$, we see that its $L^2$-orthogonal projection $\gamma$ onto $\mathcal{H}(\Sigma)$ on $\Sigma$ defines a non-zero element of $\mathcal{H}(\Sigma)$. Setting $\tilde{\gamma}:=\gamma\times \mathcal{N}$ we note that $\tilde{\gamma}\in \mathcal{H}(\Sigma)$ because taking the cross product with the outer unit-normal corresponds to applying the Hodge star operator to the corresponding $1$-forms (obtained from the musical isomorphism). Now, if $v\in C^{0,1}\mathcal{V}_0(\Sigma)$ is an arbitrary vector field, we can decompose $v$ by means of the Hodge decomposition theorem as $v=\nabla_{\Sigma}f\times \mathcal{N}+\alpha \gamma+\beta\tilde{\gamma}$ for a suitable $C^1$-function $f\in C^1(\Sigma)$ and suitable constants $\alpha,\beta\in \mathbb{R}$. We observe that $\gamma$ and $\tilde{\gamma}$ have the same $L^2$-norms, so that $\tilde{\gamma}$ is $L^2$-normalised if $\gamma$ is. Assuming that $\gamma$ is $L^2$-normalised and keeping in mind that the Hodge decomposition is $L^2$-orthogonal we find
	\begin{gather}
		\label{ExtraEquation1}
		\alpha=\int_{\Sigma}\gamma\cdot vd\sigma(x)\text{, }\beta=\int_{\Sigma}\tilde{\gamma}(x)\cdot v(x)d\sigma(x).
	\end{gather}
	Utilising \Cref{T25} we conclude
	\begin{gather}
		\nonumber
		\mathcal{H}(v)=\mathcal{H}(\alpha\gamma+\beta\tilde{\gamma})=\alpha^2\mathcal{H}(\gamma)+2\alpha\beta\mathcal{H}_c(\gamma,\tilde{\gamma})+\beta^2\mathcal{H}(\tilde{\gamma}).
	\end{gather}
	According to \Cref{L31} we have $\mathcal{H}(\tilde{\gamma})=0$. In addition, we have shown in the proof of \Cref{T25} that $2\mathcal{H}_c(\gamma,\tilde{\gamma})=\|\gamma\|^2_{L^2(\Sigma)}=1$ due to the normalisation of $\gamma$. We arrive at
	\begin{gather}
		\label{ExtraEquation2}
		\mathcal{H}(v)=\alpha^2\mathcal{H}(\gamma)+\alpha\beta.
	\end{gather}
	Now we recall that we are given a purely poloidal curve $\sigma_p$ and a toroidal curve $\sigma_t$ which uniquely determine elements $\gamma_p,\gamma_t\in \mathcal{H}(\Sigma)$ according to the relations $\int_{\sigma_p}\gamma_t=0=\int_{\sigma_t}\gamma_p$ and $\int_{\sigma_p}\gamma_p=1=\int_{\sigma_t}\gamma_t$. We claim that $\gamma_t=\frac{\gamma}{\int_{\sigma_t}\gamma}$. To see this we use the fact that $\Gamma|_{\Sigma}-\gamma$ is the gradient of a $C^1$-function so that $\int_{\sigma_p}\gamma=\int_{\sigma_p}\Gamma$. Further, we may assume that the $C^1$-curve $\sigma_p$, being purely poloidal, bounds a $C^1$ disc $D\subset \Omega$, $\partial D=\sigma_p$. The idea now is to apply Stokes' theorem to deduce $\int_{\sigma_p}\Gamma=\int_D\operatorname{curl}(\Gamma)\cdot n d\sigma=0$ where $n$ denotes the corresponding normal to $D$ and where we used that $\Gamma$ is curl-free. However, due to the boundary regularity, we only know that $\Gamma\in \bigcap_{1\leq p<\infty}W^{1,p}\mathcal{V}(\Omega)$ so that the classical Stokes' theorem is not immediately applicable. To bypass this problem one can exploit the fact that $\Gamma$ is in fact analytic within $\Omega$, since it is a weak solution of $\Delta \Gamma=0$ in $\Omega$. We can then fix a (compactly supported) $C^{\infty}$-vector field $X$ defined on $\mathbb{R}^3$ which is everywhere inward pointing along $\Sigma$. If we let $\Xi_{\tau}$ denote the flow of $X$ we obtain a $C^{1}$-curve $\sigma_p(\tau):=\Xi_\tau\circ\sigma_p\subset \Omega$ which still bounds a $C^1$-disc within $\Omega$ since $X$ is inward pointing. It then follows from an application of Stokes' theorem and interior regularity of $\Gamma$ that $\int_{\sigma_p(\tau)}\Gamma=0$ for all $\tau>0$. On the other hand, Sobolev embeddings tell us that $\Gamma$ is continuous up to the boundary, so that it is not hard to see that $\int_{\sigma_p}\Gamma=\lim_{\tau\searrow 0}\int_{\sigma_p(\tau)}\Gamma=0$ as previously suggested so that $\int_{\sigma_p}\gamma=0$. Finally, if $\int_{\sigma_t}\gamma=0$ would be true, it would follow, because $\sigma_p$ and $\sigma_t$ generate the first fundamental group, that $\int_{\sigma}\gamma=0$ for every closed loop $\sigma$ so that a standard construction would imply that $\gamma$ is a gradient field implying that $\gamma=0$ since $\gamma$ is div-free and hence $L^2$-orthogonal to the gradient fields. Since $\gamma\neq 0$ we conclude that $\frac{\gamma}{\int_{\sigma_t}\gamma}$ satisfies the defining equations of $\gamma_t$ and so by uniqueness $\gamma_t=\frac{\gamma}{\int_{\sigma_t}\gamma}$. Using this, we obtain from (\ref{ExtraEquation1}) and \Cref{L212}
	\begin{gather}
		\label{ExtraEquation3}
		\alpha=\overline{Q}(v)|\Sigma|\int_{\sigma_t}\gamma.
	\end{gather}
	Now, since $\gamma$ and $\tilde{\gamma}$ form a basis of $\mathcal{H}(\Sigma)$ we can also express $\gamma_p=\mu \gamma+\lambda \tilde{\gamma}$ for suitable $\mu,\lambda\in \mathbb{R}$. We find $1=\int_{\sigma_p}\gamma_p=\lambda\int_{\sigma_p}\tilde{\gamma}$ so that $\lambda=\frac{1}{\int_{\sigma_p}\tilde{\gamma}}$ where $\int_{\sigma_p}\tilde{\gamma}\neq 0$ because otherwise $\gamma$ and $\tilde{\gamma}$ would be linearly dependent. Lastly, $0=\int_{\sigma_t}\gamma_p=\mu\int_{\sigma_t}\gamma+\frac{\int_{\sigma_t}\tilde{\gamma}}{\int_{\sigma_p}\tilde{\gamma}}$. We can solve this for $\mu$ and conclude from (\ref{ExtraEquation1}) and \Cref{L212}
	\begin{gather}
		\label{ExtraEquation4}
		\beta=\int_{\Sigma}\frac{\gamma_p}{\lambda}\cdot v-\frac{\mu}{\lambda}\gamma\cdot vd\sigma(x)=|\Sigma|\left(\int_{\sigma_p}\tilde{\gamma}\right)\left(\overline{P}(v)+\frac{\int_{\sigma_t}\tilde{\gamma}}{\int_{\sigma_p}\tilde{\gamma}}\overline{Q}(v)\right)
	\end{gather}
	where we used that $\gamma_t=\frac{\gamma}{\int_{\sigma_t}\gamma}$. Inserting the expressions (\ref{ExtraEquation3}) and (\ref{ExtraEquation4}) into (\ref{ExtraEquation2}) yields the result.
	
	Finally, if $\sigma_t$ bounds a surface $\mathcal{A}$ outside of $\Omega$, the result follows from \Cref{CP2}.
\end{proof}
\subsection{Proof of \Cref{T218}}
\begin{proof}[Proof of \Cref{T218}]
	We will divide the proof in several steps. We at first consider the divergence-free vector field $w\in C^{0,1}\mathcal{V}_0(\Sigma)$. We recall that by assumption there exists a positive function $f\in C^{0,1}(\Sigma,(0,\infty))$ such that $v=\frac{w}{f}$ is rectifiable. Then in the last step we show that the rotational transform of $v$ coincides with that of $w$ which will conclude the proof.
	\newline
	\newline
	\underline{Step 1:} We observe that $v$ being rectifiable implies that either all of its field lines are closed or none of them are. Assume in the first step that the field lines of $v$ are all closed. Then by the rectifiability property it follows further that all field lines are of the same type, i.e. if $\sigma_x$ is a field line of $v$ starting at $x\in \Sigma$ and we write $\sigma_x=P_x\sigma_p\oplus Q_x\sigma_t$ for suitable $P_x,Q_x\in \mathbb{Z}$ (and where $\sigma_p,\sigma_t$ denote the fixed purely poloidal and toroidal curve respectively), then in fact $P_x$ and $Q_x$ are independent of the chosen point $x$. Since $w$ and $v$ have the same field lines, which are merely traced out with a different speed, we conclude that also all field lines of $w$ are of the type $P\sigma_p\oplus Q\sigma_t$ with the same $P,Q\in \mathbb{Z}$ (which are independent of $x$ and we thus drop the subscript $x$). We recall (\ref{E2}) and see that if for a given $x\in \Sigma$ we let $\tau_x>0$ denote the period of the field line of $w$ starting at $x$, then $\hat{q}(x)=\frac{Q}{\tau_x}$ and with the same reasoning $\hat{p}(x)=\frac{P}{\tau_x}$ where $\hat{q},\hat{p}$ are the quantities associated with $w$. Further, by assumption, $\overline{Q}(w)\neq 0$ which by definition implies $0\neq \int_{\Sigma}\hat{q}(x)d\sigma(x)=Q\int_{\Sigma}\frac{1}{\tau_x}d\sigma(x)$ and consequently $Q\neq 0$ so that $w$ has a well-defined rotational transform at every $x\in \Sigma$ and by definition $\iota_w(x)=\frac{\hat{p}(x)}{\hat{q}(x)}=\frac{P}{Q}$ since $\hat{q}(x)=\frac{Q}{\tau_x}$, $\hat{p}(x)=\frac{P}{\tau_x}$. Finally, since $\hat{p}(x)=\frac{P}{Q}\hat{q}$ and $\frac{P}{Q}$ is a constant, we conclude $\overline{P}(w)=\frac{P}{Q}\overline{Q}(w)$ so that $\iota_w(x)=\frac{\overline{P}(w)}{\overline{Q}(w)}$ is independent of $x$ and the statement follows from \Cref{L212}.
	\newline
	\newline
	\underline{Step 2:} Now we assume that the field lines of $w$ are not closed. We recall that if we let $\mu$ denote the normalised surface measure on $\Sigma$, then an area preserving flow $\psi_t$ is ergodic if for any (Borel-)measurable $A\subset \Sigma$ with $\psi_{-t}(A)\subset A$ for all $t\geq 0$ we have $\mu(A)=0$ or $\mu(A)=1$.
	
	We claim first that the flow $\psi_t$ of $w$ is ergodic. Since $w$ is div-free the flow is clearly area preserving. On the other hand, since $v=\frac{w}{f}$ where $f$ is a strictly positive function, it is easy to see that if $\Psi_t$ denotes the flow of $v$, the condition $\psi_{-t}(A)\subset A$ for all $t\geq 0$ implies $\Psi_{-t}(A)\subset A$ for all $t\geq 0$. We have to prove that this implies that either $A$ or $\Sigma\setminus A$ is a null-set. However, since diffeomorphisms preserve null-sets it is enough, due to the rectifiability of $v$, to see that the linear flow with irrational slope is ergodic on the flat torus which is a well-known fact, see \cite[Chapter II Theorem 3.2]{Mane87}. Then the ergodicity and the $\psi_t$ invariance of $\hat{q}$ and $\hat{p}$ respectively imply that $\hat{q}$ and $\hat{p}$ are both constant a.e.. Then by definition we find $\overline{P}(w)=\hat{p}(x)$, $\overline{Q}(w)=\hat{q}(x)$ for a.e. $x$ and since by assumption $\overline{Q}(w)\neq 0$ we see that $w$ has a well-defined rotational transform given by $\iota_w(x)=\frac{\hat{p}(x)}{\hat{q}(x)}=\frac{\overline{P}(w)}{\overline{Q}(w)}$ so that the statement now follows once more from \Cref{L212}.
	\newline
	\newline
	\underline{Step 3:} Here we prove that $\iota_v(x)=\iota_w(x)$ for a.e. $x\in \Sigma$, which will complete the proof.
	
	We recall the definition $\hat{p}_v(x)=\lim_{T\rightarrow\infty}\frac{1}{T}\int_{\sigma_x^v[0,T]}\gamma_p$, where the subscript indicates that $\hat{p}$ refers to the poloidal twists of $v$ and where $\sigma_x^v$ denotes the field line of $v$ starting at $x$. Further recall that $w=fv$ for a strictly positive function $f\in C^{0,1}(\Sigma,(0,\infty))$. We can then consider 
	\begin{gather}
		\nonumber
	\hat{f}(x):=\lim_{S\rightarrow\infty}\frac{1}{S}\int_0^Sf(\sigma_x^w(t))dt,
\end{gather}
where $\sigma_x^w$ denotes the field line of $w$ starting at $x$. By standard ergodic theory, c.f. \cite[Theorem 2]{TME67}, this defines an element of $L^1(\Sigma)$ since the flow of $w$ is area-preserving. We observe that $f$ is bounded above and away from zero so that $0\leq c_1\leq \hat{f}(x)\leq c_2<\infty$ for suitable constants $0<c_1\leq c_2$ and all $x\in \Sigma$. Further, since $v$ and $w$ have the same field lines we may parametrise $\sigma^v_x(t)$ via $w$ so that we can write $\sigma_x^v[0,T]=\sigma_x^w[0,S]$ for a suitable $S>0$ (in case that $\sigma_x^v$ is periodic we pick $S$ such that $\sigma_x^w[0,S]$ passed $x$ the same amount of times as $\sigma_x[0,T]$ did). We therefore obtain
	\begin{gather}
		\label{ErgodicEquation1}
		\hat{p}_v(x)=\lim_{T\rightarrow\infty}\frac{1}{T}\int_{\sigma_x^v[0,T]}\gamma_p=\lim_{T\rightarrow\infty}\frac{1}{T}\int_{\sigma_x^w[0,S]}\gamma_p
	\end{gather}
	where we note that $S$ depends on $T$. To establish a relation between $T$ and $S$ we observe that $T=\int_{\sigma_x^v[0,T]}\frac{v}{|v|^2}$ and therefore we can write by choice of $S$
	\begin{gather}
		\nonumber
		T=\int_{\sigma_x^w[0,S]}\frac{v}{|v|^2}=\int_0^Sf(\sigma_x^w(t))dt
	\end{gather}
	where we used that $w=fv$. Since $f$ is bounded away from zero and bounded above we see that $T\rightarrow\infty \Leftrightarrow S\rightarrow\infty$. Inserting this relation into (\ref{ErgodicEquation1}) we find
	\begin{gather}
		\nonumber
		\hat{p}_v(x)=\lim_{T\rightarrow\infty}\frac{S}{T}\cdot \frac{1}{S}\int_{\sigma_x^w[0,S]}\gamma_p=\frac{\hat{p}_w(x)}{\hat{f}(x)}.
	\end{gather}
	In the same way one proves that $\hat{q}_v(x)=\frac{\hat{q}_w(x)}{\hat{f}(x)}$ so that in particular $v$ has a well-defined rotational transform if and only if $w$ has a well-defined rotational transform and we obtain $\iota_v(x)=\frac{\hat{p}_v(x)}{\hat{q}_v(x)}=\frac{\hat{p}_w(x)}{\hat{q}_w(x)}=\iota_w(x)$ as claimed.
\end{proof}
\subsection{Proof of \Cref{T223}}
\begin{proof}[Proof of \Cref{T223}]
	\underline{Case (i):} If $\Sigma\cong S^2$ it follows from \Cref{T25} that $\mathcal{H}(w)=0$ for all $w\in L^2\mathcal{V}_0(\Sigma)$ which implies the claim.
	\newline
	\newline
	\underline{Case (ii):} Given any $w\in L^2\mathcal{V}_0(\Sigma)$ we can express $w$ by means of the Hodge decomposition theorem as $w=\operatorname{grad}_{\Sigma}(f)\times\mathcal{N}+\gamma$ for a suitable $f\in H^1(\Sigma)$ and $\gamma\in \mathcal{H}(\Sigma)$. It then follows from \Cref{T25} that if $\mathcal{H}(w)\neq 0$, we must have $\gamma\neq 0$. In addition, if $\mathcal{H}(w)>0$, the $L^2(\Sigma)$-orthogonality of the Hodge-decomposition implies that if $\operatorname{grad}_{\Sigma}(f)\neq 0$ we must have $\frac{\mathcal{H}(w)}{\|w\|^2_{L^2(\Sigma)}}<\frac{\mathcal{H}(w)}{\|\gamma\|^2_{L^2(\Sigma)}}=\frac{\mathcal{H}(\gamma)}{\|\gamma\|^2_{L^2(\Sigma)}}$ and that therefore $w$ could not have been a maximiser. We will now at first show that there exists some $w\in L^2\mathcal{V}_0(\Sigma)$ with $\mathcal{H}(w)>0$ which by the previous argument will allow us to look for a maximiser within the space $\mathcal{H}(\Sigma)$. To see this we let $\Omega$ denote the bounded domain bounded by $\Sigma$ and fix any element $\Gamma\in \mathcal{H}_N(\Omega)\setminus \{0\}$ (the space of curl- and div-free $H^1(\Omega)$ vector fields which are tangent to $\Sigma$). Further, we let $\gamma$ denote the $L^2(\Sigma)$-orthogonal projection of $\Gamma|_{\Sigma}$ onto $\mathcal{H}(\Sigma)$ and we define $\tilde{\gamma}:=\gamma\times \mathcal{N}$. We recall that we have shown in the proof of \Cref{T25} that $2\mathcal{H}_c(\gamma,\tilde{\gamma})=\|\gamma\|^2_{L^2(\Sigma)}$ ($\mathcal{H}_c$ being the cross-helicity) and that \Cref{L31} states that $\mathcal{H}(\tilde{\gamma})=0$. We can now let $\beta\neq 0$ be any constant and set $\alpha:=\frac{1}{|\beta|}$. We define $\hat{\gamma}:=\alpha\gamma+\beta\tilde{\gamma}$ and observe that the previous considerations yield
	\begin{gather}
		\nonumber
		\mathcal{H}(\hat{\gamma})=\alpha^2\mathcal{H}(\gamma)+2\alpha\beta\mathcal{H}_c(\gamma,\tilde{\gamma})+\beta^2\mathcal{H}(\tilde{\gamma})=\alpha^2\mathcal{H}(\gamma)+\alpha\beta\|\gamma\|^2_{L^2(\Sigma)}=\frac{1}{\beta^2}\mathcal{H}(\gamma)+\operatorname{sign}(\beta)\|\gamma\|^2_{L^2(\Sigma)}
	\end{gather}
	where $\operatorname{sign}(\beta)$ equals $+1$ if $\beta>0$ and equals $-1$ if $\beta<0$. Letting $\beta \gg1$ or $\beta \ll -1$ we see that there is some $\hat{\gamma}_{\pm}$ with $\mathcal{H}(\hat{\gamma}_{\pm})\cdot (\pm1)>0$. Then we are left with observing that the functional $\{\hat{\gamma}\in \mathcal{H}(\Sigma)\mid \|\hat{\gamma}\|_{L^2(\Sigma)}=1\}\rightarrow\mathbb{R}$,$ \hat{\gamma}\mapsto \mathcal{H}(\hat{\gamma})$ admits a (necessarily positive) global maximum since the unit sphere in finite dimensional vector spaces is compact. The claim follows then from the scaling properties of the helicity, i.e. $\mathcal{H}(\lambda v)=\lambda^2\mathcal{H}(v)$ for all $\lambda\in \mathbb{R}$ and $v\in L^2\mathcal{V}_0(\Sigma)$.
\end{proof}
\subsection{Proof of \Cref{T224}}
\begin{proof}[Proof of \Cref{T224}]
	We recall that we let $\pi:L^2\mathcal{V}(\Sigma)\rightarrow L^2\mathcal{V}_0(\Sigma)$ denote the $L^2(\Sigma)$-orthogonal projection from the space of square integrable vector fields on $\Sigma$ into the space of square-integrable div-free fields on $\Sigma$. We prove first that $\pi\circ \operatorname{BS}_{\Sigma}:L^2\mathcal{V}_0(\Sigma)\rightarrow L^2\mathcal{V}_0(\Sigma)$ is self-adjoint with respect to $L^2(\Sigma)$. This is easy to see because for every $v,w\in L^2\mathcal{V}_0(\Sigma)$ we have $\langle v,(\pi\circ \operatorname{BS}_{\Sigma})(w)\rangle_{L^2(\Sigma)}=\langle v,\operatorname{BS}_{\Sigma}(w)\rangle_{L^2(\Sigma)}$ and one can then write out the definitions, use the cyclic property of the Euclidean inner product and Fubini's theorem to conclude the self-adjointness.
	
	We will now show that the image of $\pi\circ \operatorname{BS}_{\Sigma}$ is finite dimensional and therefore this operator must be compact. Indeed, it follows from \Cref{T25} that for any $f\in H^1(\Sigma)$ and $w\in L^2\mathcal{V}_0(\Sigma)$, we have $\langle \operatorname{grad}_{\Sigma}(f)\times \mathcal{N},(\pi\circ \operatorname{BS}_{\Sigma}(w))\rangle_{L^2(\Sigma)}=\langle \operatorname{grad}_{\Sigma}(f)\times \mathcal{N},\operatorname{BS}_{\Sigma}(w)\rangle_{L^2(\Sigma)}=\mathcal{H}_c(\operatorname{grad}_{\Sigma}(f)\times \mathcal{N},w)=0$ and so the Hodge decomposition theorem implies that $(\pi\circ \operatorname{BS}_{\Sigma})(w)\in \mathcal{H}(\Sigma)$ for every $w\in L^2\mathcal{V}_0(\Sigma)$ which is a finite dimensional space. It now follows from the spectral theorem for compact, self-adjoint operators that $\pi\circ \operatorname{BS}_{\Sigma}$ admits a discrete spectrum which accumulates at most at zero and a corresponding eigenbasis which together with $\operatorname{Ker}(\pi\circ \operatorname{BS}_{\Sigma})$ spans the space $L^2\mathcal{V}_0(\Sigma)$. According to the proof of \Cref{T223} we see that there exist elements $v,w\in L^2\mathcal{V}_0(\Sigma)$ with $\mathcal{H}(v)>0$ and $\mathcal{H}(w)<0$. Since $\mathcal{H}(v)=\langle v,(\pi\circ \operatorname{BS}_{\Sigma})(v)\rangle_{L^2(\Sigma)}$ and similarly for $w$, it must be the case that $\pi\circ \operatorname{BS}_{\Sigma}$ admits a positive and negative eigenvalue and since the spectrum accumulates at most at zero there must exist a largest positive and smallest negative eigenvalue. It then follows from the eigenspace decomposition that $v\in L^2\mathcal{V}_0(\Sigma)$ with $\|v\|_{L^2(\Sigma)}=1$ maximises helicity among all other such fields if and only if $v$ is an eigenfield of $\pi\circ\operatorname{BS}_{\Sigma}$ corresponding to the largest positive eigenvalue. Due to the scaling properties of helicity the same remains true for the maximisation of the Rayleigh quotient in (\ref{E4}).
	\end{proof}
	\subsection{Proof of \Cref{T225}}
	Before we come to the proof of \Cref{T225} we prove the claim in \Cref{R227}.
	\begin{lem}
		\label{LemmaExtra1}
		Let $T^2\cong\Sigma\subset \mathbb{R}^3$ be a closed, connected $C^{1,1}$-surface. Then
		\begin{gather}
			\nonumber
			\inf_{T^2\cong \Sigma\in C^{1,1}}\Lambda(\Sigma)=\inf_{T^2\cong \Sigma\in C^{1,1},|\Sigma|=1}\Lambda(\Sigma)\text{ and }\sup_{T^2\cong \Sigma\in C^{1,1}}\Lambda(\Sigma)=\sup_{T^2\cong \Sigma\in C^{1,1},|\Sigma|=1}\Lambda(\Sigma).
		\end{gather}
	\end{lem}
	\begin{proof}[Proof of \Cref{LemmaExtra1}]
	We first prove that for given $\Sigma$ and $\lambda>0$ we have $\Lambda(\Sigma)=\Lambda(\Sigma_{\lambda})$ where $\Sigma_{\lambda}:=\{\lambda x\mid x\in \Sigma\}$. To see this we can start with an arbitrary vector field $w\in L^2\mathcal{V}_0(\Sigma)$ and define $w_{\lambda}(x):=w\left(\frac{x}{\lambda}\right)$ for $x\in \Sigma_{\lambda}$. One easily verifies that $w_{\lambda}\in L^2\mathcal{V}_0(\Sigma_{\lambda})$ and that accordingly the map $w\mapsto w_{\lambda}$ gives rise to an isomorphism between $L^2\mathcal{V}_0(\Sigma)$ and $L^2\mathcal{V}_0(\Sigma_{\lambda})$ (note that $\Sigma=(\Sigma_{\lambda})_{\frac{1}{\lambda}}$). Then, using the change of variable formula and noting that $\psi_{\lambda}:\Sigma\rightarrow \Sigma_{\lambda}\text{, }x\mapsto \lambda x$ provides an (orientation-preserving) diffeomorphism between $\Sigma$ and $\Sigma_{\lambda}$, we find $\|w_{\lambda}\|^2_{L^2(\Sigma_{\lambda})}=\lambda^2\|w\|^2_{L^2(\Sigma)}$ for all $w\in L^2\mathcal{V}_0(\Sigma)$. As for helicity, we can write, keeping in mind that we perform a double integration,
	\begin{gather}
		\nonumber
		\mathcal{H}(w_{\lambda})=\frac{1}{4\pi}\int_{\Sigma_{\lambda}}\int_{\Sigma_{\lambda}}w_{\lambda}(x)\cdot \left(w_{\lambda}(y)\times \frac{x-y}{|x-y|^3}\right)d\sigma(y)d\sigma(x)
		\\
		\nonumber
		=\frac{\lambda^4}{4\pi}\int_{\Sigma}\int_{\Sigma}w(q)\cdot \left(w(p)\times \frac{\lambda q-\lambda p}{\left|\lambda q-\lambda p\right|^3}\right)d\sigma(p)d\sigma(q)=\lambda^2\mathcal{H}(w).
	\end{gather}
	Since $w\mapsto w_{\lambda}$ defines an isomorphism we obtain
	\begin{gather}
		\nonumber
		\Lambda_{\pm}(\Sigma_{\lambda})=\max_{\substack{\tilde{w}\neq 0\\ \tilde{w}\in L^2\mathcal{V}_0(\Sigma_{\lambda})}}\frac{\pm\mathcal{H}(\tilde{w})}{\|\tilde{w}\|^2_{L^2(\Sigma_{\lambda})}}=\max_{\substack{w\neq 0\\ w\in L^2\mathcal{V}_0(\Sigma)}}\frac{\pm\mathcal{H}(w_{\lambda})}{\|w_{\lambda}\|^2_{L^2(\Sigma_{\lambda})}}=\max_{\substack{w\neq 0\\ w\in L^2\mathcal{V}_0(\Sigma)}}\frac{\pm\mathcal{H}(w)}{\|w\|^2_{L^2(\Sigma)}}=\Lambda_{\pm}(\Sigma)
	\end{gather}
	and therefore $\Lambda(\Sigma_{\lambda})=\Lambda(\Sigma)$ for all $\lambda>0$. We can now take sequences $(\Sigma_n)_n$ approaching either the infimum or the supremum on the left hand side of the statement of the lemma respectively and rescale the $\Sigma_n$ by appropriate $\lambda_n>0$ to normalise their area which in combination with $\Lambda(\Sigma_{\lambda})=\Lambda(\Sigma)$ immediately implies the claim.
	\end{proof}
	\begin{proof}[Proof of \Cref{T225}]
		We recall that according to \Cref{T223} and \Cref{T224} the $L^2(\Sigma)$-normalised vector fields $\gamma_+,\gamma_-$ realising the Rayleigh quotient (\ref{E4}) are elements of $\mathcal{H}(\Sigma)$ and eigenfields of $\pi\circ \operatorname{BS}_{\Sigma}$ where $\pi:L^2\mathcal{V}(\Sigma)\rightarrow L^2\mathcal{V}_0(\Sigma)$ denotes the $L^2(\Sigma)$-orthogonal projection onto $L^2\mathcal{V}_0(\Sigma)$. In addition, while proving the compactness of $\pi\circ \operatorname{BS}_{\Sigma}$ during the course of the proof of \Cref{T224}, we also showed that $\pi\circ \operatorname{BS}_{\Sigma}$ maps $L^2\mathcal{V}_0(\Sigma)$ onto $\mathcal{H}(\Sigma)$. Therefore, the operator $\pi\circ \operatorname{BS}_{\Sigma}:\mathcal{H}(\Sigma)\rightarrow \mathcal{H}(\Sigma)$ is well-defined and the Rayleigh quotient (\ref{E4}) is realised by some eigenfields of this restricted operator. However, the Hodge isomorphism \cite[Theorem 2.6.1]{S95}, implies $\dim\left(\mathcal{H}(\Sigma)\right)=\dim\left(H^1_{\operatorname{dR}}(\Sigma)\right)=2$, where $H^1_{\operatorname{dR}}(\Sigma)$ denotes the first de Rham cohomology group and where we used that $\Sigma\cong T^2$. It follows further from \Cref{T224} that $\pi\circ \operatorname{BS}_{\Sigma}$ admits at least one positive and one negative eigenvalue. Consequently, $\pi\circ \operatorname{BS}_{\Sigma}|_{\mathcal{H}(\Sigma)}$ is a self-adjoint operator on the $2$-dimensional Hilbert space $\left(\mathcal{H}(\Sigma),\langle\cdot,\cdot\rangle_{L^2(\Sigma)}\right)$. It therefore admits precisely two eigenvalues and thus one eigenvalue is strictly positive $\lambda_+(\Sigma)>0$ and one eigenvalue is negative $\lambda_-(\Sigma)<0$ with corresponding ($L^2(\Sigma)$-orthonormal) eigenfields $\gamma_+(\Sigma)$ and $\gamma_-(\Sigma)$ respectively. It follows immediately from the eigenspace decomposition that $\Lambda_+(\Sigma)=\mathcal{H}(\gamma_+)=\lambda_+(\Sigma)$ and $\Lambda_-(\Sigma)=-\mathcal{H}(\gamma_-)=-\lambda_-(\Sigma)=|\lambda_-(\Sigma)|$. In order to obtain a relation between $\lambda_+(\Sigma)$ and $\lambda_-(\Sigma)$ we will now find an explicit matrix representation of $\left(\pi\circ \operatorname{BS}_{\Sigma}\right)|_{\mathcal{H}(\Sigma)}$. To this end we let $\Omega\subset\mathbb{R}^3$ denote the bounded domain which is bounded by $\Sigma$, $\partial\Omega=\Sigma$. Then we can fix any $\Gamma\in \mathcal{H}_N(\Omega)\setminus \{0\}$ and as usual let $\gamma\in \mathcal{H}(\Sigma)$ denote the $L^2(\Sigma)$-orthogonal projection of $\Gamma|_{\Sigma}$ onto $\mathcal{H}(\Sigma)$. Upon rescaling we may assume that $\gamma$ is $L^2(\Sigma)$-normalised. We then define $\tilde{\gamma}:=\gamma\times\mathcal{N}$ where $\mathcal{N}$ is the outward unit normal and note that $\tilde{\gamma}\in \mathcal{H}(\Sigma)$ is also $L^2(\Sigma)$-normalised. We then observe that
		\begin{gather}
			\nonumber
			(\pi\circ \operatorname{BS}_{\Sigma})(\gamma)=\langle (\pi\circ \operatorname{BS}_{\Sigma})(\gamma),\gamma\rangle_{L^2(\Sigma)}\gamma+\langle (\pi\circ \operatorname{BS}_{\Sigma})(\gamma),\tilde{\gamma}\rangle_{L^2(\Sigma)}\tilde{\gamma}=\mathcal{H}(\gamma)\gamma+\mathcal{H}_c(\gamma,\tilde{\gamma})\tilde{\gamma}.
		\end{gather}
		Similarly we find $(\pi\circ\operatorname{BS}_{\Sigma})(\tilde{\gamma})=\mathcal{H}_c(\tilde{\gamma},\gamma)\gamma+\mathcal{H}(\tilde{\gamma})\tilde{\gamma}$ and therefore $\pi\circ \operatorname{BS}_{\Sigma}$ has the following matrix representation with respect to the basis $\mathcal{B}:=\{\gamma,\tilde{\gamma}\}$
		\begin{gather}
			\nonumber
			\mathcal{M}:=(\pi\circ \operatorname{BS}_{\Sigma})_{\mathcal{B}}^{\mathcal{B}}=\begin{pmatrix}
				\mathcal{H}(\gamma)      & \mathcal{H}_c(\gamma,\tilde{\gamma})  \\
				\mathcal{H}_c(\tilde{\gamma},\gamma)       & \mathcal{H}(\tilde{\gamma}) 
			\end{pmatrix}.
		\end{gather}
		It now follows from \Cref{L31} that $\mathcal{H}(\tilde{\gamma})=0$ and it follows from the first step in the proof of \Cref{T25} and the symmetry of the Biot-Savart operator that $2\mathcal{H}_c(\tilde{\gamma},\gamma)=2\mathcal{H}_c(\gamma,\tilde{\gamma})=\|\gamma\|^2_{L^2(\Sigma)}=1$ due to the normalisation of $\gamma$. We obtain $\mathcal{M}=\begin{pmatrix}
			\mathcal{H}(\gamma)      & \frac{1}{2}  \\
			\frac{1}{2}       & 0 
		\end{pmatrix}$ and conclude 
		\begin{gather}
			\label{ProofMinimisation1}
			\lambda_+(\Sigma)\lambda_-(\Sigma)=\det(\mathcal{M})=-\frac{1}{4}\Leftrightarrow \Lambda_+(\Sigma)\Lambda_-(\Sigma)=\frac{1}{4}.
		\end{gather}
		Now if both $\Lambda_+(\Sigma)$ and $\Lambda_-(\Sigma)$ would be strictly smaller than $\frac{1}{2}$ it would contradict (\ref{ProofMinimisation1}) and therefore we must have $\Lambda(\Sigma)\geq \frac{1}{2}$.
		
		Lastly we observe that $\Lambda_+(\Sigma)-\Lambda_-(\Sigma)=\lambda_+(\Sigma)+\lambda_-(\Sigma)=\operatorname{Trace}(\mathcal{M})=\mathcal{H}(\gamma)$. So if $\mathcal{H}(\gamma)=0$ we see that $\Lambda_+(\Sigma)=\Lambda_-(\Sigma)$ and then (\ref{ProofMinimisation1}) becomes $\Lambda^2(\Sigma)=\frac{1}{4}\Rightarrow \Lambda(\Sigma)=\frac{1}{2}$. Conversely, if $\Lambda(\Sigma)=\frac{1}{2}$ we note that (\ref{ProofMinimisation1}) can be equivalently expressed as $\Lambda(\Sigma)\min\{\Lambda_+(\Sigma),\Lambda_-(\Sigma)\}=\frac{1}{4}$ and that therefore $\min\{\Lambda_+(\Sigma),\Lambda_-(\Sigma)\}=\frac{1}{2}=\Lambda(\Sigma)$ and so $\Lambda_+(\Sigma)=\Lambda_-(\Sigma)$ so that $0=\operatorname{Trace}(\mathcal{M})=\mathcal{H}(\gamma)$.
	\end{proof}
	\subsection{Proof of \Cref{C226}}
	\begin{proof}[Proof of \Cref{C226}]
		According to \Cref{T225} we only need to prove that $\mathcal{H}(\gamma)=0$ whenever $\Sigma$ is a $C^{1,1}$-rotationally symmetric torus where we recall that $\gamma$ is the $L^2(\Sigma)$-orthogonal projection of any element $\Gamma\in \mathcal{H}_N(\Omega)\setminus \{0\}$ onto $\mathcal{H}(\Sigma)$ and where $\Omega$ denotes the bounded domain with boundary $\Sigma$. Upon applying isometries to $\Sigma$ we may assume that the axis of symmetry of $\Sigma$ is the $z$-axis. One can then consider the vector field $Y(x,y,z):=(-y,x,0)^{\operatorname{tr}}$ defined on all of $\mathbb{R}^3$ and which induces isometries. Since $\Sigma$ is a rotationally symmetric regular torus, it does not intersect the axis of symmetry so that in particular $\Omega$ does not intersect it either. One can therefore consider the vector field $\Gamma:=\frac{Y}{|Y|^2}$ on $\Omega$ which can be easily verified to satisfy $\operatorname{curl}(\Gamma)=0$ and $\operatorname{div}(\Gamma)=0$. In addition, since $Y$ generates rotations around the $z$-axis, we find $Y\parallel \Sigma$ and hence $\Gamma\parallel \Sigma$ so that $\Gamma\in \mathcal{H}_N(\Omega)\setminus \{0\}$. Since $\Gamma$ is curl-free and tangent to $\Sigma$ it is standard that $\Gamma|_{\Sigma}$ is also curl-free as a vector field on $\Sigma$. Finally we note that $Y$ induces isometries so that it is a Killing field. But the restriction of Killing fields to invariant surfaces remain Killing so that one finds that $Y|_{\Sigma}$ is div-free. In addition, a direct calculation yields $Y\cdot \operatorname{grad}(|Y|^2)=0$ on all of $\mathbb{R}^3$ so that the product rule gives us $\operatorname{div}_{\Sigma}(\Gamma|_{\Sigma})=0$. Consequently $\Gamma|_{\Sigma}\in \mathcal{H}(\Sigma)$ and so $\gamma=\Gamma|_{\Sigma}$. Lastly we note that $\Gamma|_{\Sigma}$ and $Y|_{\Sigma}$ have the same field lines which are all mutually unlinked and periodic with a uniformly bounded period. It then follows from the linking interpretation of helicity, \Cref{T26} and \Cref{R27}, that $\mathcal{H}(\gamma)=0$ so that the claim follows from \Cref{T225}.
	\end{proof}
	\subsection{Proof of \Cref{T230}}
	\begin{proof}[Proof of \Cref{T230}]
		We recall that we call an element $j\in L^2\mathcal{V}_0(\Sigma)$ simple, if it satisfies $\int_{\Sigma}j\cdot \gamma_td\sigma=0$ or equivalently, c.f. \Cref{L212}, $\overline{Q}(j)=0$ where $\gamma_t$ is the unique element of $\mathcal{H}(\Sigma)$ satisfying $\int_{\sigma_p}\gamma_t=0$ and $\int_{\sigma_t}\gamma_t=1$ (recall also that the notion of being simple is independent of the chosen curves $\sigma_p,\sigma_t$). We have to show that for any given $j\in L^2\mathcal{V}_0(\Sigma)$ there exists some $j^\prime\in L^2\mathcal{V}_0(\Sigma)$ which is simple and satisfies $\operatorname{BS}(j)(x)=\operatorname{BS}(j^\prime)(x)$ for all $x\in \Omega$, where $\operatorname{BS}(j)(x):=\frac{1}{4\pi}\int_{\Sigma}j(y)\times \frac{x-y}{|x-y|^3}d\sigma(y)$ corresponds physically to the magnetic field inside the bounded domain $\Omega$ bounded by $\Sigma$ which is induced by the current density $j\in L^2\mathcal{V}_0(\Sigma)$. So let $j\in L^2\mathcal{V}_0(\Sigma)$ be given, it then follows from \cite[Theorem 5.1]{G24} that there exists some $j_0\in \operatorname{Ker}(\operatorname{BS})\setminus\{0\}\subset L^2\mathcal{V}_0(\Sigma)\cap \bigcap_{0<\alpha<1}C^{0,\alpha}\mathcal{V}(\Sigma)$ so that $\operatorname{BS}(j+\alpha j_0)=\operatorname{BS}(j)$ for every $\alpha\in \mathbb{R}$. Further, it was shown in the first step of the proof of \cite[Proposition 5.8]{G24} that if we fix any $\Gamma\in \mathcal{H}_N(\Omega)\setminus \{0\}$, then $j_0$ can be chosen such that $\int_{\Sigma}j_0(y)\cdot \Gamma(y)d\sigma(y)=\|\Gamma\|^2_{L^2(\Omega)}\neq 0$. Now we observe first that $\operatorname{curl}(\Gamma)=0$ in $\Omega$ implies that $\Gamma|_{\Sigma}$ is curl-free on $\Sigma$ and hence we can use the Hodge decomposition theorem to express $\Gamma=\operatorname{grad}_{\Sigma}(f)+\gamma$ for suitable $f\in \bigcap_{0<\alpha<1}C^{1,\alpha}(\Sigma)$ and $\gamma\in \mathcal{H}(\Sigma)$. Since $j_0$ is div-free it is $L^2(\Sigma)$-orthogonal to all gradient fields, which yields $\int_{\Sigma}j_0\cdot \gamma d\sigma=\int_{\Sigma}j_0\cdot \Gamma d\sigma\neq 0$. Finally we recall that we have shown in the proof of \Cref{T213} that $\gamma_t=\frac{\gamma}{\int_{\sigma_t}\gamma}$ and consequently $\int_{\Sigma}j_0\cdot \gamma_td\sigma\neq 0$. Hence, by linearity, we may always find a suitable $\alpha_j\in \mathbb{R}$ such that $\int_{\Sigma}(j+\alpha_j j_0)\cdot \gamma_td\sigma=0$ so that we may set $j^\prime:=j+\alpha_jj_0$ and the regularity claims follow immediately from the regularity of $j_0$. Finally, we need to show that such a $j^\prime$ is unique. To this end, suppose $j^\prime,\hat{j}\in L^2\mathcal{V}_0(\Sigma)$ are simple and satisfy $\operatorname{BS}(j^\prime)=\operatorname{BS}(\hat{j})$. Then $\bar{j}:=j^\prime-\hat{j}$ is also simple and contained in the kernel of $\operatorname{BS}$. Therefore $\bar{j}\in \operatorname{Ker}(\operatorname{BS})$ and consequently there exists some $\beta\in \mathbb{R}$ with $\bar{j}=\beta j_0$ where $j_0$ is chosen like above. In particular, $\int_{\Sigma}j_0\cdot \gamma_td\sigma\neq 0$ so that the simplicity of $\bar{j}$ implies $\beta=0$ and consequently $j^\prime=\hat{j}$.
	\end{proof}
	\subsection{Proof of \Cref{C231}}
	\begin{proof}[Proof of \Cref{C231}]
		It follows from \cite[Corollary 3.10 (iii,b)]{G24} that for every $B_T\in L^2\mathcal{H}(P)$ and every $\epsilon>0$ there exists some $j\in L^2\mathcal{V}_0(\Sigma)$ with $\|\operatorname{BS}(j)-B_T\|_{L^2(P)}\leq \epsilon$. Then \Cref{T230} implies that there exists $j^\prime\in L^2\mathcal{V}_0(\Sigma)$ which is simple and with $\operatorname{BS}(j)=\operatorname{BS}(j^\prime)$ which concludes the proof.
	\end{proof}
	\subsection{Proof of \Cref{T232}}
	\begin{proof}[Proof of \Cref{T232}]
		The proof is in essence identical to the proof of \cite[Proposition 4.1]{G24}. For convenience of the reader we present it here nonetheless. Given $B_T\in L^2\mathcal{H}(P)$ we can fix any $\epsilon>0$. According to \Cref{C231} we can find some $j_{\epsilon}\in L^2\mathcal{V}_0(\Sigma)$ which is simple and such that $\|\operatorname{BS}(j_{\epsilon})-B_T\|_{L^2(P)}^2\leq \epsilon$. Then, by definition, $j_{\lambda}^0$ is the unique global minimiser of $C_0(\lambda;B_T)$ (recall (\ref{E25})) so that
		\begin{gather}
			\nonumber
			C_0(\lambda;B_T)\leq \|\operatorname{BS}(j_{\epsilon})-B_T\|^2_{L^2(P)}+\lambda \|j_{\epsilon}\|^2_{L^2(\Sigma)}\leq \epsilon+\lambda\|j_{\epsilon}\|^2_{L^2(\Sigma)}
		\end{gather}
		by choice of $j_{\epsilon}$. Further, we note that
		\begin{gather}
			\nonumber
			\|\operatorname{BS}(j_{\lambda}^0)-B_T\|^2_{L^2(P)}\leq \|\operatorname{BS}(j_{\lambda}^0)-B_T\|^2_{L^2(P)}+\lambda\|j^0_{\lambda}\|^2_{L^2(\Sigma)}=C_0(\lambda;B_T)
		\end{gather}
		and hence
		\begin{gather}
			\nonumber
			\|\operatorname{BS}(j^0_{\lambda})-B_T\|^2_{L^2(P)}\leq \epsilon+\lambda\|j_{\epsilon}\|^2_{L^2(\Sigma)}.
		\end{gather}
		We recall that $j_{\epsilon}$ depends only on $B_T$ and $\epsilon$ but is independent of $\lambda$. Therefore we may take the limit and obtain
		\begin{gather}
			\nonumber
			0\leq \liminf_{\lambda \searrow 0}	\|\operatorname{BS}(j^0_{\lambda})-B_T\|^2_{L^2(P)}\leq \limsup_{\lambda \searrow 0}	\|\operatorname{BS}(j^0_{\lambda})-B_T\|^2_{L^2(P)}\leq \epsilon.
		\end{gather}
		Since $\epsilon>0$ was arbitrary we conclude $\lim_{\lambda\searrow 0}\|\operatorname{BS}(j^0_{\lambda})-B_T\|_{L^2(P)}=0$.
	\end{proof}
\section*{Acknowledgements}
I would like to thank Ugo Boscain and Mario Sigalotti for fruitful discussions about the interpretation of surface helicity, in particular for suggesting to try to artificially close the field lines as depicted in \Cref{ClosedCurve}. This work has been supported by the Inria AEX StellaCage.
\appendix
\section{Continuity of surface Biot-Savart operator}
\label{AS1}
\begin{lem}
	\label{AL1}
	\Cref{D21} is well-defined, i.e. the Biot-Savart operator indeed defines a $q$-integrable vector field for any choice of orthonormal frame and it is independent of the chosen frame.
\end{lem}
\begin{proof}[Proof of \Cref{AL1}]
	First let us suppose that we have shown that
	\[
	\sum_{i=1}^2\left(\int_\Sigma v(y)\times \frac{x-y}{|x-y|^3}d\sigma(y)\cdot E_i(x)\right)E_i(x)
	\]
	is finite for a.e. $x$ in the domain of definition of any fixed orthonormal frame $E_i$. Then, if $E_i$ and $\tilde{E}_j$ are any two orthonormal frames whose domains of definition overlap, then on this overlap the integrals will exist almost surely for any $x$ in the overlap. For any such fixed $x$ we then have
	\begin{gather}
		\nonumber
		\sum_{i=1}^2\left(\int_\Sigma v(y)\times \frac{x-y}{|x-y|^3}d\sigma(y)\cdot E_i(x)\right)E_i(x)
		\\
		\nonumber
		=\sum_{j=1}^2\sum_{i=1}^2\left(\int_\Sigma v(y)\times \frac{x-y}{|x-y|^3}d\sigma(y)\cdot E_i(x)\right)(E_i(x)\cdot \tilde{E_j}(x))\tilde{E}_j(x)
		\\
		\nonumber
		=\sum_{j=1}^2\left(\int_\Sigma v(y)\times \frac{x-y}{|x-y|^3}d\sigma(y)\cdot \sum_{i=1}^2\left[(E_i(x)\cdot \tilde{E_j}(x))E_i(x)\right]\right)\tilde{E}_j(x)
		\\
		\nonumber
		=\sum_{j=1}^2\left(\int_\Sigma v(y)\times \frac{x-y}{|x-y|^3}d\sigma(y)\cdot \tilde{E}_j(x)\right)\tilde{E}_j(x)
	\end{gather}
	and so the value of the Biot-Savart operator is independent of the chosen local orthonormal frames.
	
	In order to establish the $q$-integrability we define $u(y):=\mathcal{N}(y)\times v(y)$ and fix any $w\in T_x\Sigma$. We observe that because $v(y)$ is tangent to $\Sigma$ we have $v(y)=u(y)\times \mathcal{N}(y)$. Then, using the vector triple product rule, we find
	\begin{gather}
		\nonumber
		\int_\Sigma v(y)\times \frac{x-y}{|x-y|^3}d\sigma(y)\cdot w=\int_\Sigma \left(u(y)\cdot \frac{x-y}{|x-y|^3}\right)\mathcal{N}(y)-\left(\mathcal{N}(y)\cdot \frac{x-y}{|x-y|^3}\right)u(y)d\sigma(y)\cdot w.
	\end{gather}
	It is well-known that there is some $c>0$ satisfying $\left|\frac{x-y}{|x-y|^3}\cdot \mathcal{N}(y)\right|\leq \frac{c}{|x-y|}$ for all $y,x\in \Sigma$ with $y\neq x$, c.f. \cite[Lemma 41.12]{Ser17}. On the other hand, we have, since $w$ is orthogonal to $\mathcal{N}(x)$,
	\[
	\left|\left(u(y)\cdot \frac{x-y}{|x-y|^3}\right)(\mathcal{N}(y)\cdot w)\right|=\left|\left(u(y)\cdot \frac{x-y}{|x-y|^3}\right)(\mathcal{N}(y)-\mathcal{N}(x))\cdot w\right|\leq L\frac{|v(y)||w|}{|x-y|}
	\]
	where $L$ is the Lipschitz character of $\mathcal{N}$ and we used that $|u(y)|=|v(y)|$. Therefore, if $|w|=1$, then we see that there is some constant $c_0>0$ which is independent of $v$ such that
	\[
	\left|\int_\Sigma v(y)\times \frac{x-y}{|x-y|^3}d\sigma(y)\cdot w\right|\leq c_0\int_\Sigma \frac{|v(y)|}{|x-y|}d \sigma(y).
	\]
	It follows easily from the dominated convergence theorem that if $1\leq \beta <2$ then the map $\Sigma\rightarrow\mathbb{R}$, $x\mapsto \int_\Sigma \frac{1}{|x-y|^\beta}d\sigma(y)$ is continuous so that it is globally bounded on $\Sigma$. Hence, the above bound ensures that $\operatorname{BS}_\Sigma(v)\in L^\infty\mathcal{V}(\Sigma)$ for any $v\in L^p\mathcal{V}(\Sigma)$ with $p>2$ by an application of H\"{o}lder's inequality.
	
	If $v\in L^2\mathcal{V}(\Sigma)$, then in particular $v\in L^p\mathcal{V}(\Sigma)$ for any $1<p<2$ and so the claim will follow from the case $p<2$ after observing that $\frac{2p}{2-p}$ converges to infinity as $p\nearrow 2$.
	
	Suppose now that $1<p<2$. In that case one can use a partition of unity to localise the problem and hence reduce the situation to an application of the standard Hardy-Littlewood-Sobolev inequality, c.f. \cite[Chapter V]{S70}.
\end{proof}
\section{Physical Rotational transform}
\label{SB}
To simplify the discussion, we consider here only "unknotted" smooth solid tori $\Omega\subset \mathbb{R}^3$ in the sense that we demand that there exists an orientation preserving $C^{\infty}$-diffeomorphism $\Psi:\mathbb{R}^3\rightarrow\mathbb{R}^3$ such that $\Omega=\Phi(\Omega_S)$ where $\Omega_S:=\{(x,y,z)^{\operatorname{tr}}\in \mathbb{R}^3\mid (\sqrt{x^2+y^2}-2)^2+z^2\leq1\}$ is the standard rotationally symmetric solid torus bounded by a torus with minor radius $r=1$ and major radius $R=2$. We have the following parametrisation of $\Sigma_S:=\partial\Omega_S$, $(\phi,\theta)\mapsto \mu(\phi,\theta)=\left((2+\cos(\theta))\cos(\phi),(2+\cos(\theta))\sin(\phi),\sin(\theta)\right)^{\operatorname{tr}}$, $0\leq \phi,\theta\leq 2\pi$ where $\phi$ is called the "toroidal" angle and $\theta$ is called the "poloidal angle". Further, we consider the following curves $\sigma^S_p(\tau):=\left(2+\cos(\tau),0,\sin(\tau)\right)^{\operatorname{tr}}$, $\sigma^S_t(\tau):=\left(\cos(\tau),\sin(\tau),0\right)^{\operatorname{tr}}$, $0\leq \tau \leq 2\pi$ and call them the standard purely poloidal and standard purely toroidal loops. We then refer to $\sigma_p:=\Phi\circ \sigma^S_p$ and $\sigma_t:=\Phi\circ \sigma^S_t$ as a purely poloidal and a purely toroidal curve on $\Sigma=\partial\Omega$.

In addition, we observe that for any fixed $\phi\in [0,2\pi)$, the curve $\sigma_{\phi}(\tau):=\mu(\phi,\tau)$, $0\leq \tau\leq 2\pi$, is a closed circle (which bounds a disc within $\Omega_S$). We call the family of images, $\Sigma^{\phi}_S:=\sigma_{\phi}([0,2\pi])$, $(\Sigma^{\phi}_S)_{0\leq \phi<2\pi}$ the standard foliation of $\Sigma_S$ into poloidal sections. Accordingly we call $(\Sigma_{\phi})_{0\leq \phi<2\pi}$ with $\Sigma_{\phi}:=\Phi(\Sigma^S_{\phi})$ a foliation of $\Sigma$ into poloidal sections.

Now one often assumes in plasma physics that the magnetic field lines of plasma equilibria $B$ are transversal to the poloidal sections, i.e. never tangent to any such section, recall also \Cref{R222} regarding the flow structure of plasma equilibria. Then, given a field line of $B$ starting at some initial poloidal section $\Sigma_{\phi_0}$, it will, due to the transversality assumption, continue to move along these poloidal sections $\Sigma_{\phi}$ until it returns to the initial poloidal section $\Sigma_{\phi_0}$. During this toroidal turn, the field line will also have made a certain amount of poloidal turns $(\Delta \theta)_1$. More precisely, we may write $\sigma(\tau)=\Phi(\sigma_S(\tau))$ where $\sigma_S(\tau)$ is a curve on $\Sigma_S$ parametrised by $\sigma_S(\tau)=\mu(\phi(\tau),\theta(\tau))$ for suitable smooth functions $\phi(\tau)$ and $\theta(\tau)$. The transversality condition demands that $\phi(\tau)$ is a strictly increasing function in $\tau$. Therefore there will be a unique $0<\tau_1<\infty$ with $\phi(\tau_1)=2\pi$ which corresponds to a full toroidal turn upon which we returned to $\Sigma_{\phi_0}$. Then we set $(\Delta \theta)_1:=\theta(\tau_1)$. Similarly we can continue to follow our field line until we arrive again at $\Sigma_{\phi_0}$ which uniquely determines a smallest time $\tau_1<\tau_2<\infty$ with $\phi(\tau_2)=4\pi$. Accordingly we define $(\Delta \theta)_2:=\theta(\tau_2)-\theta(\tau_1)$ to be the amount of poloidal turns made during the second toroidal turn. Similarly we define $(\Delta \theta)_k$ to be the amount of poloidal turns during the $k$-th toroidal turn.

One way to define "physical rotational transform" of a magnetic field line $\sigma$ starting at $\Sigma_{\phi_0}$ is then as follows, c.f. \cite[Chapter 11.7.8]{Frei07},
\begin{gather}
	\nonumber
	\iota_P(\sigma):=\lim_{n\rightarrow\infty}\frac{\sum_{k=1}^n(\Delta\theta)_k}{2\pi n}
\end{gather}
whenever the limit exists.

Suppose for now that $\sigma$ is periodic, then it will be homotopic to $P\sigma_p\oplus Q\sigma_t$ for suitable $P,Q\in \mathbb{Z}$. We observe that this means that $\sigma$ performs precisely $Q$ toroidal turns and $P$ poloidal turns within one period. We are assuming (without loss of generality) that $\sigma$ moves in the same toroidal direction as $\sigma_t$ which corresponds to $Q\geq 0$ and the transversality assumption ensures $Q\geq 1$. We know that after every $Q$ toroidal turns, we performed $P$ poloidal turns so that $\sum_{i=1}^{kQ}(\Delta\theta)_i=2\pi kP$. Thus, if $kQ\leq n<(k+1)Q$, we find $\sum_{i=1}^n(\Delta\theta)_i=2\pi kP+\sum_{i=kQ+1}^{n}(\Delta\theta)_i$ and we observe that $-2\pi P\leq \sum_{i=kQ+1}^{n}(\Delta\theta)_i\leq 2\pi P$ because we made less than $Q$ additional toroidal turns and hence at most $P$ poloidal turns. We note further that $Q\leq \frac{n}{k}\leq (1+\frac{1}{k})Q$ so that $\frac{n}{k}\rightarrow Q$ as $n\rightarrow\infty$. Hence, we overall arrive at
\begin{gather}
	\nonumber
\iota_P(\sigma)=\lim_{n\rightarrow\infty}\frac{\sum_{k=1}^n(\Delta\theta)_k}{2\pi n}=\lim_{n\rightarrow\infty}\left(\frac{k}{n}P+\frac{R(n)}{n}\right)=\frac{P}{Q}
\end{gather}
where $R(n)$ is the remainder term which we have shown to be bounded in absolute value by $2\pi P$. We recall that with our "mathematical" rotational transform $\iota$ we have, c.f. \Cref{R222}, $\iota(\sigma)=\frac{\int_{\Sigma}\widetilde{B}\cdot \gamma_pd\sigma}{\int_{\Sigma}\widetilde{B}\cdot \gamma_t d\sigma}$ where $\widetilde{B}=fB$ and $f>0$ is a suitable positive function related to the pressure function of the plasma equilibrium and $B$ is the underlying magnetic field. Most importantly, a field line $\sigma$ of $B$ starting at some point in $\Sigma$ is of type $P\sigma_p\oplus Q\sigma_t$ if and only if the same is true for $\widetilde{B}$. But then \Cref{L212} and (\ref{E2}) imply that $\iota(\sigma)=\frac{P}{Q}=\iota_P(\sigma)$ so that the rotational transform agrees in this case, recall also \Cref{T218}.

If the field line $\sigma$ is not closed, we can close it artificially. The idea is that at time $\tau_n>0$ when we return for the $n$-th time to the initial poloidal section $\Sigma_{\phi_0}$ we close the curve by following the path $\sigma_{\phi_0}$ from $\sigma(\tau_n)$ along the poloidal section until we reach the first point on $\Sigma_{\phi_0}$ which was intersected by $\sigma$ previously (recall that we will have visited $\Sigma_{\phi_0}$ by that time precisely $(n-1)$-times resulting (due to the non-periodicity) in $(n-1)$ intersection points and since $\Sigma_{\phi_0}$ is a circle, we must necessarily arrive at such a point, provided $n\geq 2$). The important observation is that by means of this closing procedure we add at most an angle $|\alpha_n|\leq 2\pi$ to the total toroidal angle. If we denote by $\tilde{\sigma}_n=\sigma[0,\tau_n]\oplus \sigma_{\phi_0}$ the corresponding concatenated path, we see that it will have made precisely $n$ toroidal turns and $P_n\in \mathbb{Z}$ poloidal turns. Hence, if we let $\iota^n_P(\sigma):=\frac{\sum_{k=1}^n(\Delta\theta)_k}{2\pi n}$ then we obtain
\begin{gather}
	\nonumber
	\iota^n_P(\sigma)=\frac{P_n}{n}-\frac{\alpha_n}{2\pi n}\text{, }|\alpha_n|\leq 2\pi\text{, }\iota_P(\sigma)=\lim_{n\rightarrow\infty}\iota^n_P(\sigma).
\end{gather}
On the other hand, with the "mathematical" definition of rotational transform we find
\begin{gather}
	\nonumber
	\iota(\sigma)=\lim_{T\rightarrow\infty}\frac{\int_{\sigma[0,T]}\gamma_p}{\int_{\sigma[0,T]}\gamma_t}
\end{gather}
and know that this limit exists for a.e. $x\in \Sigma$ at which the field line $\sigma$ starts and that the limit is independent of the chosen sequence $T\rightarrow\infty$, c.f. \Cref{T218}. We can therefore consider the sequence $T=\tau_n$ which diverges to $+\infty$ as $n\rightarrow\infty$. But we observe that $\int_{\sigma[0,\tau_n]}\gamma_p=\int_{\tilde{\sigma}_n}\gamma_p-\int_{\sigma_{\phi_0}[t^1_n,t^2_n]}\gamma_p=P_n-r_n$ where we define the remainder $r_n:=\int_{\sigma_{\phi_0}[t^1_n,t^2_n]}\gamma_p$ and where we used that $\tilde{\sigma}_n$ is closed and makes $P_n$ poloidal turns. We observe further that $|r_n|$ can be uniformly bounded because $\gamma_p$ is globally bounded on $\Sigma$, $\dot{\sigma}_{\phi}$ is globally bounded along any poloidal section and because by construction $t^2_n-t^1_n=|\alpha_n|\leq 2\pi$ is the additional toroidal angle which we need to sweep out to close our curve. Identically we can argue that $\int_{\sigma[0,\tau_n]}\gamma_t=n+R_n$ where $R_n$ is a uniformly (in $n$) bounded error term and we made use of the fact that $\tilde{\sigma}_n$ makes $n$ toroidal turns. We therefore arrive at
\begin{gather}
	\nonumber
	\frac{\int_{\sigma[0,\tau_n]}\gamma_p}{\int_{\sigma[0,\tau_n]}\gamma_t}=\frac{P_n-r_n}{n+R_n}=\frac{\frac{P_n}{n}-\frac{r_n}{n}}{1+\frac{R_n}{n}}=\frac{\iota^n_P(\sigma)+\frac{\alpha_n}{2\pi n}-\frac{r_n}{2\pi}}{1+\frac{R_n}{n}}.
\end{gather}
Since the limit on the left hand side exists we conclude that also the physical rotational transform is well-defined and that $\iota(\sigma)=\lim_{n\rightarrow\infty}\frac{\int_{\sigma[0,\tau_n]}\gamma_p}{\int_{\sigma[0,\tau_n]}\gamma_t}=\lim_{n\rightarrow\infty}\iota^n_P(\sigma)=\iota_P(\sigma)$ and therefore the physical rotational transform and mathematical rotational transform always coincide and hence \Cref{T218} provides a way to compute the rotational transform of a plasma equilibrium $B$, see also \Cref{R222}.
\section{Eigenfields of the surface Biot-Savart operator}
\label{SC}
\begin{prop}[Biot-Savart eigenfields]
	\label{CP1}
	Let $T^2\cong \Sigma\subset\mathbb{R}^3$ be a closed, connected $C^{1,1}$-torus and let $\Omega\subset\mathbb{R}^3$ be the unique bounded domain whose boundary is $\Sigma$. Let further $\Gamma\in \mathcal{H}_N(\Omega)\setminus\{0\}$ and let $\gamma$ denote the $L^2(\Sigma)$-orthogonal projection of $\Gamma|_{\Sigma}$ onto $\mathcal{H}(\Sigma)$ and set $\tilde{\gamma}:=\gamma\times \mathcal{N}$ where $\mathcal{N}$ denotes the outward unit normal. Further we scale $\Gamma$ such that $\gamma$ is $L^2(\Sigma)$-normalised. Then the set of vector fields $v_{\pm}\in \mathcal{H}(\Sigma)$ realising the Rayleigh quotient $\Lambda_{\pm}(\Sigma)$, c.f. (\ref{E4}), are $1$-dimensional subspaces $E_{\pm}(\Sigma)$ of $\mathcal{H}(\Sigma)$ respectively and spanned by the following vector fields
	\begin{gather}
		\nonumber
		E_+(\Sigma)=\operatorname{span}\left\{\left(\sqrt{\mathcal{H}^2(\gamma)+1}+\mathcal{H}(\gamma)\right)\gamma+\tilde{\gamma}\right\}\text{, }E_-(\Sigma)=\operatorname{span}\left\{\gamma-\left(\sqrt{\mathcal{H}^2(\gamma)+1}+\mathcal{H}(\gamma)\right)\tilde{\gamma}\right\}
	\end{gather}
	and we have
	\begin{gather}
		\nonumber
		\Lambda_+(\Sigma)=\frac{\sqrt{\mathcal{H}^2(\gamma)+1}+\mathcal{H}(\gamma)}{2}\text{, }\Lambda_-(\Sigma)=\frac{\sqrt{\mathcal{H}^2(\gamma)+1}-\mathcal{H}(\gamma)}{2}.
	\end{gather}
\end{prop}
\begin{proof}[Proof of \Cref{CP1}]
	We recall that we have shown in the proof of \Cref{T225} that we may restrict attention to the operator $(\pi\circ\operatorname{BS}_{\Sigma})|_{\mathcal{H}(\Sigma)}$ which maps into $\mathcal{H}(\Sigma)$ and has the following matrix representation with respect to the basis $\mathcal{B}:=\{\gamma,\tilde{\gamma}\}$
	\begin{gather}
		\nonumber
		\mathcal{M}:=(\pi\circ\operatorname{BS}_{\Sigma})^{\mathcal{B}}_{\mathcal{B}}=\begin{pmatrix}
			\mathcal{H}(\gamma) & \frac{1}{2} \\
			\frac{1}{2} & 0
		\end{pmatrix}.
	\end{gather}
	Computing the eigenvalues and eigenvectors of this matrix is a standard exercise and yields the result upon recalling that $\Lambda_+(\Sigma)=\lambda_+(\Sigma)$ and $\Lambda_-(\Sigma)=-\lambda_-(\Sigma)$ where $\lambda_+(\Sigma),\lambda_-(\Sigma)$ denote the positive and negative eigenvalue of $\mathcal{M}$ respectively and that the eigenvectors of $\mathcal{M}$ precisely coincide with the vector fields $v_{\pm}$ of interest.
\end{proof}
\begin{rem}
	We note that we have already argued in the proof of \Cref{C226} that $\mathcal{H}(\gamma)=0$ whenever $\Sigma$ is a rotationally symmetric $C^{1,1}$-torus. In addition, upon applying isometries and therefore assuming that the $z$-axis is the symmetry axis, we had found that $\gamma=\Gamma|_{\Sigma}$ where $\Gamma=\frac{Y}{|Y|^2}$ and $Y(x,y,z)=(-y,x,0)^{\operatorname{tr}}$ it the vector field inducing rotations around the $z$-axis. Hence \Cref{CP1} tells us that in this case $\Lambda_{\pm}(\Sigma)=\frac{1}{2}$ and that $E_{\pm}(\Sigma)=\left\{\gamma\pm\tilde{\gamma}\right\}$ where $\gamma$ is explicitly known and $\tilde{\gamma}=\gamma\times \mathcal{N}$ is determined by the outward unit normal field of $\Sigma$ which can for instance be obtained from a globally defining function of $\Sigma$.
\end{rem}
The downside of \Cref{CP1} is that we still need to compute $\mathcal{H}(\gamma)$ which is again a non-trivial task. The following proposition reduces the problem to compute certain (simple) line and surface integrals.
\begin{prop}
	\label{CP2}
	Let $T^2\cong \Sigma\subset\mathbb{R}^3$ bound a $C^{1,1}$-solid torus $\Omega\subset\mathbb{R}^3$ and suppose that $\sigma_t$ is a toroidal curve which bounds a bounded $C^{1,1}$-surface $\mathcal{A}$ outside of $\Omega$, i.e. $\mathcal{A}\subset \mathbb{R}^3\setminus \overline{\Omega}$ and $\partial\mathcal{A}=\sigma_t$. Let $\Gamma\in \mathcal{H}_N(\Omega)\setminus \{0\}$ be fixed and let $\gamma$ denote the $L^2(\Sigma)$-orthogonal projection of $\Gamma|_{\Sigma}$ onto $\mathcal{H}(\Sigma)$. Further let $\tilde{\gamma}:=\gamma\times \mathcal{N}$. Then we have
	\begin{gather}
		\nonumber
		\mathcal{H}(\gamma)=-\frac{\int_{\sigma_t}\tilde{\gamma}}{\int_{\sigma_t}\gamma}\|\gamma\|^2_{L^2(\Sigma)}.
	\end{gather}
\end{prop}
\begin{proof}[Proof of \Cref{CP2}]
	We follow the ideas of the proof of \Cref{T25}. We observe that we want to compute $4\pi\mathcal{H}(\gamma)=4\pi\mathcal{H}_c(\gamma,\gamma)$. We recall that $\tilde{\gamma}=\gamma\times \mathcal{N}$ and consequently $\gamma=\mathcal{N}\times \tilde{\gamma}$. We can then follow the reasoning of Step 1 in the proof of \Cref{T25} where we computed $\mathcal{H}_c(\tilde{\gamma},\gamma)$. We observe that since $\gamma=-(\tilde{\gamma}\times \mathcal{N})$ we will get an additional minus sign in the final expression. Apart from that only $\gamma(x)$ has to be replaced by $\tilde{\gamma}(x)$ throughout the computations which eventually leads us to the following identity
	\begin{gather}
		\nonumber
		4\pi\mathcal{H}(\gamma)=-\int_{\Sigma}\gamma(y)\cdot \left[\frac{4\pi\tilde{\gamma}(y)}{2}+\int_{\Omega}\operatorname{div}(\tilde{v})(x)\frac{x-y}{|x-y|^3}d^3x+\int_{\Omega}\operatorname{curl}(\tilde{v})(x)\times \frac{x-y}{|x-y|^3}d^3x\right]d\sigma(y)
		\\
		\nonumber
		=-\int_{\Sigma}\gamma(y)\cdot \left[\int_{\Omega}\operatorname{div}(\tilde{v})(x)\frac{x-y}{|x-y|^3}d^3x+\int_{\Omega}\operatorname{curl}(\tilde{v})(x)\times \frac{x-y}{|x-y|^3}d^3x\right]d\sigma(y)
	\end{gather}
	where we used that $\gamma$ and $\tilde{\gamma}$ are perpendicular and where $\tilde{v}$ is any $H^1$-vector field on $\Omega$ such that $\tilde{v}^\parallel=\tilde{\gamma}$. We note that if $\mathcal{N}\times \tilde{v}=\gamma$, then due to the relation of $\gamma$ and $\tilde{\gamma}$ this implies $\tilde{v}^\parallel=\tilde{\gamma}$. It then follows from \cite[Lemma 5.4]{G24} and its proof and the regularity of $\gamma$ that there exists some $\tilde{v}\in W^{1,2}\mathcal{V}(\Omega)$ satisfying $\mathcal{N}\times\tilde{v}=\gamma$, $\operatorname{div}(\tilde{v})=0$ in $\Omega$ and $\operatorname{curl}(\tilde{v})=\widetilde{\Gamma}$ where $\widetilde{\Gamma}\in \mathcal{H}_N(\Omega)$ satisfies $\|\widetilde{\Gamma}\|^2_{L^2(\Omega)}=\int_{\Sigma}\gamma\cdot \widetilde{\Gamma}d\sigma$. We note that in particular $\int_{\Sigma}\gamma\cdot \widetilde{\Gamma}d\sigma\geq 0$ and since $\int_{\Sigma}\gamma\cdot \Gamma d\sigma=\|\gamma\|^2_{L^2(\Sigma)}>0$ we must have, since $\mathcal{H}_N(\Omega)$ is spanned by $\Gamma$, $\widetilde{\Gamma}=\frac{\|\widetilde{\Gamma}\|_{L^2(\Omega)}}{\|\Gamma\|_{L^2(\Omega)}}\Gamma$ from which we conclude $\|\widetilde{\Gamma}\|_{L^2(\Omega)}=\frac{\int_{\Sigma}\gamma\cdot \widetilde{\Gamma}d\sigma}{\|\widetilde{\Gamma}\|_{L^2(\Omega)}}=\frac{\int_{\Sigma}\gamma\cdot \Gamma d\sigma}{\|\Gamma\|_{L^2(\Omega)}}$. We hence arrive at $\operatorname{curl}(\tilde{v})=\frac{\int_{\Sigma}\gamma\cdot \Gamma d\sigma}{\|\Gamma\|^2_{L^2(\Omega)}}\Gamma=\frac{\|\gamma\|^2_{L^2(\Sigma)}}{\|\Gamma\|^2_{L^2(\Omega)}}\Gamma$ where we used that $\gamma$ is the $L^2(\Sigma)$-orthogonal projection of $\Gamma|_{\Sigma}$ onto $\mathcal{H}(\Sigma)$.
	
	Combining these considerations, we arrive at
	\begin{AppA}
	\label{EquationC1}
	\mathcal{H}(\gamma)=\frac{\|\gamma\|^2_{L^2(\Sigma)}}{\|\Gamma\|^2_{L^2(\Omega)}}\int_{\Sigma}\gamma(y)\cdot \operatorname{BS}_{\Omega}(\Gamma)(y)d\sigma(y)
\end{AppA}
	where $\operatorname{BS}_{\Omega}(\Gamma)(y)=\frac{1}{4\pi}\int_{\Omega}\Gamma(x)\times \frac{y-x}{|y-x|^3}d^3x$ denotes the volume Biot-Savart operator.
	
	Since $\Gamma$ is div-free and tangent to the boundary of $\Omega$ it is standard, \cite[Proposition 1]{CDG01}, that $\operatorname{curl}(\operatorname{BS}_{\Omega}(\Gamma))=\Gamma$ in $\Omega$ and $\operatorname{curl}(\operatorname{BS}_{\Omega}(\Gamma))=0$ outside of $\Omega$. We further observe that since $\operatorname{curl}(\operatorname{BS}_{\Omega}(\Gamma))=\Gamma\parallel \Sigma$, the tangent part of the restriction of $\operatorname{BS}_{\Omega}(\Gamma)$ to $\Sigma$ is curl-free. In particular, according to the Hodge decomposition theorem it follows that the $L^2(\Sigma)$-orthogonal projection of $\left(\operatorname{BS}_{\Omega}(\Gamma)\right)^\parallel$ onto $\mathcal{H}(\Sigma)$ differs from the original tangential field $\left(\operatorname{BS}_{\Omega}(\Gamma)\right)^\parallel$ only by a gradient field. Consequently, if we let $\sigma_t$ denote the toroidal curve bounding a surface $\mathcal{A}$ outside of $\Omega$, we obtain, by virtue of Stokes' theorem,
	\begin{AppA}
		\label{EquationC2}
		\int_{\sigma_t}\pi_{\mathcal{H}(\Sigma)}\left(\left(\operatorname{BS}_{\Omega}(\Gamma)\right)^\parallel\right)=\int_{\sigma_t}\operatorname{BS}_{\Omega}(\Gamma)=\int_{\mathcal{A}}\operatorname{curl}(\operatorname{BS}_{\Omega}(\Gamma))\cdot n d\sigma=0,
	\end{AppA}
	where $\pi_{\mathcal{H}(\Sigma)}$ denote the projection onto $\mathcal{H}(\Sigma)$, where we used that $\sigma_t$ is tangent to $\Sigma$, where $n$ denotes the corresponding unit normal and where we used that $\operatorname{curl}(\operatorname{BS}_{\Omega}(\Gamma))=0$ on $\mathbb{R}^3\setminus \overline{\Omega}$.
	
	Due to the scaling properties of helicity we may now without loss of generality assume that $\gamma$ (and therefore $\tilde{\gamma}$) are $L^2(\Sigma)$-normalised and hence form an $L^2(\Sigma)$-orthonormal basis of $\mathcal{H}(\Sigma)$. We can then express
	\begin{AppA}
		\label{EquationC3}
		\pi_{\mathcal{H}(\Sigma)}\left(\left(\operatorname{BS}_{\Omega}(\Gamma)\right)^\parallel\right)=\int_{\Sigma}\gamma\cdot \operatorname{BS}_{\Omega}(\Gamma)d\sigma\gamma+\int_{\Sigma}\tilde{\gamma}\cdot \operatorname{BS}_{\Omega}(\gamma)d\sigma\tilde{\gamma}
	\end{AppA}
	where we used that $\gamma$ and $\tilde{\gamma}$ are tangent to $\Sigma$. Combining (\ref{EquationC1}),(\ref{EquationC2}) and (\ref{EquationC3}) (keeping in mind that we assume $\gamma$ to be normalised) we arrive at
	\begin{AppA}
		\label{EquationC4}
		0=\int_{\sigma_t}\pi_{\mathcal{H}(\Sigma)}\left(\left(\operatorname{BS}_{\Omega}(\Gamma)\right)^\parallel\right)=\|\Gamma\|^2_{L^2(\Omega)}\mathcal{H}(\gamma)\int_{\sigma_t}\gamma+\int_{\Sigma}\tilde{\gamma}\cdot \operatorname{BS}_{\Omega}(\Gamma)d\sigma\int_{\sigma_t}\tilde{\gamma}.
	\end{AppA}
	Finally we recall that $\left(\operatorname{BS}_{\Omega}(\Gamma)\right)^\parallel$ is curl-free in the weak sense on $\Sigma$ and that so is $\Gamma|_{\Sigma}$ because $\operatorname{curl}(\operatorname{BS}_{\Omega}(\Gamma))=\Gamma$ and $\operatorname{curl}(\Gamma)=0$ in $\Omega$ are both tangent to $\Sigma$. Accordingly $\Gamma\times \mathcal{N}$ and $\gamma\times\mathcal{N}=\tilde{\gamma}$ differ only by a co-exact vector field and keeping in mind that $\left(\operatorname{BS}_{\Omega}(\Gamma)\right)^\parallel$ is $L^2(\Sigma)$-orthogonal to the co-exact fields (as it is weakly curl-free on $\Sigma$) we may calculate
	\begin{AppA}
		\nonumber
		\int_{\Sigma}\tilde{\gamma}\cdot\operatorname{BS}_{\Omega}(\Gamma)d\sigma=\int_{\Sigma}\tilde{\gamma}\cdot \left(\operatorname{BS}_{\Omega}(\Gamma)\right)^\parallel d\sigma=\int_{\Sigma}(\Gamma\times \mathcal{N})\cdot \left(\operatorname{BS}_{\Omega}(\Gamma)\right)^\parallel d\sigma=\int_{\Sigma}(\Gamma\times \mathcal{N})\cdot \operatorname{BS}_{\Omega}(\Gamma)d\sigma
		\end{AppA}
		\begin{AppA}
		\label{EquationCExtra5}
		=\int_{\Sigma}\left(\operatorname{BS}_{\Omega}(\Gamma)\times \Gamma\right)\cdot \mathcal{N}d\sigma=\int_{\Omega}\operatorname{div}\left(\operatorname{BS}_{\Omega}(\Gamma)\times \Gamma\right)d^3x=\|\Gamma\|^2_{L^2(\Omega)},
	\end{AppA}
	where we used the cyclic properties of the Euclidean product, that $\tilde{\gamma}$ it tangent to $\Sigma$ and the standard calculus formula $\operatorname{div}(X\times Y)=\operatorname{curl}(X)\cdot Y-X\cdot \operatorname{curl}(Y)$ for arbitrary $C^1$-vector fields $X,Y$ on $\Omega$. We can insert this identity in (\ref{EquationC4}) and finally arrive at
	\begin{gather}
		\nonumber
		0=\|\Gamma\|^2_{L^2(\Omega)}\left(\mathcal{H}(\gamma)\int_{\sigma_t}\gamma+\int_{\sigma_t}\tilde{\gamma}\right)
	\end{gather}
	which immediately implies the claimed identity.
\end{proof}
The following final result allows us to express the helicity of a vector field through $\Gamma$ alone.
\begin{prop}
	\label{PropProp}
	Let $T^2\cong \Sigma\subset\mathbb{R}^3$ bound a $C^{1,1}$-solid torus $\Omega\subset\mathbb{R}^3$. Suppose that $\sigma_p$ is a purely poloidal curve which bounds a $C^{1,1}$-disc $D\subset\Omega$ and suppose that $\sigma_t$ is a toroidal curve bounding a $C^{1,1}$-surface $\mathcal{A}\subset \overline{\Omega}^c$. Let further $\Gamma\in \mathcal{H}_N(\Omega)\setminus \{0\}$ and let $\gamma$ denote the projection of $\Gamma|_{\Sigma}$ onto $\mathcal{H}(\Sigma)$ and set $\tilde{\gamma}:=\gamma\times \mathcal{N}$. Then
	\begin{AppA}
		\label{EquationC5}
		\int_{\sigma_t}\gamma=\int_{\sigma_t}\Gamma\text{, }\int_{\sigma_p}\tilde{\gamma}=\frac{\operatorname{Flux}(\Gamma)}{\|\Gamma\|^2_{L^2(\Omega)}}\|\gamma\|^2_{L^2(\Sigma)}
	\end{AppA}
	where $\operatorname{Flux}(\Gamma):=\int_D\Gamma\cdot n d\sigma$ and $n$ is the unit normal on the disc $D$ compatible with a chosen orientation of $\sigma_p$ with regards to Stokes' theorem. In particular, if $v\in C^{0,1}\mathcal{V}_0(\Sigma)$, then
	\begin{AppA}
		\label{EquationC6}
		\frac{\mathcal{H}(v)}{|\Sigma|^2}=\overline{Q}(v)\overline{P}(v)\frac{\left(\int_{\sigma_t}\Gamma\right)\operatorname{Flux}(\Gamma)}{\|\Gamma\|^2_{L^2(\Omega)}}=\frac{\int_{\Sigma}v\cdot \Gamma d\sigma\int_{\Sigma}v\cdot \operatorname{BS}_{\Omega}(\Gamma)d\sigma}{\|\Gamma\|^2_{L^2(\Omega)}|\Sigma|^2}.
	\end{AppA}
\end{prop}
\begin{proof}[Proof of \Cref{PropProp}]
	The first identity in (\ref{EquationC5}) follows from the fact that $\Gamma|_{\Sigma}$ is curl-free in the weak sense and therefore $\Gamma|_{\Sigma}$ and $\gamma$ differ only by a gradient field, so that the fact that $\sigma_t$ is a closed curve yields the result. We assume now without loss of generality that $\gamma$ is $L^2(\Sigma)$-normalised. Then for the second identity in (\ref{EquationC5}) we may utilise (\ref{EquationC3}),(\ref{EquationCExtra5}) and the fact that $\int_{\sigma_p}\gamma=0$ to obtain
	\begin{gather}
		\nonumber
		\operatorname{Flux}(\Gamma)=\int_D\Gamma\cdot nd\sigma=\int_{D}\operatorname{curl}(\operatorname{BS}_{\Omega}(\Gamma))\cdot nd\sigma=\int_{\sigma_p}\operatorname{BS}_{\Omega}(\Gamma)=\int_{\sigma_p}\pi_{\mathcal{H}(\Sigma)}\left(\left(\operatorname{BS}_{\Omega}(\Gamma)\right)^\parallel\right)=\|\Gamma\|^2_{L^2(\Omega)}\int_{\sigma_p}\tilde{\gamma}
	\end{gather}
	where we used that $D\subset \Omega$, $\partial D=\sigma_p$, that $\sigma_p$ is tangent to $\Sigma$ and that $\left(\operatorname{BS}_{\Omega}(\Gamma)\right)^\parallel$ is curl-free in the weak sense.
	
	The first equality in (\ref{EquationC6}) is a direct consequence of (\ref{EquationC5}) and \Cref{T213}. As for the second identity we recall that we have argued on several occasions that $\left(\int_{\sigma_t}\gamma\right) \gamma_t=\gamma$ and we have seen that $\int_{\sigma_p}\pi_{\mathcal{H}(\Sigma)}\left(\left(\operatorname{BS}_{\Omega}(\Gamma)\right)^\parallel\right)=\operatorname{Flux}(\Gamma)$ and according to (\ref{EquationC2}) $\int_{\sigma_t}\pi_{\mathcal{H}(\Sigma)}\left(\left(\operatorname{BS}_{\Omega}(\Gamma)\right)^\parallel\right)=0$ so that by the defining properties of $\gamma_p$ we find $\operatorname{Flux}(\Gamma)\gamma_p=\pi_{\mathcal{H}(\Sigma)}\left(\left(\operatorname{BS}_{\Omega}(\Gamma)\right)^\parallel\right)$. We can then utilise \Cref{L212} and use the fact that $v$ is tangent to $\Sigma$ and div-free and that $\Gamma$ and $\gamma$ as well as $\left(\operatorname{BS}_{\Omega}(\Gamma)\right)^\parallel$ and $\pi_{\mathcal{H}(\Sigma)}\left(\left(\operatorname{BS}_{\Omega}(\Gamma)\right)^\parallel\right)$ only differ by a gradient field respectively which yields the last equality.
\end{proof}
\bibliographystyle{plain}
\bibliography{mybibfileNOHYPERLINK}
\footnotesize
\end{document}